\documentclass[aos]{imsart}

\usepackage{amsthm,amsmath}
\usepackage{amssymb}
\usepackage{mathrsfs, color}
\usepackage{subfigure}
\usepackage{graphicx}
\usepackage{graphics}
\usepackage{multirow}
\usepackage{hyperref}
\usepackage{url}
\usepackage{enumitem}
\usepackage{ulem}
\usepackage{xr}

\startlocaldefs
\newcommand{\ub}[0]{\underline{b}}
\newcommand{\va}[0]{\mathbf a}

\newcommand{\vh}[0]{\mathbf h}

\newcommand{\vX}[0]{\mathbf X}

\newcommand{\vone}[0]{\mathbf 1}
\newcommand{\mbf}[1]{\mbox{\boldmath$#1$}}
\newcommand{\Cov}[0]{\text{Cov}}
\newcommand{\Var}[0]{\text{Var}}

\newcommand{\tr}[0]{\text{tr}}

\newcommand{\E}[0]{\mathbb{E}}
\newcommand{\R}[0]{\mathbb{R}}
\newcommand{\Id}[0]{\text{Id}}

\newcommand{\mn}[0]{\text{multinomial}}
\newcommand{\all}[0]{\text{all }}

\newcommand{\Prob}[0]{\mathbb{P}}

\newcommand{\vertiii}[1]{{\left\vert\kern-0.25ex\left\vert\kern-0.25ex\left\vert #1 
    \right\vert\kern-0.25ex\right\vert\kern-0.25ex\right\vert}}

\theoremstyle{plain}% default
\newtheorem{thm}{Theorem}[section]
\newtheorem{lem}[thm]{Lemma}
\newtheorem{prop}[thm]{Proposition}
\newtheorem{cor}[thm]{Corollary}

\theoremstyle{definition}
\newtheorem{defn}{Definition}[section]

\newtheorem{exmp}{Example}[section]

\theoremstyle{remark}
\newtheorem{rmk}{Remark}

\endlocaldefs

\begin{document}

\begin{frontmatter}

% "Title of the paper"
\title{Gaussian and bootstrap approximations for high-dimensional U-statistics and their applications\thanksref{T1}}
\runtitle{Gaussian approximation for high-dimensional U-statistics}
\thankstext{T1}{First version: January 30, 2016. This version: \today.}

% indicate corresponding author with \corref{}
% \author{\fnms{John} \snm{Smith}\corref{}\ead[label=e1]{smith@foo.com}\thanksref{t1}}
% \thankstext{t1}{Thanks to somebody} 
% \address{line 1\\ line 2\\ printead{e1}}
% \affiliation{Some University}

\begin{aug}
\author{\fnms{Xiaohui} \snm{Chen}\thanksref{m1}\ead[label=e1]{xhchen@illinois.edu}\thanksref{t2}}
\runauthor{Chen}
\thankstext{t2}{Research partially supported by NSF grant DMS-1404891 and UIUC Research Board Award RB15004.} 

\affiliation{University of Illinois at Urbana-Champaign\thanksmark{m1}}

\address{Xiaohui Chen \\
Department of Statistics \\
University of Illinois at Urbana-Champaign \\
725 S. Wright Street \\
Champaign, IL 61820 \\
\printead{e1}}
\end{aug}
%
%\runauthor{Chen et al.}

%\begin{keyword}[class=AMS]
%\kwd[Primary ]{}
%\kwd{}
%\kwd[; secondary ]{}
%\end{keyword}

\begin{abstract}
This paper studies the Gaussian and bootstrap approximations for the probabilities of a non-degenerate U-statistic belonging to the hyperrectangles in $\R^d$ when the dimension $d$ is large. A two-step Gaussian approximation procedure that does not impose structural assumptions on the data distribution is proposed. Subject to mild moment conditions on the kernel, we establish the explicit rate of convergence uniformly in the class of all hyperrectangles in $\R^d$ that decays polynomially in sample size for a high-dimensional scaling limit, where the dimension can be much larger than the sample size. We also provide computable approximation methods for the quantiles of the maxima of centered U-statistics. Specifically, we provide a unified perspective for the empirical bootstrap, the randomly reweighted bootstrap, and the Gaussian multiplier bootstrap with the jackknife estimator of covariance matrix as randomly reweighted quadratic forms and we establish their validity. We show that all three methods are inferentially first-order equivalent for high-dimensional U-statistics in the sense that they achieve the same uniform rate of convergence over all $d$-dimensional hyperrectangles. In particular, they are asymptotically valid when the dimension $d$ can be as large as $O(e^{n^c})$ for some constant $c \in (0,1/7)$.

The bootstrap methods are applied to statistical applications for high-dimensional non-Gaussian data including: (i) principled and data-dependent tuning parameter selection for regularized estimation of the covariance matrix and its related functionals; (ii) simultaneous inference for the covariance and rank correlation matrices. In particular, for the thresholded covariance matrix estimator with the bootstrap selected tuning parameter, we show that for a class of subgaussian data, error bounds of the bootstrapped thresholded covariance matrix estimator %adapt the the dependency and moment of the underlying data distribution and therefore the bounds 
can be much tighter than those of the minimax estimator with a universal threshold. In addition, we also show that the Gaussian-like convergence rates can be achieved for heavy-tailed data, which are less conservative than those obtained by the Bonferroni technique that ignores the dependency in the underlying data distribution. 
%The Gaussian approximation result and bootstrap limit theorems rely on a general Gaussian approximation result and the tail probability inequalities for the maxima of non-degenerate U-statistics with the unbounded kernel. Results established in this paper are nonlinear generalizations of the Gaussian and bootstrap approximations of the maxima of the high-dimensional sample mean vectors to the U-statistics of order two.
\end{abstract}

\end{frontmatter}

\section{Introduction}

Let $X_1^n = \{X_1,\cdots,X_n\}$ be a sample of independent and identically distributed (iid) random vectors in $\R^p$ with the distribution $F$. Let $h : \R^p \times \R^p \to \R^d$ be a fixed and measurable function such that $h(x_1,x_2)=h(x_2,x_1)$ for all $x_1,x_2 \in \R^p$ and $\E |h_k(X_1,X_2)|<\infty$ for all $k=1,\cdots,d$. Consider the U-statistic of order two
\begin{equation}
\label{eqn:ustat-order2}
U_n = {1 \over n (n-1)} \sum_{1 \le i \neq j \le n} h(X_i, X_j).
\end{equation}
In this paper, we consider the uniform approximation of the probabilities of $U_n$ over a class of the Borel subsets in $\R^d$. More specifically, let $T_n=\sqrt{n} (U_n - \theta) / 2$, where $\theta =  \E [h(X_1,X_2)]$ is the parameter of interest, and ${\cal A}^{re}$ be the class of all hyperrectangles $A$ in $\R^d$ of the form
\begin{equation}
\label{eqn:hyperrectangles}
A = \{x \in \R^d : a_j \le x_j \le b_j \text{ for all } j = 1, \cdots, d\},
\end{equation}
where $-\infty \le a_j \le b_j \le \infty$ for $j=1,\cdots,d$. Our main goal are to construct a random vector $T_n^\natural$ in $\R^d$ and to derive non-asymptotic bounds for
\begin{equation}
\label{eqn:GA-general-bound}
\rho^{re}(T_n, T_n^\natural) = \sup_{A \in {\cal A}^{re}} |\Prob(T_n \in A) - \Prob(T_n^\natural \in A)|.
\end{equation}
When $p$ (and therefore $d$) is fixed, the classical central limit theorems (CLT) for approximating $T_n$ by a Gaussian random vector $T_n^\natural \sim N(0,\Gamma)$, where $\Gamma = \Cov(g(X_1))$ and $g(X_1) = \E[h(X_1,X_2) | X_1] - \theta$, have been extensively studied in literature \cite{hoeffding1948,arconesgine1993,hsingwu2004,gotze1987,delaPenaGine1999,hoeffding1963,gregory1977,serfling1980,arconesgine1993,zhang1999,ginelatalazinn2000,houdrepatricia2003,hsingwu2004,serfling1980}. Recently, due to the explosive data enrichment, regularized estimation and dimension reduction of high-dimensional data (i.e. $d$ is larger or even much larger than $n$) have attracted a lot of research attentions such as covariance matrix estimation \cite{bickellevina2008a,bickellevina2008b,karoui2008,chenxuwu2013a}, graphical models \cite{dempster1972,yuanlin2007,buhlmannvandegeer2011}, discriminant analysis \cite{maizouyuan2012a}, factor models \cite{fanliaomincheva2011,MR2933663}, among many others. Those problems all involve the consistent estimation of an expectation $\E[h(X_1,X_2)]$ of U-statistics of order two. Below are three examples.

\begin{exmp}
\label{exmp:sample-mean}
The sample mean vector $\bar{X}_n = n^{-1} \sum_{i=1}^n X_i$ is an unbiased estimator of $\E X_1$ and $\bar{X}_n$ can be written as a U-statistic of form (\ref{eqn:ustat-order2}) with the linear kernel $h(x_1,x_2) = (x_1 + x_2) / 2$ for $x_1,x_2 \in \R^p$ and $d = p$.
\end{exmp}

\begin{exmp}
\label{exmp:sample-covmat}
Let $d = p \times p$. The sample covariance matrix $\hat{S}_n = (n-1)^{-1} \sum_{i=1}^n (X_i-\bar{X}_n) (X_i-\bar{X}_n)^\top$ is an unbiased estimator of the covariance matrix $\Sigma = \Cov(X_1)$. Here, $\hat{S}_n$ is a matrix-valued U-statistic of form (\ref{eqn:ustat-order2}) with the quadratic kernel $h(x_1,x_2) = (x_1-x_2) (x_1-x_2)^\top / 2$ for $x_1,x_2 \in \mathbb{R}^p$.
\end{exmp}

\begin{exmp}
\label{exmp:kendalltau}
The covariance matrix quantifies the linear dependency in a random vector. The rank correlation is another measure for the nonlinear dependency in a random vector. Two generic vectors $y = (y_1, y_2)$ and $z = (z_1,z_2)$ in $\R^2$ are said to be {\it concordant} if $(y_1-z_1) (y_2-z_2) > 0$. For $m,k=1,\cdots,p$, define
$$
\tau_{mk} = {1 \over n (n-1)} \sum_{1 \le i \neq j \le n} \vone\{ (X_{im}-X_{jm}) (X_{ik}-X_{jk}) > 0\}.
$$
Then, Kendall's tau rank correlation coefficient matrix $T = \{\tau_{mk}\}_{m,k=1}^p$
is a matrix-valued U-statistic with a bounded kernel. It is clear that $\tau_{mk}$ quantifies the monotonic dependency between $(X_{1m}, X_{1k})$ and $(X_{2m}, X_{2k})$ and it is an unbiased estimator of $\Prob((X_{1m}-X_{2m})(X_{1k}-X_{2k})>0)$, i.e. the probability that $(X_{1m}, X_{1k})$ and $(X_{2m}, X_{2k})$ are concordant.
\end{exmp}

In this paper, we are interested in the following central questions: {\it how does the dimension impact the asymptotic behavior of U-statistics and how can we make practical statistical inference when $d \to \infty$?} Bounds on (\ref{eqn:GA-general-bound}) with the explicit dependence on $d$ are particularly useful in large-scale statistical inference problems. In particular, motivation of this paper comes from the estimation and inference problems for large covariance matrix and its related functionals \cite{meinshausenbuhlmann2006,yuanlin2007,rothmanbickellevinazhu2008a,pengwangzhouzhu2009a,yuan2010a,cailiuluo2011a,bickellevina2008b,chenxuwu2013a,chenxuwu2016+}. To establish rate of convergence for the regularized estimators or to approximate the limiting null distribution of $\ell^\infty$-tests in high-dimensions, a key issue is to characterize the distribution of the supremum norm $|U_n-\E U_n|_\infty$ that relates to the probabilities of $\Prob(T_n \in A)$ for $A$ belonging to the family of max-hyperrectangles in $\R^d$ of the form $A = \{x \in \R^d : x_j \le a \text{ for all } j =1,\cdots,d\}$ and $-\infty \le a \le \infty$.
 %Therefore, as the primary concern of the current paper, we shall consider $B=\mathbb{R}^{p\times p}$ and $\|h\| = \max_{1\le m,k \le p} |h_{mk}|$. 

Our first main contribution is to provide a Gaussian approximation scheme for the high-dimensional {\it non-degenerate} U-statistics. Different from the CLT type results for the sums of independent random vectors \cite{cck2013,cck2015a}, which are directly approximated by the Gaussian counterparts with the matching first and second moments, approximation of the U-statistics is more subtle because of its dependency and nonlinearity structures. Here, we propose a {\it two-step} Gaussian approximation method in Section \ref{sec:gaussian-approx}. In the first step, we approximate the U-statistics by the leading component of a linear form in the Hoeffding decomposition (a.k.a. the H\'ajek projection); in the second step, the linear term is further approximated by the Gaussian random vectors. To approximate the distribution of U-statistics by a linear form, a maximal moment inequality is developed to control the nonlinear and {\it canonical}, i.e. {\it completely degenerate}, form of the reminder term. Then the linear projection is handled by the recent development of Gaussian approximation in high-dimensions \cite{cck2013,cck2015a,zhangcheng2014,zhangwu2016a}. Explicit rate of convergence of the Gaussian approximation for high-dimensional U-statistics uniformly in the class of all hyperrectangles in $\R^d$ is established for unbounded kernels subject to sub-exponential and uniform polynomial moment conditions. Specifically, under either moment conditions, we show that the validity of the Gaussian approximation holds for a high-dimensional scaling limit, where $d$ can be larger or even much larger than $n$. In our results, symmetry of U-statistics is an key ingredient in the Hoeffding decomposition. Therefore our result can be viewed as nonlinear generalizations of the Gaussian approximation for the high-dimensional sample mean vector of iid $X_1,\cdots,X_n$.

The second contribution of this paper is to provide {\it computable} methods for approximating the probabilities $\Prob(T_n \in A)$ uniformly for $A \in {\cal A}^{re}$. This allows us to compute the quantiles of the maxima $|U_n-\E U_n|_\infty$. Since the covariance matrix $\Gamma$ of the H\'ajek projection of the centered U-statistics depends on the underlying data distribution $F$ which is unknown in many real applications, a practically feasible alternative is to use data-dependent approaches such as the bootstrap to approximate $\Prob(T_n \in A)$, where the insight is to implicitly construct a consistent estimator of $\Gamma$ under the supremum norm. In Section \ref{sec:bootstraps}, we provide a unified perspective for the empirical bootstrap (EB), the randomly reweighted bootstrap, and the Gaussian multiplier bootstrap with the jackknife estimator of covariance matrix as randomly reweighted quadratic forms and we establish their validity. Specifically, we show that all three methods are inferentially first-order equivalent for high-dimensional U-statistics in the sense that they achieve the same uniform rate of convergence over ${\cal A}^{re}$. In particular, they are asymptotically valid when the dimension $d$ can be as large as $O(e^{n^c})$ for some constant $c \in (0,1/7)$. One important feature of the Gaussian and bootstrap approximations is that no structural assumption on the distribution $F$ is made and the strong dependency in $F$ is allowed, which in fact helps the Gaussian and bootstrap approximations.

In Section \ref{sec:stat_apps}, we apply the proposed bootstrap method to a number of important high-dimensional problems, including the data-dependent tuning parameter selection in the thresholded covariance matrix estimator and the simultaneous inference of the covariance and Kendall's tau rank correlation matrices. Two additional applications for the estimation problems of the sparse precision matrix and the sparse linear functionals of the precision matrix are given in the Supplementary Material (SM). In those problems, we show that the Gaussian like convergence rates can be achieved for non-Gaussian data with heavy-tails, which are less conservative than those obtained by the Bonferroni technique that ignores the dependency in the underlying data distribution. For the sparse covariance matrix estimation problem, we also show that the thresholded estimator with the tuning parameter selected by the bootstrap procedure adapts the the dependency and moment in the underlying data distribution and therefore the bounds can be much tighter than those of the minimax estimator with a universal threshold that ignores the dependency in $F$ \cite{bickellevina2008b,chenxuwu2013a,caizhou2011a}.

To establish the Gaussian approximation result and the validity of the bootstrap methods, a key step is to bound the the expected supremum norm of the second-order canonical term in the Hoeffding decomposition of the U-statistics and establish its non-asymptotic maximal moment inequalities. An alternative simple data splitting approach by reducing the U-statistics to sums of iid random vectors can give the exact rate for bounding the moments in the non-degenerate case \cite{talagrand1996,massart2000,kleinrio2005,einmahlli2008}. Nonetheless, the reduction to the iid summands in terms of data splitting does not exploit the complete degeneracy structure of the canonical term and it does not lead to the convergence result in the Gaussian approximation for the non-degenerate U-statistics; see Section \ref{app:concentration_ineq_canonical-Ustat} for details. In addition, unlike the Hoeffding decomposition approach, the data splitting approximation is not asymptotically tight in distribution and therefore it is less useful in making inference of the high-dimensional U-statistics.

{\it Relation to the existing literature.} For univariate U-statistics, the empirical bootstrap was studied in \cite{bickelfreedman1981,arconegine1992b} and the randomly reweighted bootstrap of the form (\ref{eqn:randweight-bootstrap-ustat}) was proposed in \cite{janssen1994,huskovajanssen1993b}, where a different class of random weights $w_i$ was considered satisfying $w_i = \xi_i / (n^{-1} \sum_{i=1}^n \xi_i)$ such that $\xi_i$ are iid non-negative random variables and $\E\xi_i^2 < \infty$. Weights of such form contain the Bayesian bootstrap as a special case \cite{rubin1981,lo1987}. The randomly reweighted bootstrap with iid mean-zero weights was considered for the non-degenerate case in \cite{wangjing2004} and for the degenerate case in \cite{dehlingmikosch1994}. More general exchangeably weighted bootstraps can be found in \cite{masonnewton1992,huskovajanssen1993,praestgaardwellner1993}. However, none of those results in literature can be used to establish the bootstrap validity for high-dimensional U-statistics when $d \gg n$. The Gaussian and bootstrap approximations for the maxima of sums of high-dimensional independent random vectors were considered in \cite{cck2015a,cck2013}. For an iid sample, this corresponds to a U-statistic with the kernel $h(x_1,x_2) = (x_1 + x_2) / 2$ for $x_1,x_2 \in \R^d$. Thus, our results are nonlinear generalizations of those in \cite{cck2015a,cck2013} when $X_1,\cdots,X_n$ are iid.

The current paper supersedes and improves the preliminary work \cite{chen2016a} (available as an arXiv preprint) by the author. In \cite{chen2016a}, a Gaussian multiplier bootstrap was proposed by estimating the individual H\'ajek projection terms using the idea of decoupling on an independent dataset. The bootstrap validity therein is established under the Kolmogorov distance, which is a subset of ${\cal A}^{re}$ corresponding to max-hyperrectangles in $\R^d$. In addition, the rate of convergence in \cite{chen2016a} is sub-optimal while the rate derived in this paper is nearly optimal; see Remark \ref{rmk:comparison_with_naive_gaussian_wild_bootstrap} for detailed comparisons.

{\it Notations and definitions.} For a vector $x$, we use $|x|_1 = \sum_j |x_j|$, $|x|:=|x|_2 = (\sum_j x_j^2)^{1/2}$, and $|x|_\infty = \max_j |x_j|$ to denote its entry-wise $\ell^1$, $\ell^2$, and $\ell^\infty$ norms, respectively. For a matrix $M$, we use $|M|_F=(\sum_{i,j} M_{ij}^2)^{1/2}$ and $\|M\|_2 =\max_{|a|=1} |M a|$ to denote its Frobenius and spectral norms, respectively. We shall use $C, C_1, C_2,\cdots$ to denote positive constants that do not depend on $n$ and $d$ and whose values may change from place to place. Denote $a \vee b = \max(a,b)$, $a \wedge b = \min(a,b)$, $a \asymp b$ if $C_1 a \le b \le C_2 b$ for some constants $C_1, C_2 > 0$. For a random variable $X$, we write $\|X\|_q = (\E|X|^q)^{1/q}$ for $q>0$.  For $r=1,\cdots,n$, we shall write $x_1^r = (x_1,\cdots,x_r)$ and $\E h = \E h(X_1^r)$ for the random variables $X_1,\cdots,X_r$ taking values in a measurable space $(S, {\cal S})$ and a measurable function $h : S^r \to \R^d$. For two vectors $x,y \in \mathbb{R}^d$, we use $x \le y$ (or $x > y$) to mean that $x_j \le y_j$ (or $x_j > y_j$) for all $j = 1,\cdots,d$. We use ${\cal L}(X)$ to denote the law or distribution of the random variable $X$. For $\alpha > 0$, let $\psi_\alpha(x) = \exp(x^\alpha)-1$ be a function defined on $[0,\infty)$ and $L_{\psi_\alpha}$ be the collection of all real-valued random variables $\xi$ such that $\E[\psi_{\alpha}(|\xi|/C)] < \infty$ for some $C > 0$. For $\xi \in L_{\psi_\alpha}$, we define $\|\xi\|_{\psi_\alpha} = \inf\{C > 0 : \E[\psi_{\alpha}(|\xi|/C)] \le 1\}$. Then, for $\alpha \in [1,\infty)$, $\|\cdot\|_{\psi_\alpha}$ is an Orlicz norm and $(L_{\psi_\alpha}, \|\cdot\|_{\psi_\alpha})$ is a Banach space \cite{ledouxtalagrand1991}. For $\alpha \in (0,1)$, $\|\cdot\|_{\psi_\alpha}$ is a quasi-norm, i.e. there exists a constant $C(\alpha) > 0$ such that $\|\xi_1+\xi_2\|_{\psi_\alpha} \le C(\alpha) (\|\xi_1\|_{\psi_\alpha} + \|\xi_2\|_{\psi_\alpha})$ holds for all $\xi_1,\xi_2 \in L_{\psi_\alpha}$ \cite{adamczak2008}. We denote the Kolmogorov distance between two real-valued random variables $X$ and $Y$ as $\rho(X, Y) = \sup_{t \in \R} |\Prob(X \le t) - \Prob(Y \le t)|$. Throughout the paper, we assume that $n \ge 4$ and $d \ge 3$.

\section{Gaussian approximation}
\label{sec:gaussian-approx}

In this section, we study the approximation for $\Prob(T_n \in A)$ where $T_n=\sqrt{n} (U_n - \theta) / 2$ and $A \in {\cal A}^{re}$. We shall derive a Gaussian approximation result (GAR) for non-degenerate U-statistics, which is the stepping stone to study various bootstrap procedures in Section \ref{sec:bootstraps}. Let $X'$ and $X$ be two independent random vectors with the distribution $F$ that are also independent of $X_1^n$. In Section \ref{sec:gaussian-approx} and \ref{sec:bootstraps}, since we consider centered U-statistics $T_n$, we assume without loss of generality that $\theta = 0$. Define $g(X) = \E[h(X,X') | X]$ and $f(X,X') = h(X,X') - g(X) - g(X')$.

\begin{defn}
\label{def:degeneracy}
The kernel $h : \R^p \times \R^p \to \R^d$ is said to be: (i) non-degenerate if $\Var(g_m(X)) > 0$ for all $m=1,\cdots,d$; (ii) degenerate of order one, i.e. completely degenerate or $F$-canonical, if $\Prob(g(X)=0)=1$ or equivalently $\E[h(x_1,X')] = \E[h(X,x_2)] = \E[h(X,X')] = 0$ for all $x_1,x_2 \in \mathbb{R}^p$. The corresponding U-statistic in (\ref{eqn:ustat-order2}) is non-degenerate if $h$ is non-degenerate.
\end{defn}

Throughout this paper, we only consider the non-degenerate U-statistics and we assume that
\begin{enumerate}
\item[(M.1)] There exists a constant $\ub > 0$ such that $\E [g_m^2(X)] \ge \ub$ for all $m = 1,\cdots,d$.
\end{enumerate}
The Hoeffding decomposition of $T_n$ is given by $T_n = L_n + R_n$, where
\begin{eqnarray*}
L_n = {1 \over \sqrt{n}} \sum_{i=1}^n g(X_i) \quad \text{and} \quad R_n ={1 \over 2 \sqrt{n} (n-1)} \sum_{1 \le i \neq j \le n} f(X_i, X_j).
\end{eqnarray*}
Since $f$ is $F$-canonical, we expect that $L_n$ is the leading term (a.k.a. the H\'ajek projection) of $T_n$. Therefore, we can reasonably expect that $T_n$ is an approximately linear statistic such that ${\cal L}(T_n) \approx {\cal L}(L_n)$, where the latter can be further approximated by its Gaussian analogue \cite{cck2013,cck2015a}. This motivates the following two-step Gaussian approximation proceudre. Let $\Gamma = \Cov(g(X)) = \E(g(X) g(X)^\top)$ be the $d \times d$ covariance matrix of $g(X)$ and $Y \sim N(0,\Gamma)$ be a $d$-dimensional Gaussian random vector. The main result of this section is to establish non-asymptotic error bounds for $\rho^{re}(T_n, Y)$ under different moment conditions on $h$. Let $q>0$ and $B_n \ge 1$ be a sequence of real numbers possibly tending to infinity. In particular, we shall consider the following assumptions.
\begin{enumerate}
\item[(M.2)] $\E[|h_m(X,X')|^{2+\ell}] \le B_n^\ell$ for $\ell=1,2$ and for all $m=1,\cdots,d$.
\item[(E.1)] $\|h_m(X,X')\|_{\psi_1} \le B_n$ for all $m=1,\cdots,d$.
\item[(E.2)] $\E[\max_{1 \le m \le d} (|h_m(X,X')|/B_n)^q] \le 1$.
\end{enumerate}

In the high-dimensional context, the dimension $d$ grows with the sample size $n$ and the distribution function $F$ may also depend on $n$. Therefore, $B_n$ is allowed to increase with $n$. In particular, under (M.1) and (M.2), $B_n$ can be interpreted as a uniform bound on the standardized absolute moments of $g_m(X)$ for $m =1,\cdots,d$. For instance, the kurtosis parameter $\kappa_m$ of $g_m(X)$ obeys $\kappa_m = [\E g_m^4(X)] / [\E g_m^2(X)]^2 - 3 \le B_n^2 / \ub^2 - 3$. Define
\begin{equation}
\label{eqn:GAR_rates}
\varpi_{1,n} = \left({B_n^2 \log^7 (nd) \over n}\right)^{1/6} \quad \text{and} \quad \varpi_{2,n} = \left({B_n^2 \log^3 (nd) \over n^{1-2/q}}\right)^{1/3}.
\end{equation}

\begin{thm}[{\bf Main result I}: Gaussian approximation for high-dimensional U-statistics for hyperrectangles]
\label{thm:CLT-hyperrec-momconds}
Assume that (M.1) and (M.2) hold. Suppose that $\log d \le \bar{b} n$ for some constant $\bar{b} > 0$. \\
(i) If (E.1) holds, then there exists a constant $C := C(\ub,\bar{b})>0$ such that
\begin{equation}
\label{eqn:CLT-hyperrec-subexp}
\rho^{re}(T_n, Y) \le C \varpi_{1,n}.
\end{equation}
(ii) If (E.2) holds with $q \ge 4$, then there exists a constant $C := C(\ub,\bar{b},q)>0$ such that
\begin{equation}
\label{eqn:CLT-hyperrec-poly}
\rho^{re}(T_n, Y) \le C \left\{ \varpi_{1,n} + \varpi_{2,n} \right\}.
\end{equation}
\end{thm}

The following Corollary is an immediate consequence of Theorem \ref{thm:CLT-hyperrec-momconds}.

\begin{cor}
\label{cor:CLT-hyperrec-momconds}
Assume that (M.1) and (M.2) hold. Let $K \in (0,1)$ and $\bar{b} > 0$. (i) If (E.1) holds and $B_n^2 \log^7(d n)  \le \bar{b} n^{1-K}$, then there exists a constant $C := C(\ub,\bar{b})>0$ such that
\begin{equation}
\label{eqn:CLT-hyperrec-explicit_rate}
\rho^{re}(T_n, Y) \le C n^{-K/6}.
\end{equation}
In particular, 
\begin{equation}
\label{eqn:CLT-kolmogorov-explicit_rate}
\rho(\bar{T}_n, \bar{Y}) \le C n^{-K/6},
\end{equation}
where $\bar{T}_n = \max_{1 \le m \le d} T_{nm}$ and $\bar{Y} = \max_{1 \le m \le d} Y_m$. \\
(ii) If (E.2) holds with $q = 4$ and $B_n^4 \log^7(d n)  \le \bar{b} n^{1-K}$, then there exists a constant $C := C(\ub,\bar{b})>0$ such that (\ref{eqn:CLT-hyperrec-explicit_rate}) and (\ref{eqn:CLT-kolmogorov-explicit_rate}) hold.
\end{cor}

Theorem \ref{thm:CLT-hyperrec-momconds} and Corollary \ref{cor:CLT-hyperrec-momconds} are non-asymptotic, showing that the validity of the Gaussian approximation for centered non-degenerate U-statistics holds even if $d$ can be much larger than $n$ and no structural assumption on $F$ is required. In particular, Theorem \ref{thm:CLT-hyperrec-momconds} applies to kernels with the sub-exponential distribution such that $\|h_m\|_q \le C q$ for all $q \ge 1$, in which case $B_n = O(1)$ and the dimension $d$ is allowed to have a subexponential growth rate in the sample size $n$, i.e. $d = O(\exp(n^{(1-K)/7}))$. Condition (E.1) also covers bounded kernels $\|h\|_\infty \le B_n$, where $B_n$ may increase with $n$.

\begin{rmk}[Comments on the near-optimality of the convergence rate in Theorem \ref{thm:CLT-hyperrec-momconds}]
The rate of convergence $n^{-K/6}$ obtained in (\ref{eqn:CLT-hyperrec-explicit_rate}) is slower than the Berry-Esseen rate $n^{-1/2}$ when $d$ is fixed. Similar observations have been made in the existing literature \cite{portnoy1986,bentkus2003} for the normalized sample mean vectors of iid mean-zero random vectors $X_i \in \mathbb{R}^d$, which corresponds to a U-statistic with the linear kernel $h(x_1,x_2) = (x_1 + x_2) / 2$. Assuming $\Cov(X_i) = \Id_d$, \cite{portnoy1986} showed that $\sqrt{n} \bar{X}_n$ has the asymptotic normality if $d = o(\sqrt{n})$ and \cite{bentkus2003} showed that 
$$\sup_{A \in {\cal A}} |\Prob(\sqrt{n} \bar{X}_n \in A) - \Prob(Y \in A)| \le C d^{1/4} \E |X_1|^3 / n^{1/2},$$
where ${\cal A}$ is the class of all convex subsets in $\mathbb{R}^d$, $Y \sim N(0, \Id_d)$, and $C > 0$ is an absolute constant. In either case, the dependence of the CLT rate on the dimension $d$ is polynomial ($d/n^{1/2}$ and $d^{7/4}/n^{1/2}$, resp). On the contrary, our Theorem \ref{thm:CLT-hyperrec-momconds} allows $d$ can be larger than $n$ in order to obtain the CLT type results in much higher dimensions. Since the rate $O(n^{-1/6})$ is minimax optimal in infinite-dimensional Banach spaces for the linear kernel case \cite{bentkus1985,cck2015a}, we argue that the rates derived in Theorem \ref{thm:CLT-hyperrec-momconds} for U-statistics seem un-improvable in $n$ in the following sense. Let $\{X_{ij}\}_{i=1,\cdots,n; j=1,\cdots,d}$ be an array of iid mean-zero random variables with the distribution $F$ such that $\E X_{ij}^2 = 1$ and $\|X_{ij}\|_{\psi_1} \le c$ for all $i=1,\cdots,n$ and $j = 1,\cdots,d$. Consider the linear kernel. Let $Y \sim N(0,\Id_d)$ and $\bar{Y} = \max_{1 \le j \le d} Y_j$. Denote $\Phi(\cdot)$ and $\phi(\cdot)$ as the cdf and pdf of the standard normal distribution, respectively. By the moderate deviation principle for sums of subexponential random variables (c.f. \cite[equation (1.1)]{chenfangshao2013} or \cite[Chapter 8, equation (2.41)]{petrov1975}), there exist constants $C_0,C_1 > 0$ depending only on $c$ such that
$$
{\Prob(T_{nj} > x) \over 1 - \Phi(x)} = 1 + {\eta_1 (1+x^3) \over n^{1/2}}, \quad j = 1,\cdots,d
$$
for $0 \le x \le C_0 n^{1/6}$ and $|\eta_1| \le C_1$. Then, for all such $x$ in the power zone of normal convergence, we have
$$
\Prob(\bar{T}_n \le x) - \Prob(\bar{Y} \le x) = \Prob(\bar{Y} \le x) \{ [1 + \eta_2 (1-\Phi(x)) / \Phi(x) ]^d - 1 \},
$$
where $\eta_2 = -\eta_1 (1+x^3) n^{-1/2}$. Take a distribution $F$ such that $\eta_1 < 0$. By the inequality $(1+x)^d \ge 1+dx$ for $x \ge 0$,
$$
\Prob(\bar{T}_n \le x) - \Prob(\bar{Y} \le x) \ge |\eta_1| (1+x^3) n^{-1/2} \Prob(\bar{Y} \le x) d [1 - \Phi(x)].
$$
Let $x^*$ be the median of $\bar{Y}$; i.e. $\Prob(\bar{Y} \le x^*) = 1/2$. Then $x^* \asymp \sqrt{2 \log{d}}$. In fact, by \cite[Corollary 3.1]{dasguptalahiristoyanov2014}, we have $x^* \le \sqrt{2 \log{d}}$ for $d \ge 31$. Thus, if $x^* \le C_0 n^{1/6}$, then using $[1-\Phi(x)] / [x^{-1} \phi(x)] \to 1$ as $x \to \infty$ we have
$$
\rho(\bar{T}_n, \bar{Y}) \ge C_2 n^{-1/2} {x^*}^2 d \exp(-{x^*}^2/2) \ge C_2 n^{-1/2}{x^*}^2.
$$
Hence, there exist constants $C$ and $C'$ depending only on $F$ such that if $(\log{d})^3 \le C' n$, then $\rho(\bar{T}_n, \bar{Y}) \ge C n^{-1/2} \log{d}$. In particular, taking $(\log{d})^3 \asymp n$, we have $\rho(\bar{T}_n, \bar{Y}) \ge C n^{-1/6}$. Therefore, in view of  the upper bound in (\ref{eqn:CLT-hyperrec-subexp}) and the lower bound in \cite{bentkus1985,cck2015a}, we conjecture that the optimal rate for $\rho(\bar{T}_n, \bar{Y})$ in the high-dimensional setting is $O( (n^{-1} B_n^2 \log^a(nd))^{1/6} )$ for some $a > 0$, based on which the rate of convergence in (\ref{eqn:CLT-hyperrec-subexp}) is also nearly optimal in $d$. However, a rigorous lower bound for $\rho(\bar{T}_n, \bar{Y})$ is still an open question. By the moderate deviations for self-normalized sums \cite{shaowang2013} and the argument above, we expect that similar comments apply for $X_{ij}$ with weaker polynomial moment conditions.
\qed
\end{rmk}

Theorem \ref{thm:CLT-hyperrec-momconds} and Corollary \ref{cor:CLT-hyperrec-momconds} can be viewed as nonlinear generalization of the results in \cite{cck2013,cck2015a}, which considered the Gaussian approximation for $\max_{1 \le j \le d} \sqrt{n} \bar{X}_{nj}$. Therefore, for U-statistics with a nonlinear kernel $h$ (possibly unbounded and discontinuous), the effect of higher-order terms than the H\'ajek projection to a linear subspace in the Hoeffding decomposition vanishes in the Gaussian approximation. For multivariate symmetric statistics of order two, to the best of our knowledge, the Gaussian approximation result (\ref{eqn:CLT-hyperrec-subexp}), (\ref{eqn:CLT-hyperrec-poly}), (\ref{eqn:CLT-hyperrec-explicit_rate}), and (\ref{eqn:CLT-kolmogorov-explicit_rate}) with the explicit convergence rate is new. When $d$ is fixed, the rate of convergence and the Edgeworth expansion of such statistics can be found in \cite{bickelgotzevanzwet1986,gotze1987,bentkusgotzevanzwet1997}. In those papers, assuming the Cram\'er condition on $g(X_1)$ and suitable moment conditions on $h(X_1,X_2)$, the Edgeworth expansion of U-statistics was established for the univariate case ($d=1$) with remainder $o(n^{-1/2})$ or $O(n^{-1})$ \cite{bickelgotzevanzwet1986,bentkusgotzevanzwet1997} and the multivariate case ($d>1$ fixed) with remainder $o(n^{-1/2})$ \cite{gotze1987}. In the latter work \cite{gotze1987}, it is unclear that how the constant in the error bound depends on the dimensionality parameter $d$.

Theorem \ref{thm:CLT-hyperrec-momconds} and Corollary \ref{cor:CLT-hyperrec-momconds} allow us to approximate the probabilities of $T_n$ belonging to the hyperrectangles in $\R^d$ by those probabilities of $Y$, with the knowledge of $\Gamma$. Such results are useful for approximating the quantiles of $\bar{T}_n$ by those of $\bar{Y}$. In practice, the covariance matrix $\Gamma$ and the H\'ajek projection terms $g(X_i), i=1,\cdots,n,$ depend on the underlying data distribution $F$, which is unknown. Thus, quantiles of $\bar{Y}$ need to be estimated in real applications. However, we shall see in Section \ref{sec:bootstraps} that Theorem \ref{thm:CLT-hyperrec-momconds} can still be used to derive valid and computable (i.e. fully data-dependent) methods to approximate the quantiles of $\bar{T}_n$.

\section{Bootstrap approximations}
\label{sec:bootstraps}

In this section, we consider {\it computable} approximations of the probabilities $\Prob(T_n \in A)$ for $A \in {\cal A}^{re}$. Before proceeding to the rigorous results, we shall explain our general strategy. The validity of the bootstrap procedures is established by a series of approximations
\begin{equation}
\label{eqn:approx-diagram}
{\cal L}(T_n) \approx_{(1)} {\cal L}(Y) \approx_{(2)} {\cal L}(Z^X | X_1^n) \approx_{(3)}  {\cal L}(T_n^\natural | X_1^n),
\end{equation}
where $Z^X$ is a conditionally mean-zero Gaussian random vector in $\R^d$ given the observed sample $X_1^n$. The choice of $Z^X$ and $T_n^\natural$ depends on the specific bootstrap method such that the conditional covariance matrix of $Z^X$ given $X_1^n$ is a consistent estimator of $\Gamma$ under the supremum norm. Step (1) follows from the GAR and CLT in Section \ref{sec:gaussian-approx}. Step (2) relies on a (conditional) Gaussian comparison principle and the tail probability inequalities of maximal U-statistics to bound the probability of the events on which the Gaussian comparison can be applied. Those tail probability inequalities are developed in the SM (Section \ref{sec:conc-ineq-ustat}), which are of independent interest and may be used for other high-dimensional problems. Step (3) is a conditional version of Step (1) given $X_1^n$.

\subsection{Empirical bootstrap}
\label{subsec:efron-bootstrap}

Let $X_1^*,\cdots,X_n^*$ be a bootstrap sample independently drawn from the empirical distribution $\hat{F}_n = n^{-1} \sum_{i=1}^n \delta_{X_i}$, where $\delta_x$ is the Dirac point mass at $x$. Define
\begin{equation}
\label{eqn:efron-bootstrap-ustat}
U_n^* = {1 \over n(n-1)} \sum_{1 \le i \neq j \le n} h(X_i^*, X_j^*).
\end{equation}
Then, the conditional distribution of $T_n^* = \sqrt{n}(U_n^*-\E[U^*_n | X_1^n]) / 2$ given $X_1^n$ is used to approximate the distribution of $T_n$. Here $T_n^\natural = T_n^*$ in (\ref{eqn:approx-diagram}). Note that $\E[U_n^* | X_1^n] = V_n$, where $V_n=n^{-2} \sum_{i,j=1}^n h(X_i,X_j)$ is a $V$-statistic. Let
$$
\xi_i \stackrel{\text{iid}}{\sim} \mn(1; 1/n, \cdots, 1/n).
$$
Denote $\mbf\xi_{n \times n} = (\xi_1,\cdots,\xi_n)$ and $\vX_{p \times n}:=X_1^n=(X_1,\cdots,X_n)$. Then we can write $\vX^* = (X_1^*, \cdots, X_n^*) = \vX \mbf\xi$. The key observation is that conditional on $\vX$, $U_n^*$ is a U-statistic of $\xi_1,\cdots,\xi_n$ since
$$
U_n^* = {1 \over n(n-1)} \sum_{1 \le i \neq j \le n} h(\vX\xi_i, \vX\xi_j).
$$
Therefore, we can perform the conditional Hoeffding decomposition as follows. Let
\begin{eqnarray*}
g^X(\xi_1) &=& \E[h(\vX\xi_1, \vX\xi_2) | \xi_1, X_1^n] - V_n \\
&=& {1\over n} \sum_{j=1}^n h(\vX\xi_1, X_j) - {1\over n^2} \sum_{i,j=1}^n h(X_i,X_j).
\end{eqnarray*}
Then, $\E[g^X(\xi_1) | X_1^n] = 0$ and 
\begin{equation}
\label{eqn:emprical_cov_mat_eb}
\hat\Gamma_n := \Cov(g^X(\xi_1) | X_1^n) = {1\over n^3} \sum_{i=1}^n \sum_{j=1}^n  \sum_{k=1}^n  h(X_i,X_j) h(X_i,X_k)^\top -V_n V_n^\top.
\end{equation}
For the special case $d=1$, by the strong law of large numbers for U-statistics \cite[Theorem A, page 190]{serfling1980}, we have with probability one 
$$\lim_{n \to \infty} \Var(g^X(\xi_1) | X_1^n) = \Var(g(X_1)) = \E\{ \E[h(X_1,X_2) | X_1] \}^2 - \{\E[h(X_1,X_2)]\}^2.$$
Therefore, we expect that $T_n^*$ is a reasonable approximation of $T_n$ and our goal to is bound the random quantity 
$$\rho^{B}(T_n, T^*_n) = \sup_{A \in {\cal A}^{re}} | \Prob(T_n \in A) - \Prob(T_n^* \in A | X_1^n) |.$$
In addition to (M.2), (E.1) and (E.2), we shall also assume that 
\begin{enumerate}
\item[(M.2')] $\E[|h_m(X,X)|^{2+\ell}] \le B_n^\ell$ for $\ell=1,2$ and for all $m=1,\cdots,d$.
\item[(E.1')] $\|h_m(X,X)\|_{\psi_1} \le B_n$ for all $m=1,\cdots,d$.
\item[(E.2')] $\E[\max_{1 \le m \le d} (|h_m(X,X)|/B_n)^q] \le 1$.
\end{enumerate}
(M.2'), (E.1') and (E.2') are the von Mises conditions on the empirical bootstrap of U-statistics \cite{bickelfreedman1981}, which require that the diagonal entries of the kernel $h$ obey the same moment conditions as the off-diagonal ones (M.2), (E.1) and (E.2), respectively. Without (M.2'), (E.1') and (E.2'), the empirical bootstrap (Theorem \ref{thm:eb-bootstrap-rate}) can fail and a counterexample was given in \cite{bickelfreedman1981}; see also \cite[Chapter 6.5]{lehmann2004}. For $\gamma \in (0,e^{-1})$, define
\begin{equation}
\varpi^{B}_{1,n}(\gamma) = \left({B_n^2 \log^5 (nd) \log^2(1/\gamma) \over n}\right)^{1/6} \quad \text{and} \quad \varpi^{B}_{2,n}(\gamma) = \left({B_n^2 \log^3 (nd) \over \gamma^{2/q} n^{1-2/q}}\right)^{1/3}.
\end{equation}

\begin{thm}[{\bf Main result II}: rate of convergence of the empirical bootstrap for U-statistics]
\label{thm:eb-bootstrap-rate}
Suppose that (M.1), (M.2) and (M.2') are satisfied. Assume that $\log(1/\gamma) \le K \log(dn)$ and $\log d \le \bar{b} n$ for some constants $K, \bar{b}>0$. \\
(i) If (E.1) and (E.1') hold, then there exists a constant $C := C(\ub,\bar{b},K)>0$ such that we have with probability at least $1-\gamma$
\begin{equation}
\label{eqn:eb-bootstrap-rate-subexp}
\rho^{B}(T_n, T^*_n) \le C \varpi_{1,n}.
\end{equation}
(ii) If (E.2) and (E.2') hold with $q \ge 4$, then there exists a constant $C := C(\ub,\bar{b},q,K)>0$ such that  we have with probability at least $1-\gamma$
\begin{equation}
\label{eqn:eb-bootstrap-rate-poly}
\rho^{B}(T_n, T^*_n) \le C \{ \varpi_{1,n} + \varpi^{B}_{2,n}(\gamma) \}.
\end{equation}
\end{thm}

Theorem \ref{thm:eb-bootstrap-rate} is non-asymptotic, which implies the asymptotic validity of the EB for U-statistics in the almost sure sense.

\begin{cor}[Asymptotic validity of the empirical bootstrap for U-statistics in the almost sure sense]
\label{thm:eb-bootstrap-asymptotic_validity}
Suppose that (M.1), (M.2) and (M.2') are satisfied and $\log d \le \bar{b} n$ for some constant $\bar{b}>0$. (i) Under (E.1) and (E.1'), we have $\Prob(\rho^{B}(T_n, T^*_n) \le C \varpi_{1,n} \text{ for all but finitely many } n) = 1$, where $C > 0$ is a constant depending only on $\ub$ and $\bar{b}$. In particular, if $B_n^2 \log^7(n d) = o(n)$, then $\rho^{B}(T_n, T^*_n) \to 0$ almost surely. (ii) Under (E.2) and (E.2') with $q > 4$, we have $\Prob(\rho^{B}(T_n, T^*_n) \le C \{\varpi_{1,n} + {\varpi'}^B_{2,n} \}  \text{ for all but finitely many } n) = 1$, where 
$${\varpi'}^B_{2,n} = \left({B_n^2 \log^3(n d) \log^{4/q}(n) \over n^{1-4/q}} \right)^{1/3}$$
and $C > 0$ is a constant depending only on $\ub, \bar{b}$, and $q$. In particular, if $B_n^2 \log^7(n d) = o(n)$ and $B_n^2 \log^3(n d) \log^{4/q}(n) = o(n^{1-4/q})$, then $\rho^{B}(T_n, T^*_n) \to 0$ almost surely. 
\end{cor}

\subsection{Randomly reweighted bootstrap with iid Gaussian weights}
\label{subsec:iid-weighted-bootstrap}

Let $w_1,\cdots,w_n$ be iid $N(1,1)$ random variables that are also independent of $X_1^n$ and $Y$. Consider 
\begin{equation}
\label{eqn:randweight-bootstrap-ustat}
U_n^\diamond = {1 \over n(n-1)} \sum_{1 \le i \neq j \le n} w_i w_j h(X_i, X_j),
\end{equation}
Then, $U_n^\diamond$ is the stochastically reweighted version of $U_n$ and it can also be viewed as a random quadratic form in $w_1,\cdots,w_n$.  Denote $T^\diamond = \sqrt{n} (U_n^\diamond - U_n) / 2$ and $T_n^\natural = T_n^\diamond$ in (\ref{eqn:approx-diagram}). Since the main focus of this paper is to approximate the distribution of the centered U-statistics; i.e $\theta := \E h = 0$, we first consider the bootstrap of the centered U-statistics of the random quadratic form (\ref{eqn:randweight-bootstrap-ustat}) and discuss the effect of centering in the bootstraps in Remark \ref{rmk:effect-of-centering-bootstraps}.

\begin{thm}[{\bf Main result III}: rate of convergence of the randomly reweighted bootstrap for centered U-statistics]
\label{thm:iid-weighted-bootstrap}
Assume that $\theta = 0$. Suppose that (M.1) and (M.2) are satisfied. Assume that $\log(1/\gamma) \le K \log(dn)$ and $\log d \le \bar{b} n$ for some constants $K, \bar{b}>0$. (i) If (E.1) holds, then there exists a constant $C := C(\ub,\bar{b},K)>0$ such that we have $\rho^{B}(T_n, T^\diamond_n) \le C \varpi_{1,n}$ holds with probability at least $1-\gamma$. (ii) If (E.2) holds with $q \ge 4$, then there exists a constant $C := C(\ub,\bar{b},q,K)>0$ such that we have $\rho^{B}(T_n, T^\diamond_n) \le C \{ \varpi_{1,n} + \varpi^{B}_{2,n}(\gamma) \}$ holds with probability at least $1-\gamma$.
\end{thm}

From Theorem \ref{thm:eb-bootstrap-rate} and \ref{thm:iid-weighted-bootstrap}, we see that the empirical and the randomly reweighted bootstraps are first-order equivalent, both achieving the same uniform rate of convergence for approximating the probabilities $\Prob(T_n \in A)$ for $A \in {\cal A}^{re}$. However, unlike the EB, the randomly reweighted bootstrap does not assume the von Mises moment conditions on the diagonal entries. 
%Therefore, similarly as the EB in Corollary \ref{cor:comparison_with_naive_gaussian_wild_bootstrap}, asymptotic validity of the iid randomly reweighted bootstrap can be established when $d = O(e^{n^c})$ for some constant $c \in (0,1/7)$ and $B_n = O(1)$ without assuming (M.2'), (E.1'), and (E.2').

\begin{rmk}[Effect of centering in the randomly reweighted bootstrap]
\label{rmk:effect-of-centering-bootstraps}
If $\theta \neq 0$, then we can show that the iid reweighted bootstrap $T^\diamond_n$ is not an asymptotically valid bootstrap approximation for $T_n$. The reason is that centering is a key structure to maintain in the conditional distribution of $T^\diamond_n$. Therefore, in the case, we shall consider the following modified version
\begin{eqnarray}
\label{eqn:Uflat}
U^\flat_n = {1 \over n (n-1)} \sum_{1 \le i \neq j \le n} w_i w_j h(X_i,X_j) - 2 (\bar{w}-1) U_n
\end{eqnarray}
and $T^\flat_n = \sqrt{n}(U^\flat_n-U_n)/2$. Then, $T_n^\flat = T_n^\diamond - \sqrt{n}(\bar{w}-1) U_n$. Since $w_i$ are iid $N(1,1)$, we have $\sqrt{n}(\bar{w}-1) U_n | X_1^n \sim N(0, U_n U_n^\top)$, which is not asymptotically vainishing for $\theta \neq 0$. Therefore, without the centering term $2(\bar{w}-1) U_n$ in $U_n^\flat$, $T_n^\diamond$ is not an asymptotically tight sequence for approximating $T_n$.

\begin{thm}[Rate of convergence of the randomly reweighted bootstrap for non-centered U-statistics]
\label{thm:iid-weighted-bootstrap-no-centering}
Suppose that (M.1) and (M.2) are satisfied. Assume that $\log(1/\gamma) \le K \log(dn)$ and $\log d \le \bar{b} n$ for some constants $K, \bar{b}>0$. (i) If (E.1) holds, then there exists a constant $C := C(\ub,\bar{b},K)>0$ such that we have $\rho^{B}(T_n, T^\flat_n) \le C \varpi_{1,n}$ holds with probability at least $1-\gamma$. (ii) If (E.2) holds with $q \ge 4$, then there exists a constant $C := C(\ub,\bar{b},q,K)>0$ such that  we have $\rho^{B}(T_n, T^\flat_n) \le C \{ \varpi_{1,n} + \varpi^{B}_{2,n}(\gamma) \}$ holds with probability at least $1-\gamma$.
\end{thm}

Theorem \ref{thm:iid-weighted-bootstrap-no-centering} is valid regardless $\theta \neq 0$ and $\theta = 0$ since in the latter case, $\sqrt{n}(\bar{w}-1) U_n$ is conditionally negligible compared with $T_n^\diamond$. For the EB, centering in the empirical analog $\hat\Gamma_n$ of the covariance matrix $\Gamma$ is automatically fulfilled; see (\ref{eqn:emprical_cov_mat_eb}). Similar comments apply to the Gaussian multiplier bootstrap $T_n^\sharp$ in Section \ref{subsec:gaussian-multiplier-bootstrap-jackknife}.
\qed
\end{rmk}

\subsection{Gaussian multiplier bootstrap with jackknife covariance matrix estimator}
\label{subsec:gaussian-multiplier-bootstrap-jackknife}

The iid reweighted bootstrap is closely related to the Gaussian multiplier bootstrap with the jackknife estimator of the covariance matrix of $T_n$. Let $e_1,\cdots,e_n$ be iid $N(0,1)$ random variables that are independent of $X_1^n$ and $Y$ and 
\begin{equation}
\label{eqn:multiplier_bootstrap}
T^\sharp_n = {1 \over \sqrt{n}} \sum_{i=1}^n \Big[ {1 \over n-1} \sum_{j \neq i} h(X_i, X_j) - U_n \Big] e_i.
\end{equation}
Define
\begin{equation}
\label{eqn:jk_Gamma_n}
\hat\Gamma_n^{JK} = {1 \over (n-1) (n-2)^2} \sum_{i=1}^n \sum_{j \neq i} \sum_{k \neq i} (h(X_i,X_j) -U_n) (h(X_i,X_k) - U_n)^\top.
\end{equation}
Then, $\hat\Gamma_n^{JK}$ is the jackknife estimator of the covariance matrix of $T_n$ \cite{callaertveraverbeke1981} and $T^\sharp_n | X_1^n \sim N(0, \tilde\Gamma_n)$, where 
\begin{equation}
\label{eqn:tilde_Gamma_n}
\tilde\Gamma_n = {(n-2)^2 \over n(n-1)} \hat\Gamma_n^{JK}.
\end{equation}
Therefore, it is interesting to view the Gaussian multiplier bootstrap $T^\sharp_n$ as a plug-in estimator of the distribution of $T_n$ by its jackknife covariance matrix estimator. To distinguish the Gaussian wild bootstrap $\hat{L}_0^*$ in (\ref{eqn:gwb_hajek}) (c.f. Remark \ref{rmk:comparison_with_naive_gaussian_wild_bootstrap}), we call $T^\sharp_n$ the {\it jackknife Gaussian multiplier bootstrap}.

\begin{thm}[{\bf Main result IV}: rate of convergence of the jackknife Gaussian multiplier bootstrap for U-statistics]
\label{thm:mb-bootstrap-rate}
Suppose that (M.1) and (M.2) are satisfied. Assume that $\log(1/\gamma) \le K \log(dn)$ and $\log d \le \bar{b} n$ for some constants $K, \bar{b}>0$. (i) If (E.1) holds, then there exists a constant $C := C(\ub,\bar{b},K)>0$ such that we have $\rho^{B}(T_n, T^\sharp_n) \le C \varpi_{1,n}$ holds with probability at least $1-\gamma$. (ii) If (E.2) holds with $q \ge 4$, then there exists a constant $C := C(\ub,\bar{b},q,K)>0$ such that  we have $\rho^{B}(T_n, T^\sharp_n) \le C \{ \varpi_{1,n} + \varpi^{B}_{2,n}(\gamma) \}$ holds with probability at least $1-\gamma$.
\end{thm}

In the special case $h(x_1,x_2) = (x_1 + x_2) / 2$ for $x_1,x_2 \in \R^d$, we have $U_n = \bar{X}_n = n^{-1} \sum_{i=1}^n X_i$ is the sample mean vector and $T_n = \sqrt{n} (\bar{X}_n - \theta) / 2$ where $\theta = \E(X_1)$. Some algebra shows that $\hat\Gamma_n^{JK} =  [4 (n-1)]^{-1} \sum_{i=1}^n (X_i - \bar{X}_n) (X_i - \bar{X}_n)^\top$, $\hat\Gamma_n = n^{-1} (n-1) \hat\Gamma_n^{JK}$ in (\ref{eqn:emprical_cov_mat_eb}), and $\tilde\Gamma_n = [(n-2)^2 / (n(n-1))]\hat\Gamma_n^{JK}$ in (\ref{eqn:tilde_Gamma_n}). Then, $T^\sharp_n \sim N(0, \tilde\Gamma_n)$ which is the equivalent to the multiplier bootstrap of \cite{cck2015a}. Therefore, for iid samples, Theorem \ref{thm:eb-bootstrap-rate} and \ref{thm:mb-bootstrap-rate} are nonlinear generalizations of the empirical and Gaussian multiplier bootstraps considered in \cite{cck2015a}.

\begin{rmk}[Comparison with the Gaussian wild bootstrap of \cite{chen2016a}]
\label{rmk:comparison_with_naive_gaussian_wild_bootstrap}
In \cite{chen2016a}, a Gaussian wild bootstrap based on decoupling was proposed. Specifically, let $X'_1,\cdots,X'_n$ be an independent copy of $X_1,\cdots,X_n$. The H\'ajek projection terms $g(X_i), i =1, \cdots,n,$ are estimated by
$$
\hat{g}_i = {1 \over n} \sum_{j=1}^n h(X_i,X'_j) - {1 \over n (n-1)} \sum_{1 \le j \neq l \le n} h(X'_j, X'_l).
$$
Since $g(X_i) = \E[h(X_i,X')|X_i] - \E[h(X,X')]$, $\hat{g}_i$ can be viewed as an unbiased estimator of $g(X_i)$ conditionally on $X_i$ for $i = 1,\cdots,n$. Then the Gaussian wild bootstrap procedure is defined as
\begin{equation}
\label{eqn:gwb_hajek}
\hat{L}_0^* = \max_{1 \le m \le d} {1 \over \sqrt{n}} \sum_{i=1}^n \hat{g}_{im} e_i,
\end{equation}
where $e_i$ are iid $N(0,1)$ random variables. Let $a_{\hat{L}_0^*}(\alpha)$ be the $\alpha$-th conditional quantile of $\hat{L}_0^*$ given $X_1,\cdots,X_n$ and $X'_1,\cdots,X'_n$. Similarly, denote $a_{\bar{T}_n^\sharp}(\alpha)$ as the $\alpha$-th conditional quantile of $\bar{T}_n^\sharp$ given $X_1, \cdots, X_n$. Let $K \in (0,1)$ be a constant. Assuming (M.1), (M.2), and in addition $D_2 \le 1$, it was shown in \cite{chen2016a} that: (i) if $B_n^2 \log^7(dn) \le n^{1-K}$ and (E.1) holds, then $\sup_{\alpha \in (0,1)} | \Prob(\bar{T}_n \le a_{\hat{L}_0^*}(\alpha)) - \alpha | \le C n^{-K/8};$ (ii) if $B_n^4 \log^7(dn) \le n^{1-K}$ and (E.2) holds with $q=4$, then $\sup_{\alpha \in (0,1)} | \Prob(\bar{T}_n \le a_{\hat{L}_0^*}(\alpha)) - \alpha | \le C n^{-K/12}.$ Here, the constant $C > 0$ depends only on $\ub$ in (M.1) in both cases. The following theorem shows that the jackknife Gaussian multiplier bootstrap improves the convergence rate of Gaussian wild bootstrap in \cite{chen2016a}.
\begin{thm}
\label{thm:comparison_with_naive_gaussian_wild_bootstrap}
Suppose that (M.1) and (M.2) are satisfied. Let $K \in (0,1)$. (i) Assume (E.1). If $B_n^2 \log^7(dn) \le n^{1-K}$, then
\begin{equation}
\label{eqn:comparison_with_naive_gaussian_wild_bootstrap}
\sup_{\alpha \in (0,1)} | \Prob(\bar{T}_n \le a_{\bar{T}_n^\sharp}(\alpha)) - \alpha | \le C n^{-K/6},
\end{equation}
where $\bar{T}^\sharp_n = \max_{1 \le j \le d} T^\sharp_{nj}$ and $C > 0$ is a constant depending only on $\ub$. \\
(ii) Assume (E.2) with $q \ge 4$. If $B_n^2 \log^7(dn) \le n^{1-K}$ and $B_n^4 \log^6(d) \le n^{2-4(1+K/6)/q-K}$, then we have (\ref{eqn:comparison_with_naive_gaussian_wild_bootstrap}) with the constant $C$ depending only on $\ub$ and $q$.
\end{thm}

In particular, for both subexponential and uniform polynomial kernels, the convergence rate of jackknife Gaussian multiplier bootstrap $T_n^\sharp$ is $O(n^{-K/6})$. The improved dependence on $n$ is due to two reasons. First, we established a GAR for high-dimensional U-statistics with sharper rate (Theorem \ref{thm:CLT-hyperrec-momconds}). Second, $T_n^\sharp$ does not estimate the individual terms $g(X_i)$ in the H\'ajek projection which requires a strong control on the maximal deviation $|\hat{g}_i - g(X_i)|_\infty$ over $i=1,\cdots,n$; see Lemma C.4 in \cite{chen2016a}. Instead, $T_n^\sharp$ implicitly constructs an estimator $\tilde\Gamma_n$ in (\ref{eqn:tilde_Gamma_n}) for the covariance matrix of the linear projection part in the Gaussian approximation. There is a slight trade-off between the moment and scaling limit for uniform polynomial kernels in Theorem \ref{thm:comparison_with_naive_gaussian_wild_bootstrap} since the conditions $B_n^2 \log^7(dn) \le n^{1-K}$ and $B_n^4 \log^6(d) \le n^{2-4(1+K/6)/q-K}$ are implied by either $B_n^4 \log^7(dn) \le n^{1-7K/6}$ for $q = 4$ or $B_n^4 \log^7(dn) \le n^{1-K}$ for $q \ge 4(1+K/6)$. However, in either case, Theorem \ref{thm:comparison_with_naive_gaussian_wild_bootstrap} asymptotically permits $d = O(e^{n^c})$ for some $c \in (0, 1/7)$ when $q \ge 4$ and $B_n = O(1)$.
\qed
\end{rmk}

\section{Statistical applications}
\label{sec:stat_apps}

In this section, we present two statistical applications for bootstrap methods in Section \ref{subsec:efron-bootstrap}. For simplicity, we only present the results for the jackknife Gaussian multiplier bootstrap $T^\sharp_n$ defined in (\ref{eqn:multiplier_bootstrap}). Similar results hold for other bootstraps in Section \ref{subsec:efron-bootstrap}. Two additional examples can be found in the SM. Throughout the section, we consider the bootstrap of the sample covariance matrix (i.e. $h(x_1,x_2) = (x_1-x_2) (x_1-x_2)^\top / 2$ and $\R^d = \R^{p \times p}$). We define $\bar{T}_n^\sharp = 2 n^{-1/2} \max_{1 \le m,k \le p} |T_{n,mk}^\sharp|$ by rescaling and denote the $\alpha$-th conditional quantile of $\bar{T}_n^\sharp$ given the data $X_1^n$ as
\begin{equation}
\label{eqn:conditional_quantitle_jackknife_gaussian_multiplier_bootstrap}
a_{\bar{T}_n^\sharp}(\alpha) = \inf\{ t \in \mathbb{R} : \Prob_e(\bar{T}_n^\sharp \le t) \ge \alpha \}.
\end{equation}
where $\Prob_e$ is the probability taken w.r.t. the iid $N(0,1)$ random variables $e_1,\cdots,e_n$. We can compute the conditional quantile $a_{T^\sharp_n}(\alpha)$ by repeatedly drawing independent samples of the standard Gaussian random variables $e_1,\cdots,e_n$.

\subsection{Tuning parameter selection for the thresholded covariance matrix estimator}
\label{subsec:tuning_selection_thresholded_cov_mat}

Consider the problem of {\it sparse} covariance matrix estimation. Let $r \in [0,1)$ and
\begin{equation*}
{\cal G}(r, \zeta_p) = \Big\{ \Sigma \in \mathbb{R}^{p\times p} :  \max_{1 \le m \le p} \sum_{k=1}^p |\sigma_{mk}|^r \le \zeta_p \Big\}
\end{equation*}
be the class of sparse covariance matrices in terms of the strong $\ell^r$-ball. Here, $\zeta_p > 0$ may grow with $p$. Let $\hat{S}_n = \{\hat{s}_{mk}\}_{m,k=1}^p$ be the sample covariance matrix and
$$
\hat\Sigma(\tau) = \{\hat{s}_{mk} \vone\{|\hat{s}_{mk}| > \tau \} \}_{m,k=1}^p, \qquad \tau \ge 0,
$$
be the thresholded sample covariance matrix estimator of $\Sigma$. A similar matrix class as ${\cal G}(r, \zeta_p)$ was introduced in \cite{bickellevina2008b} by further requiring that $\max_{1 \le m \le p} \sigma_{mm} \le C_0$ for some constant $C_0 > 0$. Here, we do not assume the diagonal entires of $\Sigma$ are bounded. Performance bounds of the thresholded estimator $\hat\Sigma(\tau)$ critically depend on the tuning parameter $\tau$. The oracle choice of the threshold for establishing the rate of convergence under the spectral and Frobenius norms is $\tau_\diamond = |\hat{S}_n-\Sigma|_\infty$. Note that $\tau_\diamond$ is a random variable and its distribution depends on the unknown underlying data distribution $F$. High probability bounds of $\tau_\diamond$ were given in \cite{bickellevina2008b,chenxuwu2013a} and asymptotic properties of $\hat\Sigma(\tau)$ were analyzed in \cite{bickellevina2008b,caizhou2011a} for iid sub-Gaussian data and in \cite{chenxuwu2013a,chenxuwu2016+} for heavy-tailed time series with polynomial moments. In both scenarios, the rates of convergence were obtained with the Bonferroni (i.e. the union bound) technique and one-dimensional concentration inequalities. In the problem of the high-dimensional sparse covariance matrix estimation, data-dependent tuning parameter selection is often empirically done with the cross-validation (CV) and its theoretical properties when compared with $\tau_\diamond$ largely remain unknown since the CV threshold does not approximate $\tau_\diamond$. Here, we provide a principled and fully data-dependent way to determine the threshold $\tau$. We first consider subgaussian observations.

\begin{defn}[Subgaussian random variable]
\label{defn:subgaussian_rv}
A random variable $X$ is said to be {\it subgaussian} with mean zero and {\it variance factor} $\nu^2$, if 
\begin{equation}
\label{eqn:subgaussian_rv}
\E [\exp(X^2 / \nu^2) ] \le \sqrt{2}.
\end{equation}
Denote $X \sim \text{subgaussian}(\nu^2)$. In particular, if $X \sim N(0,\sigma^2)$, then $X \sim \text{subgaussian}(4 \sigma^2)$.
\end{defn}

The upper bound $\sqrt{2}$ in (\ref{eqn:subgaussian_rv}) is not essential and it is chosen for conveniently comparing with $\|X\|_{\psi_2}$: if $X \sim \text{subgaussian}(\nu^2)$, then $\nu^2 \ge \|X\|_{\psi_2}$. Clearly, bounded random variables are subgaussian. In addition, random variables with the mixture of subgaussian distributions are also subgaussian. Let $K$ be a positive integer and $\{\pi_k\}_{k=1}^K$ be subgaussian distributions with the variance factors $\{\nu_k^2\}_{k=1}^K$. If a random variable $X$ follows a mixture of $K$ subgaussian distributions $\sum_{k=1}^K \varepsilon_k \pi_k$ with $0 \le \varepsilon_k \le 1$ and $\sum_{k=1}^K \varepsilon_k = 1$, then $X \sim \text{subgaussian}(\bar\nu^2)$, where $\bar\nu^2 = \max\{\nu_1^2, \cdots, \nu_K^2\}$. In general, the variance factor for a subgaussian random variable is {\it not} equivalent to the variance. For a sequence of random variables $X_n, n = 1,2,\cdots,$ if $X_n \sim \text{subgaussian}(\nu_n^2)$ and $\sigma_n^2 = \Var(X_n)$, then by Markov's inequality, we always have $\sigma_n^2 \le \sqrt{2} \nu_n^2$, while $\nu_n^2$ may diverge at faster rate than $\sigma_n^2$ as $n \to \infty$. Below we shall give two such examples.

\begin{exmp}[Mixture of two Gaussian distributions]
\label{exmp:mixture_two_gaussians}
Let $\{X_n\}_{n=1}^\infty$ be a sequence of random variables with the distribution $F_n = (1-\varepsilon_n) N(0,1) + \varepsilon_n N(0,a_n^2)$. Suppose that $a_n \ge 1$, $a_n \to \infty$ as $n \to \infty$, and consider $\varepsilon_n = a_n^{-4}$. Then, we have $X_n \sim \text{subgaussian}(4 a_n^2)$, $\Var(X_n) \asymp 1$, $\|X_n\|_4 \asymp 1$, $\|X_n\|_6 \asymp a_n^{1/3}$, and $\|X_n\|_8 \asymp a_n^{1/2}$. The distribution $F_n$ can be viewed as a $\varepsilon_n$-contaminated one-dimensional normal distribution given by (\ref{eqn:eps_contaminated_normal_distn}) in the SM.
\end{exmp}

\begin{exmp}[Mixture of two symmetric discrete distributions]
\label{exmp:P_2}
Let $\pi_1$ be the distribution of a Rademacher random variable $Y$ (i.e. $\Prob(Y = \pm 1) = 1/2$) and $\pi_2$ be the distribution of a discrete random variable $Z_n$ such that $\Prob(Z_n = \pm a_n) = (2 a_n^2)^{-1}$ and $\Prob(Z_n = 0) = 1 - a_n^{-2}$. Let $\{X_n\}_{n=1}^\infty$ be a sequence of random variables with the distribution $F_n = (1-\varepsilon_n) \pi_1 + \varepsilon_n \pi_2$, where $\varepsilon_n = 2/(a_n^2-1)$, $a_n > \sqrt{3}$, and $a_n \to \infty$ as $n \to \infty$. Then $X_n \sim \text{subgaussian}(C a_n^2)$ for some large enough constant $C > 0$ and elementary calculations show that $\Var(X_n) = 1$, $\|X_n\|_4 =  3^{1/4}$, $\|X_n\|_6 \asymp a_n^{1/3}$, and $\|X_n\|_8 \asymp a_n^{1/2}$.
\end{exmp}

%\begin{defn}[A family of symmetric discrete distributions]
%\label{def:k_point_sym_dist}
%Let $L$ be a positive integer. A discrete random variable $X$ has a symmetric and $(2L+1)$-point distribution, denote as $X \in {\cal P}_L$, if there are two sequences of reals $\{a_l\}_{l=1}^L$ and $\{p_l\}_{l=1}^L$ such that: (i) $0 < a_1 < \cdots < a_L$, $0 \le p_l \le 1$, $\sum_{l=1}^L p_l \le 1$; (ii) $X$ takes values in $\{0, \pm a_1, \cdots, \pm a_L\}$, $\Prob(X = \pm a_l) = p_l/2$, and $\Prob(X = 0) = 1- \sum_{l=1}^L p_l$.
%\end{defn}
%
%If $X \in {\cal P}_L$, then $X$ is bounded and thus subgaussian with mean zero and variance factor $C a_L^2$ for some large enough constant $C > 0$. 
%
%\begin{exmp}
%\label{exmp:P_2}
%Let $\{X_n\}_{n=1}^\infty$ be a sequence of random variables such that $\Prob(X_n = \pm 1) = p_1 / 2$, $\Prob(X_n = \pm a_n) = p_2 / 2$ and $\Prob(X_n = 0) = 1 - p_1-p_2$, where $a_n > 3$ be a sequence of diverging real numbers, $p_1 = (a_n^2 - 3) / (a_n^2-1)$, and $p_2 = 2 / [a_n^2 (a_n^2-1)]$. Then $X_n \in {\cal P}_2$ and $X_n \sim \text{subgaussian}(C a_n^2)$ for some large enough constant $C > 0$. Elementary calculations show that $\Var(X_n) = 1$, $\|X_n\|_4 =  3^{1/4}$, $\|X_n\|_6 \asymp a_n^{1/3}$, and $\|X_n\|_8 \asymp a_n^{1/2}$.
%\end{exmp}

Therefore, in the statistical applications for subgaussian data, we allow $\nu^2_n \to \infty$ as $n \to \infty$. Let $\xi_q = \max_{1 \le k \le p} \|X_{1k}\|_q$ and recall $\Gamma = \Cov(g(X_1))$.

%As a simple example, let $a_n > 0$ be a sequence of real numbers such that $a_n \to \infty$ and consider random variables $X_1,\cdots,X_n$ such that $\Prob(X_n = \pm a_n) = (2 a_n^2)^{-1}$ and $\Prob(X_n = 0) = 1 - a_n^{-2}$. Obviously, $\E X_n = 0$ and $\Var(X_n) = 1$. Let $\nu_n^2 = C a_n^2$ for some constant $C > 0$. Then $\E[\exp(X_n^2 / \nu_n^2)] = (1-a_n^{-2}) + a_n^{-2} e^{C^{-1}} \le \sqrt{2}$ for all large enough $n$; i.e. $X_n \sim \text{subgaussian}(C a_n^2)$. 

%However, for the Gaussian random variables $X \sim N(0, \sigma^2)$, we have $\sigma^2 = \nu^2/4$.

\begin{thm}[{\bf Main result V}: data-driven threshold selection: subgaussian observations]
\label{thm:thresholded_cov_mat_rate_adaptive}
Let $\nu_n \ge 1$ and $X_i$ be iid mean zero random vectors in $\R^p$ such that $X_{ik} \sim \text{subgaussian}(\nu_n^2)$ for all $k=1,\cdots,p$ and $\Sigma \in {\cal G}(r, \zeta_p)$. Suppose that there exist constants $C_i > 0, i=1,\cdots,3,$ such that $\Gamma_{(j,k),(j,k)} \ge C_1$, $\xi_6 \le C_2 \nu_n^{1/3}$ and $\xi_8 \le C_3 \nu_n^{1/2}$ for all $j,k=1,\cdots,p$. Let $\alpha, \beta, K \in (0, 1)$ and $\tau_* = \beta^{-1} a_{\bar{T}_n^\sharp}(1-\alpha)$. If $\nu_n^4 \log^7(np) \le C_4 n^{1-K}$, then we have
\begin{eqnarray}
\label{eqn:thresholded_cov_mat_rate_spectral_adaptive}
\|\hat\Sigma(\tau_*)-\Sigma\|_2 &\le& \left[{3+2\beta \over \beta^{1-r}}+ \left({\beta \over 1-\beta}\right)^r \right]  \zeta_p  a_{\bar{T}_n^\sharp}^{1-r}(1-\alpha), \\
\label{eqn:thresholded_cov_mat_rate_F_adaptive}
{1 \over p} | \hat\Sigma(\tau_*)-\Sigma|_F^2 &\le& 2\left[{4+3\beta^2 \over \beta^{2-r}}+ 2 \left({\beta \over 1-\beta}\right)^r \right]  \zeta_p a_{\bar{T}_n^\sharp}^{2-r}(1-\alpha),
\end{eqnarray}
with probability at least $1-\alpha-C n^{-K/6}$ for some constant $C > 0$ depending only on $C_1,\cdots,C_4$. In addition, $\E[a_{\bar{T}_n^\sharp}(1-\alpha)] \le C' \xi_4^2 (\log(p)/n)^{1/2}$ and
\begin{equation}
\label{eqn:tau_*-bound_subgaussian}
\E[\tau_*] \le C' \beta^{-1} \xi_4^2 (\log(p)/n)^{1/2},
\end{equation}
where $C' > 0$ is a constant depending only on $\alpha$ and $C_1,\cdots,C_4$.
\end{thm}

\begin{rmk}[Comments on the conditions in Theorem \ref{thm:thresholded_cov_mat_rate_adaptive}]
\label{rmk:cov-kernel-comments-on-conditions}
Conditions on the growth rate $\xi_6 \le C_2 \nu_n^{1/3}$ and $\xi_8 \le C_3 \nu_n^{1/2}$ are satisfied by Example \ref{exmp:mixture_two_gaussians} and \ref{exmp:P_2}. The non-degeneracy condition $\Gamma_{(j,k),(j,k)} \ge C_1$ is quite mild. Consider the multivariate cumulants of the joint distribution of the random vector $X = (X_1,\cdots,X_p)^\top$ following a distribution $F$ in $\mathbb{R}^p$. Let $\chi(t) = \E[\exp(\iota t^\top X)]$ be the characteristic function of $X$, where $t = (t_1,\cdots,t_p)^\top \in \mathbb{R}^p$ and $\iota = \sqrt{-1}$. Then, the {\it multivariate cumulants} $\kappa_{r_1 r_2 \cdots r_p}^{1 2 \cdots p}$ of the joint distribution of $X$ are the coefficients in the expansion
$$
\log \chi(t) = \sum_{r_1,r_2,\cdots,r_p = 0}^\infty  \kappa_{r_1 r_2 \cdots r_p}^{1 2 \cdots p} {(\iota t_1)^{r_1} (\iota t_2)^{r_2} \cdots (\iota t_p)^{r_p} \over r_1! r_2! \cdots r_p!}.
$$
For the covariance matrix kernel, we have
\begin{equation}
\label{eqn:write-Gamma_g-as-fourth-culmulant}
\Gamma_{(j,k), (m,l)} = (\kappa_{1111}^{jkml} + \sigma_{jm}\sigma_{kl} + \sigma_{jl}\sigma_{km}) / 4,
\end{equation}
where $\kappa_{1111}^{jkml}$ is the joint fourth-order cumulants of $F$. %By the Cauchy-Schwarz inequality $\sigma_{jk}^2 \le \sigma_{jj} \sigma_{kk}$.
Therefore, if $\kappa_{1111}^{jkjk} \ge 4 C_1$, then $\Gamma_{(j,k),(j,k)} \ge C_1$.

If the data follow a distribution in the elliptic family \cite[Chapter 1]{muirhead1982}, then the condition $\Gamma_{(j,k),(j,k)} \ge C_1$ is equivalent to $\min_{1 \le j \le p} \sigma_{jj} \ge C$ for some constant $C > 0$ depending only on $C_1$. To see this, for $F$ in the elliptic family, it is known that $\kappa_{1111}^{jkml} = \kappa (\sigma_{jk}\sigma_{ml}+\sigma_{jm}\sigma_{kl}+\sigma_{jl}\sigma_{km})$, where $\kappa$ is the kurtosis of $F$. Therefore, $\Gamma_{(j,k),(j,k)} = [(2 \kappa + 1) \sigma_{jk}^2 + (\kappa + 1) \sigma_{jj} \sigma_{kk}] / 4$ and $\Gamma_{(j,k),(j,k)} \ge C_1$ if and only if there exists a constant $C > 0$ such that $\sigma_{jj} \ge C$ for all $j=1,\cdots,p$.
\qed
\end{rmk}

%Therefore, if the data follow the Gaussian distribution, then the bootstrapped thresholded covariance matrix estimator $\hat\Sigma(\tau_*)$ attains (\ref{eqn:thresholded_cov_mat_rate_spectral_adaptive}) and (\ref{eqn:thresholded_cov_mat_rate_F_adaptive}) when $p = O(\exp(n^{(1-K)/7}))$.

%However, we shall emphasize that, for $\Sigma \in {\cal G}(r, C_0, \zeta_p)$, although the diagonal entries in $\Sigma$ are uniformly bounded by a constant $C_0$, we do allow $\nu^2$ to grow with $n$ in the subgaussian distribution, in which case the bootstrap approach can have advantages over the non-adaptive minimax thresholding procedure (see the paragraphs below for more detailed discussions).

There are a number of interesting features of Theorem \ref{thm:thresholded_cov_mat_rate_adaptive}. Consider $r=0$; i.e. $\Sigma$ is truly sparse such that $\max_{1 \le m \le p} \sum_{k=1}^p \vone\{\sigma_{mk} \neq 0\} \le \zeta_p$ for $\Sigma \in {\cal G}(0, \zeta_p)$. Then we can take $\beta=1$ (i.e. $\tau_* = a_{\bar{T}_n^\sharp}(1-\alpha)$) and the convergence rates are 
$$
\|\hat\Sigma(\tau_*)-\Sigma\|_2 \le 6 \zeta_p \tau_* \quad \text{ and } \quad p^{-1} | \hat\Sigma(\tau_*)-\Sigma|_F^2 \le 18 \zeta_p \tau_*^2.
$$
Hence, the tuning parameter can be adaptively selected by bootstrap samples while the rate of convergence is {\it nearly optimal} in the following sense.
Since the distribution of $\tau_*$ mimics that of $\tau_\diamond$, $\hat\Sigma(\tau_*)$ achieves the same convergence rate as the thresholded estimator $\hat\Sigma(\tau_\diamond)$ for the oracle choice of the threshold $\tau_\diamond$ with probability at least $1-\alpha-C n^{-K/6}$. On the other hand, the bootstrap method is not fully equivalent to the oracle procedure in terms of the constants in the estimation error bounds. Suppose that we know the support $\Theta$ of $\Sigma$, i.e. locations of the nonzero entries in $\Sigma$. Then, the {\it oracle estimator} is simply $\breve\Sigma = \{\hat{s}_{mk} \vone\{(m,k) \in \Theta\} \}_{m,k=1}^p$ and we have
\begin{eqnarray*}
\|\breve\Sigma - \Sigma\|_2 &\le& \max_{1 \le m \le p} \sum_{k=1}^p |\hat{s}_{mk}-\sigma_{mk}| \vone\{(m,k) \in \Theta\} \\
&\le& |\hat{S}_n-\Sigma|_\infty \max_{1 \le m \le p} \sum_{k=1}^p \vone\{(m,k) \in \Theta\}  \le \tau_\diamond \zeta_p.
\end{eqnarray*}
Therefore, the constant of the convergence rate for the bootstrap method does not attain the oracle estimator. However, we shall comment that $\beta$ is not a tuning parameter since it does not depend on $F$ and the effect of $\beta$ only appears in the constants in front of the convergence rates (\ref{eqn:thresholded_cov_mat_rate_spectral_adaptive}) and (\ref{eqn:thresholded_cov_mat_rate_F_adaptive}).

Assuming that the observations are subgaussian$(\nu^2)$ and the variance factor $\nu^2$ is a {\it fixed} constant, it is known that the threshold value $\tau_\Delta = C(\nu) \sqrt{\log(p)/n}$ achieves the minimax rate for estimating the sparse covariance matrix \cite{caizhou2011a}. Compared with the minimax optimal tuning parameter $\tau_\Delta$, our bootstrap threshold $\tau_*$ exhibits several advantages for certain subgaussian distributions which we shall highlight (with stronger side conditions).

First, the bootstrap threshold $\tau_*$ does not need the knowledge of $\nu_n^2$ and it allows $\nu_n^2 \to \infty$ as $n \to \infty$. In this case, from (\ref{eqn:tau_*-bound_subgaussian}), the bootstrap threshold $\tau_* = O_\Prob( \xi_4^2 \sqrt{\log(p)/n})$, where the constant of $O_\Prob(\cdot)$ depends only on $\alpha,\beta,C_1,\cdots,C_4$ in Theorem \ref{thm:thresholded_cov_mat_rate_adaptive}, while the universal thresholding rule $\tau_\Delta = C' \nu_n^2 \sqrt{\log(p)/n}$. Therefore, if $\xi_4 = o(\nu_n)$, then $\tau_* = o_\Prob(\tau_\Delta)$ and the bootstrap threshold $\tau_*$ is less conservative than the minimax threshold. For instance, suppose that $X_{im}, i=1,\cdots,n; m=1,\cdots,p$ have the same marginal distribution in Example \ref{exmp:mixture_two_gaussians} (continuous case) or Example \ref{exmp:P_2} (discrete case). Then we have $\E[\tau_*] = O(\sqrt{\log(p)/n})$ by (\ref{eqn:tau_*-bound_subgaussian}) and $\tau_\Delta = C a_n^2 \sqrt{\log(p)/n}$. Thus $\tau_* = o_\Prob(\tau_\Delta)$ for $a_n \to \infty$.

Second, $\tau_\Delta$ is non-adaptive to the observations $X_1^n$ since the minimax lower bound is based on the worst case analysis and the matching upper bound is obtained by the Bonferroni inequality which ignores the dependence structures in $F$. %since the constant $C(\nu) > 0$ depends on the underlying distribution $F$ through $\nu^2$ and  in view of (\ref{eqn:tau_*-bound_subgaussian}). 
On the contrary, $\tau_*$ takes into account the dependence information of $F$ by conditioning on the observations. Therefore, the bootstrap threshold may better adjust to the dependence structure for some designs of $X_i$. 

\begin{exmp}[A block diagonal covariance matrix example with reduced rank]
\label{exmp:thresholded_estimator_reduced_rank}
Let $L,m$ be two positive integers and $p = L m$. Let $Z_{il},i=1,\cdots,n; l=1,\cdots,L,$ be iid mean zero subgaussian$(\nu_n^2)$ random variables with unit variance and $Y_{il} = \vone_m Z_{il}$, where $\vone_m$ is the $m \times 1$ vector containing all ones. Let $X_i = (Y_{i1}^\top,\cdots,Y_{iL}^\top)^\top$. Under the assumptions in Theorem \ref{thm:thresholded_cov_mat_rate_adaptive}, we can show that $\E[\tau_*] \le C' \beta^{-1} \xi_4^2 (\log(L)/n)^{1/2}$. If $\log{L} = o(\log{p})$, then $\tau_* = o_\Prob(\tau_\Delta)$ and $\hat\Sigma(\tau_*)$ can gain much tighter performance bounds in (\ref{eqn:thresholded_cov_mat_rate_spectral_adaptive}) and (\ref{eqn:thresholded_cov_mat_rate_F_adaptive}) than $\hat\Sigma(\tau_\Delta)$. Note that the covariance matrix $\Sigma = \Cov(X_i)$ in this example is block diagonal such that the diagonal blocks of $\Sigma$ are rank-one matrices $\vone_m \vone_m^\top$. Therefore, $\Sigma$ has the simultaneous sparsity (i.e. $\zeta_p = m$) and reduced rank (i.e. $\text{rank}(\Sigma) = L$). 
\end{exmp}

Third, as we shall demonstrate in Theorem \ref{thm:thresholded_cov_mat_rate_adaptive_polymom}, the Gaussian type convergence rate of the bootstrap method in Theorem \ref{thm:thresholded_cov_mat_rate_adaptive} can be achieved even for heavy-tailed data with polynomial moments.

\begin{thm}[Data-driven threshold selection: uniform polynomial moment observations]
\label{thm:thresholded_cov_mat_rate_adaptive_polymom}
Let $X_i$ be iid mean zero random vectors such that $\|\max_{1 \le k \le p} |X_{1k}| \|_8 \le \nu_n$ and $\Sigma \in {\cal G}(r,\zeta_p)$. Suppose that there exist constants $C_i > 0, i=1,\cdots,3,$ such that $\Gamma_{(j,k),(j,k)} \ge C_1$, $\xi_6 \le C_2 \nu_n^{1/3}$ and $\xi_8 \le C_3 \nu_n^{1/2}$ for all $j,k=1,\cdots,p$. Let $\alpha, \beta, K \in (0, 1)$ and $\tau_* = \beta^{-1} a_{\bar{T}_n^\sharp}(1-\alpha)$. If $\nu_n^8 \log^7(np) \le C_4 n^{1-7K/6}$, then (\ref{eqn:thresholded_cov_mat_rate_spectral_adaptive}) and (\ref{eqn:thresholded_cov_mat_rate_F_adaptive}) hold with probability at least $1-\alpha-C n^{-K/6}$ for some constant $C > 0$ depending only on $C_1,\cdots,C_4$. In addition, (\ref{eqn:tau_*-bound_subgaussian}) holds for some constant $C' > 0$ depending only on $\alpha$ and $C_1,\cdots,C_4$.
\end{thm}

We compare Theorem \ref{thm:thresholded_cov_mat_rate_adaptive_polymom} with the threshold obtained by the union bound approach. Assume that $\max_{1 \le k \le p} \E|X_{1k}|^q < \infty$ for $q \ge 8$. By the Nagaev inequality \cite{nagaev1979a} applied to the split sample in Remark \ref{rmk:data-splitting-reduction}, one can show that
$$\tau_\sharp = C(q) \Big\{ {p^{4/q} \over n^{1-2/q}} \xi_q^2 + \Big( {\log{p} \over n}\Big)^{1/2} \xi_4^2 \Big\}$$
is the right threshold that gives a large probability bound for $\tau_\diamond = |\hat{S}_n-\Sigma|_\infty$. Consider $q = 8$, $\xi_8 = O(1)$, and the scaling limit $p = n^A$ for $A > 0$. Then the universal threshold $\tau_\sharp = o(1)$ if $0 < A < 3/2$. In contrast, since  $\|\max_{1 \le k \le p} |X_{1k}| \|_8 \le p^{1/8} \xi_8 = O(p^{1/8})$, it follows from Theorem \ref{thm:thresholded_cov_mat_rate_adaptive_polymom} that the bootstrap threshold $\tau_*$ is asymptotically valid if $0 < A < 1$ and by (\ref{eqn:tau_*-bound_subgaussian}), $\E[\tau_*] = O(\sqrt{(\log{p})/n})$. Therefore, in the least favorable case for the bootstrap, we conclude that: (i) if $A \in (0, 1/2]$, then $\E[\tau_*] \asymp \tau_\sharp$; (ii) if $A \in (1/2, 1)$, then $\E[\tau_*] = o(\tau_\sharp)$ and $\tau_\sharp = o(1)$; (iii) if $A \in [1, 3/2)$, then $\tau_\sharp = o(1)$ while the bootstrap threshold $\tau_*$ is not asymptotically valid; (iv) if $A \in [3/2, \infty)$, then neither $\hat\Sigma(\tau_*)$ or $\hat\Sigma(\tau_\sharp)$ is consistent for estimating $\Sigma$. Hence, the bootstrap method gives better convergence rate than the universal thresholding method under the spectral and Frobenius norms when $A \in  (1/2, 1)$. On the other hand, since $\tau_\sharp = o(1)$ when $A \in (0, 3/2)$, the cost of the bootstrap to achieve the Gaussian-like convergence rate $\tau_* = O_\Prob(\sqrt{(\log{p})/n})$ for the heavy-tailed distribution $F$ is a stronger requirement on the scaling limit for $A \in (0, 1)$. Moreover, to the best of our knowledge, the minimax lower bound is currently not available to justify $\tau_\sharp$ for observations with polynomial moments. Finally, we remark that bootstrap can adapt to the dependency structure in $F$. For Example \ref{exmp:thresholded_estimator_reduced_rank} with a block diagonal covariance matrix, we only need $L \log^7(n L) = o(n)$, where $L$ can be much smaller than $p$ and the dimension $p$ may still be allowed to be larger or even much larger than the sample size $n$.

%From Theorem \ref{thm:thresholded_cov_mat_rate_adaptive_polymom}, the subgaussian assumption on $F$ is not essential: for non-Gaussian data with heavier tails than subgaussian, the thresholded covariance matrix estimator with the threshold selected by the bootstrap approach again attains the Gaussian type convergence rate at the asymptotic confidence level $100(1-\alpha)\%$. In particular, the dimension $p$ may still be allowed to increase subexponentially fast in the sample size $n$.

\subsection{Simultaneous inference for covariance and rank correlation matrices}

Another related important problem of estimating the sparse covariance matrix $\Sigma$ is the consistent recovery of its support, i.e. non-zero off-diagonal entries in $\Sigma$ \cite{lamfan2009a}. Towards this end, a lower bound of the minimum signal strength ($\Sigma$-min condition) is a necessary condition to separate the weak signals and true zeros. Yet, the $\Sigma$-min condition is never verifiable. To avoid this undesirable condition, we can alternatively formulate the recovery problem as a more general hypothesis testing problem
\begin{equation}
\label{eqn:cov_mat_simultaneous_test_formulation}
H_0: \Sigma = \Sigma_0  \quad \text{versus} \quad  H_1: \Sigma \neq \Sigma_0,
\end{equation}
where $\Sigma_0$ is a known $p \times p$ matrix. In particular, if $\Sigma_0=\Id_{p \times p}$, then the support recovery can be re-stated as the following simultaneously testing problem: for all $m, k \in \{ 1,\cdots,p \}$ and $m \neq k$,
\begin{equation}
\label{eqn:cov_mat_simultaneous_test_formulation_2}
H_{0,mk}: \sigma_{mk} = 0 \quad \text{versus} \quad H_{1,mk}: \sigma_{mk} \neq 0.
\end{equation}
The test statistic we construct is $\bar{T}_0 = |\hat{S}_n-\Sigma_0|_{\infty,\text{off}}$, which is an $\ell^\infty$ statistic by taking the maximum magnitudes on the off-diagonal entries. Then $H_0$ is rejected if $\bar{T}_0 \ge a_{\bar{T}_n^\sharp}(1-\alpha)$.

\begin{cor}[Asymptotic size of the simultaneous test: subgaussian observations]
\label{cor:cov_mat_simultaneous_test}
Let $\nu_n \ge 1$ and $X_i$ be iid mean zero random vectors in $\R^p$ such that $X_{ik} \sim \text{subgaussian}(\nu_n^2)$ for all $k=1,\cdots,p$. Suppose that there exist constants $C_i > 0, i=1,\cdots,3,$ such that $\Gamma_{(j,k),(j,k)} \ge C_1$, $\xi_6 \le C_2 \nu_n^{1/3}$ and $\xi_8 \le C_3 \nu_n^{1/2}$ for all $j,k=1,\cdots,p$. Let $\alpha,\beta, K \in (0, 1)$ and $\tau_* = \beta^{-1} a_{\bar{T}_n^\sharp}(1-\alpha)$. If $\nu_n^4 \log^7(np) \le C_4 n^{1-K}$, then the above test based on $\bar{T}_0$ for (\ref{eqn:cov_mat_simultaneous_test_formulation}) has the size $\alpha + O(n^{-K/6})$; i.e. the family-wise error rate of the simultaneous test problem (\ref{eqn:cov_mat_simultaneous_test_formulation_2}) is asymptotically controlled at the level $\alpha$.
\end{cor}

From Corollary \ref{cor:cov_mat_simultaneous_test}, the test based on $\bar{T}_0$ is asymptotically exact of size $\alpha$ for subgaussian data. A similar result can be established for observations with polynomial moments. Due to the space limit, details are omitted. \cite{changzhouzhouwang2016} proposed a similar test statistic for comparing the two-sample large covariance matrices. Their results (Theorem 1 in \cite{changzhouzhouwang2016}) are analogous to Corollary \ref{cor:cov_mat_simultaneous_test} in this paper in that no structural assumptions in $\Sigma$ are needed in order to obtain the asymptotic validity of both tests. However, we shall note that their assumptions (C.1), (C.2), and (C.3) on the non-degeneracy are stronger than our condition $\Gamma_{(j,k),(j,k)} \ge C_1$. For subgaussian observations $X_{ik} \sim \text{subgaussian}(\nu_n^2)$, (C.3) in \cite{changzhouzhouwang2016} assumed that $\min_{1 \le j \le k \le p} \gamma_{jk} /\nu_n^4 \ge c$ for some constant $c > 0$, where $\gamma_{jk}=\Var(X_{1j}X_{1k})$. If $\nu_n^2 \to \infty$, then \cite[Theorem 1]{changzhouzhouwang2016} requires that $\gamma_{jk}$ for all $j,k=1,\cdots,p$ have to obey a uniform lower bound that diverges to infinity. For the covariance matrix kernel, since $g(x) = (x x^\top - \Sigma) / 2$, we only need that $\min_{j,k} \gamma_{jk} \ge c$ for some fixed lower bound.

Next, we comment that a distinguishing feature of our bootstrap test from the $\ell^2$ test statistic \cite{chenzhangzhong2010} is that no structural assumptions are made on $F$ and we allow for the strong dependence in $\Sigma$. For example, consider again the elliptic distributions \cite[Chapter 1]{muirhead1982} with the positive-definite $V=\varrho \vone_p \vone_p^\top + (1-\varrho) \Id_{p\times p}$ such that the covariance matrix $\Sigma$ is proportion to $V$. Then, we have 
\begin{eqnarray*}
\tr(V^4) &=& p[1+(p-1)\varrho^2]^2 + p(p-1)[2\varrho+(p-2)\varrho^2]^2, \\
\tr(V^2) &=& \varrho^2 p^2 + (1-\varrho^2) p.
\end{eqnarray*}
For any $\varrho\in(0,1)$, $\tr(V^4) / \tr^2(V^2) \to 1$ as $p \to \infty$. Therefore, the limiting distribution of the $\ell^2$ test statistic in \cite{chenzhangzhong2010} is no longer normal and its asymptotic distribution remains unclear.

Finally, the covariance matrix testing problem (\ref{eqn:cov_mat_simultaneous_test_formulation}) can be generalized further to nonparametric forms which can gain more robustness to outliers and the nonlinearity in the dependency structures. Let $U_\diamond = \E[h(X_1,X_2)]$ be the expectation of the random matrix associated with $h$ and $U_0$ be a known $p \times p$ matrix. Consider the testing problem
$$
H_0: U_\diamond = U_0  \quad \text{versus} \quad  H_1: U_\diamond \neq U_0.
$$
Then, the test statistic can be constructed as $\bar{T}_0 = |U_n-U_0|_\infty$ (or $\bar{T}'_0 = |U_n-U_0|_{\infty,\text{off}}$) and $H_0$ is rejected if $\bar{T}_0 \ge a_{\bar{T}_n^\sharp}(1-\alpha)$ (or $\bar{T}'_0 \ge a_{\bar{T}_n^\sharp}(1-\alpha)$), where the bootstrap samples are generated w.r.t. the kernel $h$. The above test covers Kendall's tau rank correlation matrix as a special case where $h$ is the bounded kernel.

\begin{cor}[Asymptotic size of the simultaneous test for Kendall's tau rank correlation matrix]
\label{cor:kendall_tau_mat_simultaneous_test}
Let $X_i$ be iid random vectors with a distribution $F$ in $\R^p$. Suppose that there exists a constant $C_1 > 0$ such that $\Gamma_{(j,k),(j,k)} \ge C_1$ for all $j,k=1,\cdots,p$. Let $\alpha,\beta, K \in (0, 1)$ and $\tau_* = \beta^{-1} a_{\bar{T}_n^\sharp}(1-\alpha)$, where the bootstrap samples are generated with Kendall's tau rank correlation coefficient matrix kernel. If $\log^7(np) \le C_2 n^{1-K}$, then the test based on $\bar{T}'_0$ has the size $\alpha + O(n^{-K/6})$.
\end{cor}

Therefore, the asymptotic validity of the bootstrap test for large Kendall's tau rank correlation matrix is obtained when $\log{p} = o(n^{1/7})$ without imposing structural and moment assumptions on $F$.

\section{Proof of the main results}
\label{sec:proof}

The rest of the paper is organized as follows. In Section \ref{app:concentration_ineq_canonical-Ustat}, we first present a useful inequality for bounding the expectation of the sup-norm of the {\it canonical} U-statistics and then compare with an alternative simple data splitting bound by reducing to the moment bounding exercise for the sup-norm of sums of iid random vectors. We shall discuss several advantages of using the U-statistics approach by exploring the degeneracy structure. Section \ref{subsec:proof-of-GA} contains the proof of the Gaussian approximation result and Section \ref{subsec:proof-of-bootstraps} proves the convergence rate of the bootstrap validity. Proofs of the statistical applications are given in Section \ref{subsec:proof-of-applications}. Additional proofs and technical lemmas are given in the SM.

\subsection{A maximal inequality for canonical U-statistics}
\label{app:concentration_ineq_canonical-Ustat}

Before proving our main results, we first establish a maximal inequality of the canonical U-statistics of order two. The derived expectation bound is useful in controlling the size of the nonlinear and completely degenerate error term in the Gaussian approximation.

\begin{thm}[A maximal inequality for canonical U-statistics]
\label{thm:expectation-bound}
Let $X_1^n$ be a sample of iid random variables in a separable and measurable space $(S, {\cal S})$. Let $f: S \times S \to \mathbb{R}^d$ be an ${\cal S} \otimes {\cal S}$-measurable, symmetric and canonical kernel such that $\E|f_m(X_1,X_2)|<\infty$ for all $m=1,\cdots,d$. Let $V_n = [n(n-1)]^{-1} \sum_{1\le i \neq j \le n} f(X_i,X_j)$, $M = \max_{1\le i \neq j \le n} \max_{1 \le m \le d} |f_m(X_i,X_j)|$, $D_q = \max_{1 \le m \le d} (\E|f_m(X_1,X_2)|^q)^{1/q}$ for $q > 0$. If $2 \le d \le  \exp(b n)$ for some constant $b>0$, then there exists an absolute constant $K > 0$ such that 
\begin{equation}
\label{eqn:expectation-bound-canonical}
\E [|V_n|_\infty] \le K (1 + b^{1/2}) \Big\{ \Big( {\log{d} \over n} \Big)^{3/2} \|M\|_4 + {\log{d} \over n} D_2 + \Big( {\log{d} \over n} \Big)^{5/4}  D_4 \Big\}.
\end{equation}
\end{thm}

Note that Theorem \ref{thm:expectation-bound} is non-asymptotic. As immediate consequences of Theorem \ref{thm:expectation-bound}, we can derive the rate of convergence of $\E[|V_n|_\infty]$ with kernels under the subexponential and uniform polynomial moment conditions.

\begin{cor}[Kernels with subexponential and uniform polynomial moments]
\label{cor:expectation-subexponential-polynomial-kernel}
Let $B_n, B'_n$ be two sequences of positive reals and $f$ be a symmetric and canonical kernel. Suppose that $2 \le d \le  \exp(b n)$ for some constant $b>0$. (i)
If
\begin{equation}
\label{eqn:subexponential-moment}
\max_{1 \le m \le d}  \E \left[ \exp(|f_m| / B_n) \right] \le 2,
\end{equation}
then there exists a constant $C(b)>0$ such that
\begin{equation}
\label{eqn:expectation-subexponential-kernel}
\E [|V|_\infty] \le C(b) B_n  \{ (n^{-1} \log{d})^{3/2} \log(nd)   + n^{-1} \log{d} \}.
\end{equation}
(ii) Let $q \ge 4$. If \begin{equation}
\label{eqn:uniform-polynomial-moment}
\E (\max_{1 \le m \le d} |f_m| / B_n)^q \vee \max_{1 \le m \le d} \E (|f_m| / B'_n)^4 \le 1,
\end{equation}
then there exists a constant $C(b)>0$ such that
\begin{equation}
\label{eqn:expectation-polynomial-kernel}
\E [|V|_\infty] \le C(b) \{ B_n n^{-3/2+2/q} (\log{d})^{3/2} + B'_n n^{-1} \log{d}  \}.
\end{equation}
\end{cor}

\begin{rmk}[Comparison of Theorem \ref{thm:expectation-bound} with sums of iid random vectors by data splitting]
\label{rmk:data-splitting-reduction}
We can also bound the expected norm of a $U$-statistic by the expected norm of sums of iid random vectors. Assume that $\E |f_k(X_1,X_2)| < \infty$ for all $k=1,\cdots,d$ and let $m = [n/2]$ be the largest integer no greater than $n/2$. As noted in \cite{hoeffding1963}, we can write
\begin{equation}
\label{eqn:hoeffding_average}
m(V_n - \E V_n) = {1 \over n!} \sum_{\text{all } \pi_n} S(X_{\pi_n(1)}, \cdots, X_{\pi_n(n)}),
\end{equation}
where $S(X_1^n) = \sum_{i=1}^{m} [f(X_{2i-1},X_{2i}) - \E f]$ and the summation $\sum_{\text{all } \pi_n}$ is taken over all possible permutations $\pi_n : \{1,\cdots,n\} \to \{1,\cdots,n\}$. By Jensen's inequality and the iid assumption of $X_i$, we have
\begin{equation}
\label{eqn:reducing-Ustat-to-iid-sum}
\E |V_n - \E V_n |_\infty \le {1 \over m} \E \Big |\sum_{i=1}^{m} [f(X_i,X_{i+m}) - \E f ] \Big|_\infty,
\end{equation}
which can be viewed as a data splitting method into two halves. Assuming (\ref{eqn:subexponential-moment}), it follows from Bernstein's inequality \cite[Proposition 5.16]{vershynin2010a} that
\begin{equation}
\label{eqn:reducing-Ustat-to-iid-sum-bernstein-bound}
\E |V_n - \E V_n |_\infty \le K_1 B_n ( \sqrt{\log(d)/n} + \log(d)/n )
\end{equation}
for some absolute constant $K_1 > 0$. So if $\log{d} \le b n^{1-\varepsilon}$ for some $\varepsilon \in (0,1)$, then $\E |V_n - \E V_n |_\infty \le C(b) B_n (\log(d)/n)^{1/2} \le C(b) B_n n^{-\varepsilon/2}$. For the canonical kernel where $\E V_n = 0$, there are two advantages of using the U-statistics approach in Theorem \ref{thm:expectation-bound} over the data splitting method into iid summands (\ref{eqn:reducing-Ustat-to-iid-sum}) and (\ref{eqn:reducing-Ustat-to-iid-sum-bernstein-bound}). 

First, we can obtain from (\ref{eqn:expectation-subexponential-kernel}) that $\E |V_n|_\infty \le C(b) B_n \{ n^{1-5\varepsilon/2} + n^{-\varepsilon} \}$. Therefore, sharper rate is obtained by (\ref{eqn:expectation-subexponential-kernel}) when $\varepsilon \in (1/2,1)$ which covers the regime of valid Gaussian approximation and bootstrap. Under the scaling limit for the Gaussian approximation validity, i.e. $B_n^a \log^7(np) / n \le C n^{-K_2}$ for some $K_2 \in (0,1)$, where $a=2$ for the subexponential moment kernel and $a=4$ for the uniform polynomial moment kernel, it is easy to see that $\log{d} \le \log(n d) \le C n^{(1-K_2)/7}$ so we can take $\varepsilon = (6+K_2)/7$.

Second and more importantly, the rate of convergence obtained by the Bernstein bound (\ref{eqn:reducing-Ustat-to-iid-sum-bernstein-bound}) does not lead to a convergence rate for the Gaussian and bootstrap approximations. The reason is that, although (\ref{eqn:reducing-Ustat-to-iid-sum-bernstein-bound}) is rate-exact for {\it non-degenerate} U-statistics, where the dependence of the rate in (\ref{eqn:reducing-Ustat-to-iid-sum-bernstein-bound}) on the sample size is $O(B_n n^{-1/2})$, it is not strong enough to control the size of the nonlinear remainder term $\E[|n^{1/2} V_n |_\infty]$ when $d \to \infty$ (recall that $R_n = n^{1/2} V_n / 2$); c.f. Proposition \ref{prop:general-GA-rate}. On the contrary, our bound in Theorem \ref{thm:expectation-bound} exploits the degeneracy structure of $V$ and the dependence of the rate in (\ref{eqn:expectation-bound-canonical}) on the sample size is $O(B_n n^{-1} + \|M\|_4 n^{-3/2})$. Therefore, Theorem \ref{thm:expectation-bound} is mathematically more appealing in the degenerate case.

For non-degenerate U-statistics $U_n = [n (n-1)]^{-1} \sum_{1 \le i \neq j \le n} h(X_i, X_j)$, the reduction to sums of iid random vectors in (\ref{eqn:reducing-Ustat-to-iid-sum}) does not give tight asymptotic distributions. To illustrate this point, we consider the case $d=1$ and let $X_i$ be iid mean zero random variables with variance $\sigma^2$. Let $\zeta_1^2 = \Var(g(X_1))$ and $\zeta_2^2 = \Var(h(X_1, X_2))$. Assume that $\zeta_1^2 > 0$. So $\zeta_1^2$ is the variance of the leading projection term used in the Gaussian approximation and by Jensen's inequality $\zeta_1^2 \le \zeta_2^2$. Note that $\sqrt{n}(U_n-\E U_n) \stackrel{D}{\to} N(0, 4 \zeta_1^2)$ \cite[Theorem A, page 192]{serfling1980} and by the CLT $\sqrt{2/m} \sum_{i=1}^m [f(X_i,X_{i+m}) - \E f] \stackrel{D}{\to} N(0, 2 \zeta_2^2)$. Since in general $\zeta_2^2 \neq 2 \zeta_1^2$, the limiting distribution of the U-statistic is not the same as that in the data splitting method. For example, consider the non-degenerate covariance kernel $h(x_1,x_2)=(x_1-x_2)^2 / 2$. Denote $\mu_4 = \E X_1^4$ and $g(x_1)=(x_1^2-\sigma^2)/2$. Then, $\zeta_2^2 = (\mu_4 + \sigma^4)/2$ and $\zeta_1^2 = (\mu_4-\sigma^4) / 4$ so that $\zeta_2^2 > 2 \zeta_1^2$ when $\sigma^2 > 0$. In particular, if $X_i$ are iid $N(0,\sigma^2)$, then $\mu_4=3\sigma^4$, $4 \zeta_1^2 = 2\sigma^4$, and $2 \zeta_2^2= 4\sigma^4$. Therefore, even though (\ref{eqn:reducing-Ustat-to-iid-sum-bernstein-bound}) gives better rate in the non-degenerate case, the reduction by splitting the data into the iid summands is not optimal for the Gaussian approximation purpose, which is the main motivation of this paper. In fact, $\zeta_2^2$ serves no purpose in the limiting distribution of $\sqrt{n}(U_n - \E U_n)$.
\qed
\end{rmk}

\subsection{Proof of results in Section \ref{sec:gaussian-approx}}
\label{subsec:proof-of-GA}

For $q > 0$ and $\phi \ge 1$, we define 
\begin{eqnarray*}
D_{g,q} &=& \max_{1\le m \le d} \E|g_m(X)|^q, \\
M_{g,q}(\phi) &=& \E \left[ \max_{1 \le m \le d} |g_m(X)|^q \vone\left( \max_{1 \le m \le d} |g_m(X)| > {\sqrt{n} \over 4 \phi \log{d} } \right) \right], \\
M_{Y,q}(\phi) &=& \E \left[ \max_{1 \le m \le d} |Y_m|^q \vone\left( \max_{1 \le m \le d} |Y_m| > {\sqrt{n} \over 4 \phi \log{d} } \right) \right].
\end{eqnarray*}
and $M_q(\phi) = M_{g,q}(\phi) + M_{Y,q}(\phi)$. The Gaussian approximation result (GAR) in Proposition \ref{prop:general-GA-rate} below relies on the control of $D_{g,3}$ and $M_3(\phi)$. Interestingly, the quantity $M_{g,3}(\phi)$ can be viewed as a stronger version of the Lindeberg condition that allows us to derive the explicit convergence rate of the Gaussian approximation when $d \to \infty$. Denote $\chi_{\tau,ij} = \vone(\max_{1 \le m \le d} |h_m(X_i,X_j)| > \tau)$ for $\tau \ge 0$. Let
\begin{eqnarray*}
D_q &=& \max_{1\le m \le d} \E|h_m(X_1,X_2)|^q, \\
M_{h,q}(\tau) &=& \E \left[ \max_{1\le i \neq j \le n} \max_{1 \le m \le d} |h_m(X_i,X_j)|^q \chi_{\tau,ij} \right].
\end{eqnarray*}
For two random vectors $X$ and $Y$ in $\R^d$, we denote 
$$\tilde\rho^{re}(X,Y) = \sup_{y \in \R^d}  \left|\Prob(X \le y) -  \Prob(Y \le y) \right|.$$

\begin{prop}[A general Gaussian approximation result for U-statistics]
\label{prop:general-GA-rate}
Assume that (M.1) holds and $\log d \le \bar{b} n$ for some constant $\bar{b} > 0$. Then there exist constants $C_i := C_i(\ub, \bar{b}) > 0, i=1,2$ such that for any real sequence $\bar{D}_{g,3}$ satisfying $D_{g,3} \le \bar{D}_{g,3}$, we have
\begin{eqnarray}
\nonumber
\tilde\rho^{re}(T_n, Y) &\le& C_1 \Bigg\{ \left({\bar{D}_{g,3}^2 \log^7 d \over n}\right)^{1/6} + {M_3(\phi_n) \over \bar{D}_{g,3}} \\ \label{eqn:general-GA-rate}
&& \qquad + \phi_n \left( {\log^{3/2}{d} \over n} (M_{h,4}(\tau)^{1/4} + \tau) + {\log{d} \over n^{1/2}} D_2^{1/2} + {\log^{5/4}{d} \over n^{3/4}} D_4^{1/4} \right) \Bigg\},
\end{eqnarray}
where
\begin{equation}
\label{eqn:phi_n}
\phi_n = C_2 \left({\bar{D}_{g,3}^2 \log^4 d \over n}\right)^{-1/6}.
\end{equation}
In addition, $\rho^{re}(T_n, Y)$ obeys the same bound in (\ref{eqn:general-GA-rate}).
\end{prop}

With the help of Proposition \ref{prop:general-GA-rate}, we are now ready to prove Theorem \ref{thm:CLT-hyperrec-momconds}. 

\begin{proof}[Proof of Theorem \ref{thm:CLT-hyperrec-momconds}]
We may assume that $\varpi_{1,n} \le 1$; otherwise the proof is trivial. Let $\ell_n = \log(nd) > 1$. By (M.2) and Jensen's inequality, we have $D_2 \le B_n^{2/3}$, $D_{g,3} \le D_3 \le B_n$, and $D_4 \le B_n^2$. Write $M_{h,q} = M_{h,q}(0)$. By Proposition \ref{prop:general-GA-rate} with $\tau=0$ and $\phi_n$ is given by (\ref{eqn:phi_n}), we have
\begin{eqnarray}
\nonumber
\rho^{re}(T_n, Y) &\le& C_1 \Bigg\{ \left({\bar{D}_{g,3}^2 \log^7 d \over n}\right)^{1/6} + {M_3(\phi_n) \over \bar{D}_{g,3}} \\ \label{eqn:general-GA-rate_tau=0}
&& \qquad + \phi_n \left( {\log^{3/2}{d} \over n} M_{h,4}^{1/4} + {\log{d} \over n^{1/2}} D_2^{1/2} + {\log^{5/4}{d} \over n^{3/4}} D_4^{1/4} \right) \Bigg\},
\end{eqnarray}
where $C_1 > 0$ is a constant only depending on $\ub$ and $\bar{b}$.

{\bf Case (E.1).} By \cite[Lemma 2.2.2]{vandervaartwellner1996}, $M_{h,4}^{1/4} \le K_1 B_n \ell_n$. Choosing $\bar{D}_{g,3} = B_n$, we have
\begin{eqnarray*}
\phi_n {\log^{3/2}{d} \over n} M_{h,4}^{1/4} &\le& C_2 {B_n^{2/3} \ell_n^{11/6} \over n^{5/6}} \le C_2 \varpi_{1,n}, \\
\phi_n {\log{d} \over n^{1/2}} D_2^{1/2} &\le& C_3 {(\log d)^{1/3} \over n^{1/3}} \le C_3 \varpi_{1,n}, \\
\phi_n {\log^{5/4}{d} \over n^{3/4}} D_4^{1/4} &\le& C_4 {B_n^{1/6} (\log d)^{7/12} \over n^{7/12}} \le C_4 \varpi_{1,n}.
\end{eqnarray*}
Following the proof of \cite[Proposition 2.1]{cck2015a}, we can show that
$$
\left({\bar{D}_{g,3}^2 \log^7 d \over n}\right)^{1/6} + {M_3(\phi_n) \over \bar{D}_{g,3}} \le C_5 \varpi_{1,n}.
$$
Then, (\ref{eqn:CLT-hyperrec-subexp}) follows from (\ref{eqn:general-GA-rate_tau=0}). Here, all constants $C_i$ for $i=2,\cdots,5$ depend only on $\ub$ and $\bar{b}$. \\

{\bf Case (E.2).} $D_2$ and $D_4$ obey the same bounds as case (E.1). Assuming (E.2), $M_{h,4}^{1/4} \le n^{1/2} B_n$. Choosing $\bar{D}_{g,3} = B_n + B_n^2 n^{-1/2+2/q} (\log d)^{-1/2}$, we have 
$$
\phi_n {\log^{3/2}{d} \over n} M_{h,4}^{1/4} \le C_6 {B_n^{2/3} \ell_n^{5/6} \over n^{1/3}} \le C_6 \varpi_{1,n}.
$$
Following the proof of \cite[Proposition 2.1]{cck2015a}, we can show that
$$
\left({\bar{D}_{g,3}^2 \log^7 d \over n}\right)^{1/6} + {M_3(\phi_n) \over \bar{D}_{g,3}} \le C_7 \left\{ \varpi_{1,n} + \varpi_{2,n} \right\}.
$$
Here, $C_6, C_7$ are constants depending only on $\ub, \bar{b}$, and $q$. Then, (\ref{eqn:CLT-hyperrec-subexp}) is immediate.
\end{proof}

\subsection{Proof of results in Section \ref{sec:bootstraps}}
\label{subsec:proof-of-bootstraps}

In view of the approximation diagram (\ref{eqn:approx-diagram}), our first task is to control the random quantity
$$
\sup_{A \in {\cal A}^{re}} \left| \Prob(Y \in A) - \Prob(Z^X \in A \mid X_1^n) \right|
$$
on an event occurring with large probability, which is Step (2) in the approximation diagram (\ref{eqn:approx-diagram}). 

\begin{prop}[Gaussian comparison bound for the linear part in U-statistic and its EB version]
\label{prop:gaussian_comp_leading_term-EB}
Let $Z^X | X_1^n \sim N(0, \hat\Gamma_n)$, where $\hat\Gamma_n$ is defined in (\ref{eqn:emprical_cov_mat_eb}). Suppose that (M.1), (M.2) and (M.2') are satisfied. \\
(i) If (E.1) and (E.1') hold, then there exists a constant $C(\ub)>0$ such that with probability at least $1-\gamma$ we have 
\begin{equation}
\label{eqn:prop:gaussian_comp_leading_term-EB_subexp}
\sup_{A \in {\cal A}^{re}} \left| \Prob(Y \in A) - \Prob(Z^X \in A \mid X_1^n) \right| \le C(\ub) \varpi^{B}_{1,n}(\gamma).
\end{equation}
(ii) If (E.2) and (E.2') hold with $q \ge 4$, then there exists a constant $C(\ub,q)>0$ such that with probability at least $1-\gamma$ we have 
\begin{equation}
\label{eqn:prop:gaussian_comp_leading_term-EB_unifpoly}
\sup_{A \in {\cal A}^{re}} \left| \Prob(Y \in A) - \Prob(Z^X \in A \mid X_1^n) \right| \le C(\ub,q) \{ \varpi^{B}_{1,n}(\gamma) + \varpi^{B}_{2,n}(\gamma) \}.
\end{equation}
\end{prop}

From Proposition \ref{prop:gaussian_comp_leading_term-EB}, we are now ready to establish the rate of convergence of the empirical bootstrap for U-statistics. Let $W_{jk} =  |n^{-1} \sum_{i=1}^n h_j(X_k,X_i) - V_{nj}|$ for $j=1,\cdots,d$ and $k=1,\cdots,n$. For $q,\tau>0$, and $\phi \ge 1$, we define
\begin{eqnarray*}
\hat{D}_{g,q} &=& \max_{1 \le j \le d} n^{-1} \sum_{k=1}^n W_{jk}^q, \\
\hat{D}_q &=& \max_{1 \le j \le d} n^{-2} \sum_{k,l=1}^n |h_j(X_k,X_l)|^q, \\
\hat{M}_{h,q}(\tau) &=& n^{-2} \sum_{i,k=1}^n \max_{1 \le j \le d} |h_j(X_i,X_k)|^q \vone(\max_{1 \le j \le d} |h_j(X_i,X_k)| > \tau),\\
\hat{M}_{g,q}(\phi) &=& n^{-1} \sum_{k=1}^n \max_{1 \le j \le d} W_{jk}^q \vone \left(\max_{1 \le j \le d} W_{jk} > { \sqrt{n} \over 4\phi\log d} \right), \\
\hat{M}_{Z,q}(\phi) &=& \E\left[ \max_{1\le j \le d} |Z^X_j|^q \vone\left( \max_{1 \le j \le d} |Z^X_j| > { \sqrt{n} \over 4\phi\log d} \right) \mid X_1^n \right],
\end{eqnarray*}
$\hat{M}_q(\phi) =\hat{M}_{g,q}(\phi) + \hat{M}_{Z,q}(\phi)$, and $Z^X \mid X_1^n \sim N(0,\hat\Gamma_n)$.

\begin{proof}[Proof of Theorem \ref{thm:eb-bootstrap-rate}]
In this proof, the constants $C_1,C_2,\cdots$ depend only on $\ub,\bar{b},K$ in case (i) and $\ub,\bar{b},q,K$ in case (ii). First, we may assume that 
\begin{equation}
\label{eqn:eb_reduction1}
n^{-1} B_n^2 \log^7(nd) \le c_1 \le 1
\end{equation}
for some sufficiently small constant $c_1 > 0$, where $c_1$ depends only on $\ub,\bar{b},K$ in case (i) and on $\ub,\bar{b},q,K$ in case (ii), since otherwise the proof is trivial by setting the constants $C(\ub,\bar{b},K)$ in (i) and $C(\ub,\bar{b},q,K)$ in (ii) large enough. By (\ref{eqn:approx-diagram}) and the triangle inequality,
\begin{eqnarray}
\nonumber
\rho^{B}(T_n, T^*_n) &\le& \rho^{re}(T_n, Y) + \sup_{A \in {\cal A}^{re}} |\Prob(Y \in A) - \Prob(Z^X \in A \mid X_1^n)| \\
\label{eqn:eb-rate-decomp}
&& \qquad + \rho^{re}(Z^X, T_n^* \mid X_1^n),
\end{eqnarray}
where $\rho^{re}(Z^X, T_n^* \mid X_1^n) = \sup_{A \in {\cal A}^{re}} |\Prob(Z^X \in A \mid X_1^n) - \Prob(T_n^* \in A \mid X_1^n)|$. Since $\log(1/\gamma) \le K \log(dn)$, we have $\varpi^{B}_{1,n}(\gamma) \le K^{1/3} \varpi_{1,n}$ and $\varpi_{2,n} \le \varpi^{B}_{2,n}(\gamma)$ for $\gamma \in (0,e^{-1})$. By Theorem \ref{thm:CLT-hyperrec-momconds} and Proposition \ref{prop:gaussian_comp_leading_term-EB}, we have: (i) if (E.1) and (E.1') hold, then with probability at least $1-2\gamma/9$ we have
$$
\rho^{re}(T_n, Y) + \sup_{A \in {\cal A}^{re}} |\Prob(Y \in A) - \Prob(Z^X \in A \mid X_1^n)| \le C(\ub,\bar{b},K) \varpi_{1,n};
$$
(ii) if (E.2) and (E.2') hold, then with probability at least $1-2\gamma/9$ we have
$$
\rho^{re}(T_n, Y) + \sup_{A \in {\cal A}^{re}} |\Prob(Y \in A) - \Prob(Z^X \in A \mid X_1^n)| \le C(\ub,\bar{b},q,K) \{ \varpi_{1,n} + \varpi^{B}_{2,n}(\gamma) \}.
$$
To deal with the third term on the right-hand side of (\ref{eqn:eb-rate-decomp}), we observe that conditionally on $X_1^n$, $U_n^*$ is a U-statistics of $\xi_1,\cdots,\xi_n$ and $Z^X$ has the conditional covariance matrix $\hat\Gamma_n$; c.f. (\ref{eqn:emprical_cov_mat_eb}). So we can apply Proposition \ref{prop:general-GA-rate} conditionally.

{\bf Case (i).} As in the proof of Proposition \ref{prop:gaussian_comp_leading_term-EB}, we have with probability at least $1-\gamma/9$
\begin{equation}
\label{eqn:tail_prob_bound_hatGamma_n}
|\hat\Gamma_n-\Gamma|_\infty \le C_1 [n^{-1} B_n^2 \log(nd) \log^2(1/\gamma)]^{1/2}.
\end{equation}
By (\ref{eqn:eb_reduction1}), (M.1), (M.2) and (M.2'), there exists a constant $C_2 > 0$ such that $\ub/2 \le \hat\Gamma_{n,jj} \le C_2 B_n^{2/3} \le C_2 B_n$ for all $j=1,\cdots,d$ holds with probability at least $1-\gamma/9$. Let $\bar{D}_{g,3} = C_3 B_n$, $\bar{D}_2 = C_4 B_n^{2/3}$, and $\bar{D}_4 = C_5 B_n^2 \log(dn)$. By Lemma \ref{lem:bounds_on_terms_empirical-version-EB}, each of the three events $\{\hat{D}_{g,3} \ge \bar{D}_{g,3}\}$, $\{\hat{D}_2 \ge \bar{D}_2\}$, and $\{\hat{D}_4 \ge \bar{D}_4\}$ occur with probability at most $\gamma/9$. Let $\phi_n = C_6 (n^{-1} \bar{D}_{g,3}^2 \log^4 d)^{-1/6}$ for some $C_6 > 0$ such that $\phi_n \ge 1$. By Jensen's inequality, $\max_{k,j} W_{kj} \le 2 n^{-1} \max_{k,j} \sum_{i=1}^n |h_j(X_k,X_i)|$. Then, by the union bound and the assumptions (E.1) and (E.1'), we have
\begin{eqnarray}
%\label{eqn:M_g3_eb}
\nonumber
\Prob(\hat{M}_{g,3}(\phi_n) > 0) &=& \Prob(\max_{1 \le j \le d,1 \le k \le n} W_{jk} > \sqrt{n} / (4\phi_n \log d))\\
\nonumber
&\le& (dn) \max_{1 \le j \le d,1 \le k \le n} \Prob\left({1 \over n} \sum_{i=1}^n |h_j(X_k,X_i)| > {\sqrt{n} \over 8 \phi_n \log d} \right) \\ 
\label{eqn:M_g3_hat_bound_subexp}
&\le& (2dn) \exp\left( -{\sqrt{n} \over 8 \phi_n (\log d) B_n} \right),
\end{eqnarray}
where the last step (\ref{eqn:M_g3_hat_bound_subexp}) follows from the triangle inequality on the Orlicz space with the $\psi_1$ norm $\|n^{-1} \sum_{i=1}^n | h_j(X_k,X_i) | \|_{\psi_1} \le n^{-1} \sum_{i=1}^n \| h_j(X_k,X_i) \|_{\psi_1} \le B_n$. Substituting the value of $\phi_n$ and using (\ref{eqn:eb_reduction1}), we have
\begin{equation}
\label{eqn:M_g3_hat_lower_bound_exponent_subexp}
{\sqrt{n} \over 8 \phi_n (\log d) B_n} \ge {C_3^{1/3} \log(nd) \over 8 C_6 c_1^{1/3}} \ge {C_3^{1/3} \over 16 C_6 c_1^{1/3}}  \big[\log(nd) +{1 \over K}\log(1/\gamma) \big].
\end{equation}
Therefore $\Prob(\hat{M}_{g,3}(\phi_n) > 0) \le \gamma/9$ by choosing $c_1 > 0$ small enough. Next, we deal with $\hat{M}_{Z,3}(\phi_n)$. Since conditional on $X_1^n$, $Z^X \sim N(0,\hat\Gamma_n)$. On the event $\{\hat\Gamma_{n,jj} \le C_2 B_n, \forall j=1,\cdots,d$\}, we have $\|Z_j^X\|_{\psi_2} \le \sqrt{8 C_2 B_n / 3} B_n^{1/2}$ and $\|Z_j^X\|_{\psi_1} \le C_7 B_n^{1/2}$ for all $j=1,\cdots,d$, where $C_7= \sqrt{8 C_2 /(3 \log2)}$. Integration-by-parts yields
$$
\hat{M}_{Z,3}(\phi_n) = \int_t^\infty \Prob(\max_j |Z_j^X| > u^{1/3} | X_1^n) du + t \Prob(\max_j |Z_j^X| > t^{1/3} | X_1^n),
$$
where $t = (\sqrt{n} / 4 \phi_n \log d)^3$. Since for any $u > 0$
$$
\Prob(\max_j |Z_j^X| > u^{1/3} | X_1^n) \le (2d) \exp(-u^{1/3} / (C_7 B_n^{1/2})),
$$
we have by elementary calculations that
\begin{eqnarray*}
\int_t^\infty \Prob(\max_j |Z_j^X| > u^{1/3} | X_1^n) du &\le& C_8 d t \left[\sum_{\ell=1}^3 (B_n^{1/2} t^{-1/3})^\ell \right] \exp\left(-{t^{1/3} \over C_7 B_n^{1/2}} \right).
\end{eqnarray*}
Since $B_n^{1/2} t^{-1/3} \log(nd) \le 4 C_6 C_3^{-1/3} [n^{-1} B_n^2 \log^4 (nd)]^{1/3} \le 4 C_6 C_3^{-1/3} c_1^{1/3}$, it follows from (\ref{eqn:M_g3_hat_bound_subexp}) and (\ref{eqn:M_g3_hat_lower_bound_exponent_subexp}) that
$$
\hat{M}_{Z,3}(\phi_n) \le C_9 d t \exp\left(-{t^{1/3} \over C_7 B_n^{1/2}} \right) \le C_9 d n^{3/2} \exp\left(-{\log(nd) \over 4 C_7 C_6 C_3^{-1/3} c_1^{1/3} } \right) \le C_9 n^{-1/2}
$$
for $c_1 > 0$ small enough. For the term $\hat{M}_{h,4}(\tau)$, we note that
$$
\Prob(\hat{M}_{h,4}(\tau) > 0) = \Prob(\max_{1 \le i,k \le n} \max_{1 \le j \le d} |h_j(X_i,X_k)| > \tau)
$$
and by (E.1) and (E.1') $\|h_j(X_j,X_k)\|_{\psi_1} \le B_n$. So we have
$$
\Prob(\hat{M}_{h,4}(\tau) > 0) \le  (2 d n^2) \exp(-\tau / B_n).
$$
Choose $\tau = C_{10} n^{1/2} / [\phi_n \log d]$. Then, by (\ref{eqn:M_g3_hat_bound_subexp}) and (\ref{eqn:M_g3_hat_lower_bound_exponent_subexp}), we have $\Prob(\hat{M}_{h,4}(\tau) > 0) \le \gamma / 9$. Now, by Proposition \ref{prop:general-GA-rate} conditional on $X_1^n$ with $M_{h,4}(\tau) \le n^2 \E[\max_{1 \le j \le d} |h_j(X,X')|^4 \vone(\max_{1 \le j \le d} |h_j(X,X')| > \tau)]$, we conclude that
\begin{eqnarray*}
\rho^{re}(Z^X, T_n^* \mid X_1^n) &\le& C_{11} \Big\{ \left({B_n^2 \log^7 d \over n}\right)^{1/6} + {1 \over n^{1/2} B_n} + \phi_n {\log d \over n^{1/2}} B_n^{1/3} \\
&& \qquad + \phi_n {\log^{5/4} d \over n^{3/4}} B_n^{1/2} \log^{1/4} (dn) + {\log^{1/2} d \over n^{1/2}}  \Big\} \\
&\le& C_{12} \varpi_{1,n}
\end{eqnarray*}
holds with probability at least $1-7\gamma/9$. So (\ref{eqn:eb-bootstrap-rate-subexp}) follows. \\

{\bf Case (ii).} In addition to (\ref{eqn:eb_reduction1}), we may assume that
\begin{equation}
\label{eqn:eb_reduction2}
{B_n^2 \log^3(nd) \over \gamma^{2/q} n^{1-2/q}} \le c_2 \le 1
\end{equation}
for some small enough constant $c_2 > 0$. As in Case (i),  there exists a constant $C_1 > 0$ such that $\ub/2 \le \hat\Gamma_{n,jj} \le C_1 B_n$ for all $j=1,\cdots,d$ holds with probability at least $1-\gamma/9$. Let $\bar{D}_{g,3} = C_2 [B_n + n^{-1+3/q} B_n^3 \gamma^{-3/q} (\log d)]$, $\bar{D}_2 = C_3 [B_n^{2/3} + n^{-1+2/q} B_n^2 \gamma^{-2/q} (\log d)] $, and $\bar{D}_4 = C_4  [B_n^2 + n^{-1+4/q} B_n^2 \gamma^{-4/q} (\log d)]$. By Lemma \ref{lem:bounds_on_terms_empirical-version-EB}, each of the three events $\{\hat{D}_{g,3} \ge \bar{D}_{g,3}\}$, $\{\hat{D}_2 \ge \bar{D}_2\}$, and $\{\hat{D}_4 \ge \bar{D}_4\}$ occur with probability at most $\gamma/9$. Note that 
\begin{eqnarray}
\nonumber
\phi_n &:=& C_5 (n^{-1} \bar{D}_{g,3}^2 \log^4 d)^{-1/6} \\
\label{eqn:phi_n_bound_poly}
&\le& C_5 C_2^{-1/3} \min \{ n^{1/6} B_n^{-1/3} \log^{-2/3} d, \; n^{1/2-1/q} B_n^{-1} \gamma^{1/q} \log^{-1}d \}.
\end{eqnarray}
By (\ref{eqn:M_g3_hat_bound_subexp}), the union bound, (\ref{eqn:phi_n_bound_poly}), and choosing $C_2$ large enough we have
\begin{eqnarray*}
\Prob(\hat{M}_{g,3}(\phi_n) > 0) &\le& n \max_{1 \le k \le n} \Prob\left({1 \over n} \sum_{i=1}^n \max_j |h_j(X_k,X_i)| > {\sqrt{n} \over 8 \phi_n \log d} \right) \\
&\le& {(8 B_n \phi_n \log d)^q \over n^{q/2-1}} \le {\gamma \over 9},
\end{eqnarray*}
where the second last step follows from (E.2), (E.2') and the triangle inequality $\|n^{-1} \sum_{i=1}^n \max_j |h_j(X_k,X_i)| \|_q \le n^{-1} \sum_{i=1}^n \|\max_j |h_j(X_k,X_i)| \|_q \le B_n$. Bound on the term $\hat{M}_{Z,3}(\phi_n)$ is the same as in Case (i). Choose $\tau = C_6 n^{1/2+1/q} / [\phi_n \log d]$ for some $C_6 > 0$. Then we have
$$
\Prob(\hat{M}_{h,4}(\tau) > 0) \le C_6^{-q} n^2 {(B_n \phi_n \log d)^q \over n^{q/2+1}} \le {\gamma \over 9}.
$$
Then, we have by elementary calculations that
\begin{eqnarray*}
\rho^{re}(Z^X, T_n^* \mid X_1^n) &\le& C_7 \Big\{ \left({\bar{D}_{g,3}^2 \log^7 d \over n}\right)^{1/6} + {1 \over n^{1/2} B_n} + \phi_n {\log d \over n^{1/2}} \bar{D}_2^{1/2} \\
&& \qquad + \phi_n {\log^{5/4} d \over n^{3/4}} \bar{D}_4^{1/4} + {\log^{1/2} d \over n^{1/2-1/q}}  \Big\} \\
&\le& C_8 \{ \varpi_{1,n} + \varpi_{2,n}^B(\gamma) \}
\end{eqnarray*}
with probability at least $1-7\gamma/9$. The proof is now complete.
\end{proof}

\begin{proof}[Proof of Corollary \ref{thm:eb-bootstrap-asymptotic_validity}]
Let $\gamma_n = [n \log^2(n)]^{-1}$. Then $\sum_{n=4}^\infty \gamma_n \le \int_3^\infty [x \log^2(x)]^{-1} dx = \log^{-1}(3) < \infty$. Applying Theorem \ref{thm:eb-bootstrap-rate} with $\gamma = \gamma_n$ and by the Borel-Cantelli lemma, we have $\Prob(\rho^B(T_n, T^*_n) > C \varpi_{1,n} \text{ i.o.}) = 0$ for part (i) and $\Prob(\rho^B(T_n, T^*_n) > C \{ \varpi_{1,n} + {\varpi'}_{2,n}^B(\gamma) \} \text{ i.o.}) = 0$ for part (ii), from which the Corollary follows.
\end{proof}

The proof of the validity of the randomly reweighted bootstrap with iid Gaussian weights (Section \ref{subsec:iid-weighted-bootstrap}) and Gaussian multiplier bootstrap with the jackknife covariance matrix estimator (Section \ref{subsec:gaussian-multiplier-bootstrap-jackknife}) can be found in Section \ref{app:proof-of-bootstraps} in the SM.

\subsection{Proof of results in Section \ref{sec:stat_apps}}
\label{subsec:proof-of-applications}

\begin{proof}[Proof of Theorem \ref{thm:thresholded_cov_mat_rate_adaptive}]
Let $\tau_\diamond = \beta^{-1} |\hat{S}_n-\Sigma|_\infty$. By the subgaussian assumption and Lemma \ref{lem:moment-bounds-gaussian-obs}, it is easy to verify that there is a large enough constant $C>0$ depending only on $C_2,C_3$ such that
\begin{equation}
\label{eqn:check_GA1}
\max_{\ell=1,2} \E[|h_{mk}|^{2+\ell} / (C \nu_n^{2\ell})] \vee \E[\exp(|h_{mk}| / \nu_n^2)] \le 2,
\end{equation}
where $h$ is the covariance matrix kernel. Since $\Gamma_{(j,k),(j,k)} \ge C_1$ for all $j,k=1,\cdots,p$ and $\nu_n^4 \log^7(np) \le C_4 n^{1-K}$, we have by Theorem \ref{thm:comparison_with_naive_gaussian_wild_bootstrap} that $ |\hat{S}_n-\Sigma|_\infty \le a_{\bar{T}_n^\sharp}(1-\alpha)$ with probability at least $1-\alpha-Cn^{-K/6}$, where $C > 0$ is a constant depending only on $C_i, i=1,\cdots,4$. Therefore, $\Prob(\tau_\diamond \le \tau_*) \ge 1- \alpha -C n^{-K/6}$ and the rest of the proof is restricted to the event $\{\tau_\diamond \le \tau_*\}$. By the decomposition,
\begin{eqnarray*}
\|\hat\Sigma(\tau_*)-\Sigma\|_2 &\le& \|\hat\Sigma(\tau_*)-T_{\tau_*}(\Sigma)\|_2 + \|T_{\tau_*}(\Sigma)-\Sigma\|_2 \\
&\le& I + II + III + \tau_*^{1-r} \zeta_p,
\end{eqnarray*}
where $T_\tau(\Sigma) = \{\sigma_{mk} \vone\{|\sigma_{mk}| > \tau\}\}_{m,k=1}^p$ is the resulting matrix of the thresholding operator on $\Sigma$ at the level $\tau$ and
\begin{eqnarray*}
I &=& \max_m \sum_k |\hat{s}_{mk}| \vone\{|\hat{s}_{mk}| > \tau_*, |\sigma_{mk}| \le \tau_*\}, \\
II &=& \max_m \sum_k |\sigma_{mk}| \vone\{|\hat{s}_{mk}| \le \tau_*, |\sigma_{mk}| > \tau_*\}, \\
III &=& \max_m \sum_k |\hat{s}_{mk}-\sigma_{mk}| \vone\{|\hat{s}_{mk}| > \tau_*, |\sigma_{mk}| > \tau_*\}.
\end{eqnarray*}
Note that on the event $\{\tau_\diamond \le \tau_*\}$, $\max_{m,k} |\hat{s}_{mk}-\sigma_{mk}| \le \beta \tau_*$. Since $\Sigma \in {\cal G}(r, \zeta_p)$, we can bound
$$
III \le (\beta \tau_*) (\tau_*^{-r} \zeta_p) = \beta \tau_*^{1-r} \zeta_p.
$$
By triangle inequality,
\begin{eqnarray*}
II &\le& \max_m \sum_k |\hat{s}_{mk}-\sigma_{mk}| \vone\{|\sigma_{mk}| > \tau_*\} + \max_m \sum_k |\hat{s}_{mk}| \vone\{|\hat{s}_{mk}| \le \tau_*, |\sigma_{mk}| > \tau_*\} \\
&\le& (\beta \tau_*) (\tau_*^{-r} \zeta_p) + \tau_* (\tau_*^{-r} \zeta_p) = (1+\beta) \tau_*^{1-r} \zeta_p.
\end{eqnarray*}
Let $\eta \in (0,1)$. We have $I \le IV + V + VI$, where
\begin{eqnarray*}
IV &=& \max_m \sum_k |\sigma_{mk}| \vone\{|\hat{s}_{mk}| > \tau_*, |\sigma_{mk}| \le \tau_*\}, \\
V &=&  \max_m \sum_k |\hat{s}_{mk}-\sigma_{mk}| \vone\{|\hat{s}_{mk}| > \tau_*,  |\sigma_{mk}| \le \eta \tau_*\},  \\
VI &=& \max_m \sum_k |\hat{s}_{mk}-\sigma_{mk}| \vone\{|\hat{s}_{mk}| > \tau_*,  \eta \tau_* < |\sigma_{mk}| \le \tau_*\} .
\end{eqnarray*}
Clearly, $IV \le \tau_*^{1-r} \zeta_p$. On the indicator event of $V$, we observe that 
$$
\beta \tau_* \ge |\hat{s}_{mk}-\sigma_{mk}| \ge |\hat{s}_{mk}|-|\sigma_{mk}| > (1-\eta) \tau_*.
$$
Therefore, $V=0$ if $\eta + \beta \le 1$. For $VI$, we have
$$
VI \le (\beta\tau_*) (\eta \tau_*)^{-r} \zeta_p.
$$
Collecting all terms, we conclude that
$$
\|\hat\Sigma(\tau_*)-\Sigma\|_2 \le (3+2\beta+\eta^{-r} \beta) \zeta_p \tau_*^{1-r} + V.  % = {3+2\beta+\eta^{-r} \beta \over \beta^{1-r}} \|\hat{S}-\Sigma\|^{1-r} s_p.
$$
Then (\ref{eqn:thresholded_cov_mat_rate_spectral_adaptive}) follows from the choice $\eta=1-\beta$. The Frobenius norm rate (\ref{eqn:thresholded_cov_mat_rate_F_adaptive}) can be established similarly. Details are omitted.

Next, we prove (\ref{eqn:tau_*-bound_subgaussian}). Let $\hat{g}_i = (n-1)^{-1} \sum_{j \neq i} h(X_i,X_j) - U_n$ and denote $\Phi(\cdot)$ as the cdf of the standard Gaussian random variable. By the union bound, we have for all $t > 0$
$$
\Prob_e \Big( {2 \over \sqrt{n}} | \sum_{i=1}^n \hat{g}_i e_i |_\infty \ge t \Big) \le 2 p^2 \Big[ 1-\Phi \Big( {t \over \bar\psi} \Big) \Big],
$$
where $\bar\psi = \max_{1 \le m,k \le p} |\psi_{mk}|$ and $\psi_{mk}^2 = 4 n^{-1} \sum_{i=1}^n \hat{g}_{i,mk}^2$. Let $\tilde\tau = n^{-1/2} \beta^{-1} \bar\psi \Phi^{-1}(1-\alpha/(2p^2))$; then $\tau_* \le \tilde\tau$. Since $\Phi^{-1}(1-\alpha/(2p^2)) \asymp (\log{p})^{1/2}$, we have $\E[\tau_*] \le C' \beta^{-1} \E[\bar\psi] (\log(p)/n)^{1/2}$, where $C' > 0$ is a constant only depending on $\alpha$. Now, we bound $\E[\bar\psi]$. By Jensen's inequality,
$$
\psi_{mk}^2 \le {16 \over n (n-1)} \sum_{1 \le i \neq j \le n} h_{mk}^2(X_i,X_j).
$$
Let $\ell = [n/2]$. By the data splitting argument in (\ref{eqn:reducing-Ustat-to-iid-sum}), \cite[Lemma 9]{cck2014b} and Jensen's inequality, we have
\begin{eqnarray*}
&& \E \Big\{ \max_{m,k} {1 \over n (n-1)} \Big| \sum_{1 \le i \neq j \le n} [h_{mk}^2(X_i,X_j) - \E h_{mk}^2] \Big| \Big\} \le {1 \over \ell} \E \Big\{ \max_{m,k} \Big| \sum_{i=1}^\ell [h_{mk}^2(X_i,X_{i+\ell}) - \E h_{mk}^2] \Big| \Big\} \\
&& \qquad \le {K_1 \over \ell} \Big\{ (\log{p})^{1/2} [\max_{m,k} \sum_{i=1}^\ell  \E h_{mk}^4(X_i,X_{i+\ell})]^{1/2} + (\log{p}) [\E \max_{m,k}\max_{1 \le i \le \ell} h_{mk}^4(X_i, X_{i+\ell})]^{1/2} \Big\}.
\end{eqnarray*}
By Pisier's inequality \cite[Lemma 2.2.2]{vandervaartwellner1996}, we have 
$$
\|\max_{m,k}\max_{i \le \ell} |h_{mk}(X_i, X_{i+\ell})| \|_4 \le K_2 \nu_n^2 \log(np).
$$
Hence it follows from (\ref{eqn:check_GA1}) that
$$
\E[\bar\psi^2] \le C \Big\{ \xi_4^4 +  \xi_8^4 \Big( {\log{p} \over n} \Big)^{1/2} + \nu_n^4 {\log^3(n p) \over n} \Big\}.
$$
Since $\nu_n^4 \log^7(np) \le C_4 n^{1-K}$ and $\xi_8 \le C_3 \nu_n^{1/2}$, we have $\E[\bar\psi] \le C \xi_4^2$, where $C > 0$ is constant depending only on $C_i,i =1,\cdots,4$. Then, we conclude that $\E[\tau_*] \le C(\alpha,C_1,\cdots,C_4) \beta^{-1} \xi_4^2 (\log(p)/n)^{1/2}$.
\end{proof}

%\begin{proof}[Proof of Corollary \ref{cor:thresholded_estimator_reduced_rank}]
%Let $\tilde{X}_i = (Z_{i1},\cdots,Z_{iL})^\top$ be iid random vectors in $\R^L$ and $\tilde{S}_n$ be the sample covariance matrix of $\tilde{X}_i$. Note that $\Cov(\tilde{X}_i) = \Id_L$ and $| \hat{S}_n - \Sigma |_\infty = | \tilde{S}_n - \Id_L|_\infty$. Then the Corollary follows from Theorem \ref{thm:mb-bootstrap-rate} applied to $\tilde{X}_1,\cdots,\tilde{X}_n$.
%\end{proof}

\section*{Acknowledgments}
The author would like to thank two anonymous referees, an Associate Editor, and the Co-Editor Tailen Hsing for their careful readings and many constructive comments that lead to significant improvements of this paper. The author is also grateful to Stephen Portnoy (UIUC), Xiaofeng Shao (UIUC), and Wei Biao Wu (University of Chicago) for their helpful discussions.

% AOS,AOAS: If there are supplements please fill:
\begin{supplement}[id=suppA]
  \sname{Supplement to}
  \stitle{``Gaussian and bootstrap approximations for high-dimensional U-statistics and their applications"}
  \slink[doi]{}
  %\sdatatype{.pdf}
  \sdescription{This supplemental file contains additional proofs, technical lemmas, and simulation results.}
\end{supplement}

\bibliographystyle{imsart-number}
\bibliography{bootstrap_ustat}

\newpage

\begin{center}
{\Large \bf Supplementary Material to ``Gaussian and bootstrap approximations for high-dimensional U-statistics and their applications"} \\

\vspace{0.1in}

Xiaohui Chen \\

\vspace{0.1in}

University of Illinois at Urbana-Champaign
\end{center}

\vspace{0.2in}

\appendix

\section{Proof details in Section 5.1}

\begin{lem}
\label{lem:decoupling_max_nonnegative}
Let $X_1,\cdots,X_n$ be a sample of independent random variables taking values in a measurable space $(S, {\cal S})$ and $h : S \times S \to \R^d$ be a symmetric and non-negative kernel such that $h(x_1,x_2) = h(x_2,x_1)$ and $h(x_1,x_2) \ge 0$ for all $x_1,x_2 \in S$. Let $X'_1,\cdots,X'_n$ be an independent copy of $X_1,\cdots,X_n$. Then, 
\begin{equation}
\label{eqn:decoupling_max_nonnegative}
\E[\max_{1 \le m \le d} \max_{1 \le i \neq j \le n} h_m(X_i, X'_j)] \le 4 \E[\max_{1 \le m \le d} \max_{1 \le i \neq j \le n} h_m(X_i, X_j)].
\end{equation}
\end{lem}

\begin{proof}[Proof of Lemma \ref{lem:decoupling_max_nonnegative}]
The proof is based on \cite[Theorem 3.1.1]{delaPenaGine1999} with necessary modifications to the maximum function. Let $\varepsilon_i, i = 1,\cdots,n$, be a sequence of iid Rademacher random variables such that $\Prob(\varepsilon_i = \pm1) = 1/2$ and $\varepsilon_i, i =1,\cdots,n,$ are also independent of $X_1,\cdots,X_n$ and $X'_1,\cdots,X'_n$. Define
$$
Z_i = \left\{ \begin{array}{cc}
X_i & \text{if } \varepsilon_i = +1 \\
X'_i & \text{if } \varepsilon_i = -1 \\
\end{array} \right. ,
\qquad 
Z'_i = \left\{ \begin{array}{cc}
X'_i & \text{if } \varepsilon_i = +1 \\
X_i & \text{if } \varepsilon_i = -1 \\
\end{array} \right. .
$$
Then, 
\begin{eqnarray*}
{\cal L}(Z_1,\cdots,Z_n,Z'_1,\cdots,Z'_n) &=& {\cal L}(X_1,\cdots,X_n,X'_1,\cdots,X'_n), \\
{\cal L}(Z_1,\cdots,Z_n) &=& {\cal L}(X_1,\cdots,X_n).
\end{eqnarray*}
Let ${\cal X} = \sigma(X_i,X'_i : i=1,\cdots,n)$ be the $\sigma$-filed generated by the $X$ and $X'$ random variables. Then, we have for $1 \le i \neq j \le n$
$$
\E[h(Z_i,Z_j) | {\cal X}] =[h(X_i,X_j) + h(X_i,X'_j) + h(X'_i,X_j) + h(X'_i,X'_j)] / 4.
$$
By the symmetry of $h$, for each $m=1,\cdots,d$ we have
$$
\max_{1 \le i \neq j \le n} h_m(X_i,X'_j) = \max_{1 \le i \neq j \le n} h_m(X'_i,X_j).
$$
Then, it follows from the non-negativity of $h$ that 
\begin{eqnarray*}
\max_{1 \le i \neq j \le n} h_m(X_i,X'_j) &=& \max_{1 \le i \neq j \le n} [ h_m(X_i,X'_j) \vee h_m(X'_i,X_j) ] \\
&\le& \max_{1 \le i \neq j \le n} [ h_m(X_i,X'_j) + h_m(X'_i,X_j) + h(X_i,X_j) + h(X'_i,X'_j)] \\
&=& 4 \max_{1 \le i \neq j \le n} \E[h_m(Z_i,Z_j) | {\cal X}].
\end{eqnarray*}
Therefore, we have
\begin{eqnarray*}
\E[\max_{1 \le m \le d} \max_{1 \le i \neq j \le n} h_m(X_i,X'_j)] &\le& 4 \E \{ \max_{1 \le m \le d} \max_{1 \le i \neq j \le n} \E[h_m(Z_i,Z_j) | {\cal X}] \} \\
&\le& 4 \E[\max_{1 \le m \le d} \max_{1 \le i \neq j \le n} h_m(X_i,X_j)],
\end{eqnarray*}
where the last step follows from Jensen's inequality and the fact that $Z_1,\cdots,Z_n$ have the same joint law as $X_1,\cdots,X_n$.
\end{proof}

\begin{proof}[Proof of Theorem \ref{thm:expectation-bound}]
In the proof, we shall use $K_1,K_2,\cdots,$ to denote absolute constants whose values may differ from place to place, and the indices $i \neq j$ and $m$ implicitly run over $1 \le i \neq j \le n$ and $1 \le m \le d$, respectively. Let $X,X',X'_1,\cdots,X'_n$ be iid random variables with the distribution $F$ and they are independent of $X_1^n$ Let $\varepsilon_i, i = 1,\cdots,n$, be a sequence of iid Rademacher random variables such that $\Prob(\varepsilon_i = \pm1) = 1/2$ and $\varepsilon_i, i =1,\cdots,n,$ are also independent of $X_1^n$. By the randomization inequality \cite[Theorem 3.5.3]{delaPenaGine1999}, we have
$$
\E[ \max_m |\sum_{i \neq j} f_m(X_i,X_j)| ] \le K_1 \E[ \max_m |\sum_{i \neq j} \varepsilon_i \varepsilon_j f_m(X_i,X_j)| ].
$$
Fix an $m=1,\cdots,d$ and let $\Lambda^m$ be the $n \times n$ matrix with diagonal of zeros and $\Lambda^m_{ij} = f_m(X_i,X_j)$ for $i \neq j$. Since $\tr(\Lambda^m)=0$ and $\varepsilon_i$'s are iid sub-Gaussian, by the Hanson-Wright inequality \cite[Theorem 1]{rudelsonvershynin2013a}, conditional on $X_1^n$, we have for all $t>0$
\begin{equation*}
\Prob(|\mbf\varepsilon^\top \Lambda^m \mbf\varepsilon| \ge t \mid X_1^n) \le 2 \exp\left\{ -K_2 \min\left[ {t^2 \over |\Lambda^m|_F^2}, {t \over \|\Lambda^m\|_2} \right] \right\},
\end{equation*}
where $\mbf\varepsilon = (\varepsilon_1,\cdots,\varepsilon_n)^\top$. Denote $V_1 = \max_m |\Lambda^m|_F$ and $V_2 = \max_m \|\Lambda^m\|_2$. Let 
$$t^* = \max \left\{ V_1 \sqrt{\log(d) / K_2}, \quad V_2 {\log(d) / K_2} \right\}.$$
By the union bound, we have
\begin{eqnarray*}
&& \E[\max_m |\mbf\varepsilon^\top \Lambda^m \mbf\varepsilon|  \mid X_1^n] = \int_0^\infty \Prob(\max_m |\mbf\varepsilon^\top \Lambda^m \mbf\varepsilon| \ge t \mid X_1^n) \; dt \\
&& \qquad \le t^* + 2 d \int_{t^*}^\infty \max\left\{ \exp\left( -{K_2 t^2 \over V_1^2} \right), \; \exp\left( -{K_2 t \over V_2} \right) \right\} \; dt.
\end{eqnarray*}
Since $d \ge 2$, we have by changing variables that
$$
\int_{t^*}^\infty \exp\left( -{K_2 t^2 \over V_1^2} \right) \; dt \le {V_1 \over \sqrt{2 K_2} } \int_{\sqrt{2 \log d}}^\infty \exp\left(-{s^2 \over2} \right) \;ds.
$$
By the tail bound $1-\Phi(x) \le \phi(x)/x$ for all $x>0$, where $\Phi(\cdot)$ and $\phi(\cdot)$ are the cdf and pdf of the standard Gaussian random variable, respectively, it follows that
\begin{equation}
\label{eqn:V_1}
2 d \int_{t^*}^\infty \exp\left( -{K_2 t^2 \over V_1^2} \right) \; dt  \le {V_1 \over \sqrt{K_2 \log{d}}} \le K_2 V_1.
\end{equation}
Similarly, we have
$$
2 d \int_{t^*}^\infty \exp\left( -{K_2 t / V_2} \right) \; dt \le 2 V_2 / K_2.
$$
Note that $V_2 \le V_1$. Therefore, we have
\begin{eqnarray}
\nonumber
\E|\sum_{i \neq j} \varepsilon_i \varepsilon_j f(X_i,X_j)|_\infty \le K_3 t^* &\le&  K_3 (\log{d}) \E\max_m |\Lambda^m|_F \\
\label{eqn:Lambda2_term}
&\le& K_3 (\log{d}) (\E\max_m |\Lambda^m|_F^2)^{1/2},
\end{eqnarray}
where the last step (\ref{eqn:Lambda2_term}) follows from Jensen's inequality. 

Next, we bound the term $I := \E[\max_m |\Lambda^m|_F^2] = \E[\max_m \sum_{i \neq j} f_m^2(X_i,X_j)]$. Consider the Hoeffding decomposition of $f_{mk}^2$. Let
$$
\tilde{f}^m_1(x_1) = \E f^2_m(x_1,X') - \E f^2_{mk}
$$
and
$$
\tilde{f}^m(x_1,x_2) = f^2_m(x_1,x_2) - \E f^2_m(x_1,X') - \E f^2_m(X,x_2) + \E f^2_m.
$$
Clearly, $\E [\tilde{f}^m_1(X)] = 0$ and $\E [\tilde{f}^m(X,X')] = \E [\tilde{f}^m(x_1,X')] = \E [\tilde{f}^m(X,x_2)] = 0$ for all $x_1,x_2 \in S$; i.e. $\tilde{f}^m_1$ is centered and $\tilde{f}^m$ is a canonical kernel of U-statistic of order two. Since
$$
f^2_m(x_1,x_2) - \E f^2_m = \tilde{f}^m(x_1,x_2) + \tilde{f}^m_1(x_1) + \tilde{f}^m_1(x_2),
$$
we have by the triangle inequality
\begin{eqnarray}
\nonumber
\E \max_m \sum_{i\neq j} f_m^2(X_i,X_j) &\le& \max_m \sum_{i\neq j} \E f_m^2 + \E \max_m \Big| \sum_{i\neq j} ( f_m^2(X_i,X_j) - \E f_m^2 ) \Big| \\
\label{eqn:f^2_quad}
&\le& n^2 \max_m \E f_m^2 + \E |\sum_{i\neq j} \tilde{f}(X_i,X_j)|_\infty + (n-1) \E |\sum_{i=1}^n \tilde{f}_1(X_i)|_\infty,
\end{eqnarray}
where $\tilde{f}=\{\tilde{f}^m\}_{m=1}^d$ and $\tilde{f}_1=\{\tilde{f}^m_1\}_{m=1}^d$ are $d \times 1$ random vectors. By the Hoeffding inequality, conditional on $X_1^n$, we have for all $t>0$
\begin{equation*}
\Prob(|\sum_{i=1}^n \varepsilon_i \tilde{f}_1^m(X_i)| \ge t \mid X_1^n) \le 2 \exp\left[ - {t^2 \over 2 \sum_{i=1}^n (\tilde{f}_1^m(X_i))^2} \right].
\end{equation*}
By the symmetrization inequality \cite[Lemma 2.3.1]{vandervaartwellner1996} and the argument for bounding (\ref{eqn:V_1}), we get
\begin{equation}
\label{eqn:f_tilde_1}
\E |\sum_{i=1}^n \tilde{f}_1(X_i)|_\infty \le 2 \E |\sum_{i=1}^n \varepsilon_i \tilde{f}_1(X_i)|_\infty \le K_4 (\log{d})^{1/2} \E \Big[ \max_m \sum_{i=1}^n \tilde{f}^m_1(X_i)^2 \Big]^{1/2}.
\end{equation}
By the randomization inequality as in the previous argument before Jensen's inequality (\ref{eqn:Lambda2_term}), we get
\begin{equation}
\label{eqn:f_tilde_2}
\E |\sum_{i \neq j}\tilde{f}(X_i,X_j)|_\infty \le K_1 \E |\sum_{i\neq j} \varepsilon_i \varepsilon_j \tilde{f}(X_i,X_j)|_\infty \le K_3 (\log{d}) \E \Big[ \max_m \sum_{i \neq j} \tilde{f}^m(X_i,X_j)^2 \Big]^{1/2}.
\end{equation}
By the triangle and Jensen's inequalities, we have
\begin{eqnarray*}
&& \E \Big[ \max_m \sum_{i\neq j} \tilde{f}^m(X_i,X_j)^2 \Big]^{1/2}  \le 2 \E \Big[\max_m \sum_{i\neq j} f_m^4(X_i,X_j) \Big]^{1/2} \\
&& \qquad \qquad  + 8^{1/2} \E\Big\{ \max_m \sum_{i\neq j} [\E( f_m^2(X_i, X') | X_1^n)]^2 \Big\}^{1/2} + 2 \Big[ \max_m \sum_{i \neq j} (\E f_m^2)^2 \Big]^{1/2}.
\end{eqnarray*}
By the Cauchy-Schwarz inequality, we have
$$
\E \Big[\max_m \sum_{i\neq j} f_m^4(X_i,X_j) \Big]^{1/2} \le \sqrt{I} \sqrt{\E M^2},
$$
and
\begin{eqnarray*}
&& \E\Big\{ \max_m \sum_{i\neq j} [\E( f_m^2(X_i, X') | X_1^n)]^2 \Big\}^{1/2}\\
&\le& \E \Big\{ \Big[\max_m \max_{i\neq j} \E( f_m^2(X_i, X'_j) | X_1^n) \Big]^{1/2} \Big[\max_m \sum_{i\neq j} \E( f_m^2(X_i, X'_j) | X_1^n \Big]^{1/2} \Big\} \\
&\le& \Big[\E \max_m \max_{i\neq j} \E( f_{mk}^2(X_i, X'_j) | X_1^n) \Big]^{1/2} \Big[ \E \max_m \sum_{i\neq j} \E( f_m^2(X_i, X'_j) | X_1^n) \Big]^{1/2} \\
&\le& 4 \sqrt{\E M^2} \sqrt{I},
\end{eqnarray*}
where in the last step we used Jensen's inequality, Lemma \ref{lem:decoupling_max_nonnegative}, and the decoupling inequality \cite[Theorem 3.1.1]{delaPenaGine1999}. In addition, $\max_m \sum_{i \neq j} (\E f_m^2)^2 \le I \cdot \E M^2$. Therefore, we obtain from (\ref{eqn:f_tilde_2}) that
$$
\E |\sum_{i\neq j}\tilde{f}(X_i,X_j)|_\infty \le K_5 (\log{d}) \sqrt{I} \sqrt{\E M^2}.
$$
By (\ref{eqn:f^2_quad}), (\ref{eqn:f_tilde_1}), and (\ref{eqn:f_tilde_2}), we obtain that
$$
I \le K_6 \Big\{ n^2 \max_m \E f_m^2 + (\log{d}) \sqrt{I} \sqrt{\E M^2} + n (\log{d})^{1/2} \E \Big[ \max_m \sum_i \tilde{f}^m_1(X_i)^2 \Big]^{1/2} \Big\}.
$$
The solution of this quadratic inequality for $I$ is given by
\begin{equation}
\label{eqn:bound_I}
I \le K_7 \Big\{ (\log{d})^2 (\E M^2) + n^2 \max_m \E f_m^2 + n (\log{d})^{1/2} \E \Big[ \max_m \sum_i \tilde{f}^m_1(X_i)^2 \Big]^{1/2} \Big\}.
\end{equation}
By \cite[Lemma 9]{cck2014b}, Jensen's inequality, and Lemma \ref{lem:decoupling_max_nonnegative},
\begin{eqnarray}
\nonumber
\E \Big[ \max_m \sum_i \tilde{f}^m_1(X_i)^2 \Big] &\le& K_8 \Big\{ \max_m \E \sum_i \tilde{f}^m_1(X_i)^2 + (\log{d}) \E  \Big[ \max_m \max_i \tilde{f}^m_1(X_i)^2 \Big] \Big\} \\
\nonumber
&\le& K_8 \Big\{ n \max_m \E f_m^4(X_1,X_2) + (\log{d}) \E  \Big[ \max_m \max_{i \neq j} f_m^4(X_i,X'_j) \Big] \Big\}  \\
\label{eqn:bound_I_simplify}
&\le& K_8 \Big\{ n \max_m \E f_m^4(X_1,X_2) + (\log{d}) \E M^4 \Big\}.
\end{eqnarray}
Now, combining (\ref{eqn:Lambda2_term}), (\ref{eqn:bound_I}), and (\ref{eqn:bound_I_simplify}), we conclude that
\begin{eqnarray*}
&&\qquad  \E | \sum_{i \neq j} f(X_i,X_j) |_\infty \\
&\le& K_9 (\log{d}) \Big\{ n^2 \max_m \E f_m^2 + (\log{d})^2 \E M^2 + n (\log{d})^{1/2} [n \max_m \E f_m^4 + (\log{d}) \E M^4]^{1/2}  \Big\}^{1/2} \\
&\le& K_9 (\log{d}) \Big\{ n D_2 + (\log{d}) \|M\|_2 + n^{3/4} (\log{d})^{1/4} D_4 + (n \log{d})^{1/2} \|M\|_4 \Big\}.
\end{eqnarray*}
Since $\|M\|_2 \le \|M\|_4$ and $d \le \exp(b n)$, we have $(\log{d}) \|M\|_2 \le (b n \log{d})^{1/2} \|M\|_4$, from which (\ref{eqn:expectation-bound-canonical}) follows.
\end{proof}

\begin{proof}[Proof of Corollary \ref{cor:expectation-subexponential-polynomial-kernel}]
Part (i) of the Corollary follows from the bounds $D_2 \le D_4 \le B_n$ and $\|M\|_4 \le K B_n \log(nd)$. Part (ii) of the Corollary follows from the bounds $D_2 \le D_4 \le B'_n$ and $\|M\|_4 \le B_n n^{2/q}$.
\end{proof}

\section{Proof details of Section 5.2}

\begin{proof}[Proof of Proposition \ref{prop:general-GA-rate}]
Without loss of generality, we may assume $Y$ is independent of $X_1^n$. First, we deal with $\tilde\rho^{re}(T_n, Y)$. We may assume that $\phi_n \ge 1$ for otherwise we can choose a large enough $C_1$ to make (\ref{eqn:general-GA-rate}) trivially hold. For any $\phi \ge 1$, let $\beta = \phi \log{d}$ and $F_\beta : \R^d \to \R$ such that
\begin{equation*}
F_\beta(\omega) = {1\over\beta} \log( \sum_{j=1}^d \exp\left( \beta(\omega_j - y_j) \right) ) \qquad \text{for } \omega, y \in \R^d.
\end{equation*}
Then, for all $\omega \in \R^d$
\begin{equation}
\label{eqn:smoothing_property_F_beta}
0 \le F_\beta(\omega) - \max_{1 \le j \le d} (\omega_j-y_j) \le  {\log{d} \over \beta} = \phi^{-1}.
\end{equation}
Let $u_0 : \R \to [0,1]$ be a thrice continuously differentiable function with bounded derivatives, say by an absolute constant $K_1>0$, such that $u_0(t) = 1$ if $t \le 0$ and $u_0(t) = 0$ if $t>1$. Let $u(t) = u_0(\phi t)$; then $\lambda(\omega) = u(F_\beta(\omega))$ is a mapping from $\R^d$ to $[0,1]$, where $u(\cdot)$ satisfies the smoothing property
\begin{equation}
\label{eqn:smoothing_property_u}
\vone(t \le 0) \le u(t) \le \vone(t \le \phi^{-1}) \qquad \text{for } t \in \R.
\end{equation}
Let $Y'$ be an independent copy of $Y$ and for $ v \in [0,1]$
$$
{\cal I}_n(v) = \lambda(\sqrt{v} L_n + \sqrt{1-v} Y) - \lambda(Y').
$$
By \cite[equation (26)]{cck2015a}, we have for all $\phi \ge 1$
$$
|\E [{\cal I}_n(v)]| \le C_1(\ub) {\phi^2 \log^2{d} \over n^{1/2} } \left( \phi D_{g,3} \rho_{n,1} +  D_{g,3} \log^{1/2}{d} + \phi  M_3(\phi) \right),
$$
where
$$
\rho_{n,1} = \sup_{y \in \R^d, \; v \in [0,1]} \left| \Prob(\sqrt{v} L_n + \sqrt{1-v} Y \le y) - \Prob(Y' \le y) \right|.
$$
By the mean value theorem, we have
$$
\lambda(T_n) - \lambda(L_n) = \sum_{m=1}^d \partial_m \lambda(\xi) R_{nm} = \sum_{m=1}^d u'(F_\beta(\xi)) \eta_m(\xi) R_{nm},
$$
where $\eta_m(\omega) = \partial_m F_\beta(\omega)$ is the first-order partial derivative of $F_\beta$ w.r.t. $\omega_m$ and $\xi$ is a $d\times 1$ random vector on the line segment joining $L_n$ and $T_n$. Since $\eta_m(\omega) \ge 0$, $\sum_{m=1}^d \eta_m(\omega) = 1$ for any $\omega \in \R^d$ and $\sup_{t \in \R} |u'(t)| \le K_1 \phi$, we have
$$
| \E [\lambda(T_n) - \lambda(L_n)] | \le K_1 \phi \E |R_n|_\infty.
$$
Therefore, by the smoothing properties (\ref{eqn:smoothing_property_F_beta}) and (\ref{eqn:smoothing_property_u}), we have
\begin{eqnarray*}
&& \Prob(T_n \le y-\phi^{-1}) \le \Prob(F_\beta(T_n) \le 0) \le \E[\lambda(T_n)] \\
&\le& \E[\lambda(L_n)] + K_1 \phi \E|R_n|_\infty \le \E[\lambda(Y')] + |\E[{\cal I}_n(1)]| + K_1 \phi \E|R_n|_\infty \\
&\le& \Prob(F_\beta(Y') \le \phi^{-1}) + |\E[{\cal I}_n(1)]| + K_1 \phi \E|R_n|_\infty \\
&\le& \Prob(Y' \le y + \phi^{-1}) + |\E[{\cal I}_n(1)]| + K_1 \phi \E|R_n|_\infty \\
&\le& \Prob(Y' \le y - \phi^{-1}) + C_2(\ub) \phi^{-1} \log^{1/2}{d} + |\E[{\cal I}_n(1)]| + K_1 \phi \E|R_n|_\infty,
\end{eqnarray*}
where the last inequality follows from Nazarov's inequality \cite{nazarov2003} or \cite[Lemma A.1]{cck2015a}. For the lower bound, we interchange the roles of $T_n$ and $Y'$ to get
\begin{eqnarray*}
\Prob(Y' \le y- \phi^{-1}) \le  \Prob(T_n \le y + \phi^{-1}) + |\E[{\cal I}_n(1)]| + K_1 \phi \E|R_n|_\infty.
\end{eqnarray*}
By a second application of Nazarov's inequality \cite[Lemma A.1]{cck2015a}, $\Prob(Y' \le y- \phi^{-1}) \ge \Prob(Y' \le y + \phi^{-1}) - C_2(\ub) \phi^{-1} \log^{1/2} d$. Since $y \in \R^d$ is arbitrary, we have
\begin{eqnarray*}
&& \tilde\rho^{re}(T_n,Y') \le K_1 \phi E|R_n|_\infty + C_2(\ub) \phi^{-1} \log^{1/2}{d} \\
&& \qquad\qquad + C_1(\ub) {\phi^2 \log^2{d} \over n^{1/2} } \left( \phi D_{g,3} \rho_{n,1} +  D_{g,3} \log^{1/2}{d} + \phi  M_3(\phi) \right) .
\end{eqnarray*}
Choose $\phi=\phi_n \ge 1$.Then, by the argument of \cite[Theorem 2.1 and Lemma 5.1]{cck2015a} in bounding $\rho_{n,1}$, we get
\begin{equation}
\label{eqn:rho_re_Tn_Y}
\tilde\rho^{re}(T_n,Y') \le C_3(\ub) \left\{ \phi_n \E |R_n|_\infty +  \left({\bar{D}_{g,3}^2 \log^7 d \over n}\right)^{1/6} + {M_3(\phi_n) \over \bar{D}_{g,3}}\right\}
\end{equation}
for any real sequence $\bar{D}_{g,3} \ge D_{g,3}$. By Jensen's inequality and Lemma \ref{lem:decoupling_max_nonnegative},
$$\E[\max_{1 \le m \le d} \max_{1 \le i \neq j \le n} f_m^4(X_i, X_j)] \le K_2 \E[\max_{1 \le m \le d} \max_{1 \le i \neq j \le n} h_m^4(X_i, X_j)]$$
for some absolute constant $K_2 > 0$. By Theorem \ref{thm:expectation-bound} and in view of $M_{h,4}(0) \le M_{h,4}(\tau) + \tau^4$ for any $\tau \ge 0$, there exists an absolute constant $K_3 > 0$ such that 
\begin{equation}
\label{eqn:Rn_max_norm_bound}
\E |R_n|_\infty \le K_3 (\bar{b}^{1/2}  + 1) \left\{ {\log^{3/2}{d} \over n} (M_{h,4}(\tau)^{1/4} + \tau) + {\log{d} \over n^{1/2}} D_2^{1/2} + {\log^{5/4}{d} \over n^{3/4}} D_4^{1/4} \right\}.
\end{equation}
Now, (\ref{eqn:general-GA-rate}) follows from (\ref{eqn:rho_re_Tn_Y}) and (\ref{eqn:Rn_max_norm_bound}). The bound for $\rho^{re}(T_n,Y)$ can be similarly dealt with using the argument in \cite[Corollary 5.1]{cck2015a} and we omit the details. 
\end{proof}

\section{Proof details of Section 5.3}
\label{app:proof-of-bootstraps}

\begin{lem}
\label{lem:gauss-comp-general}
Suppose that (M.1) holds. Let $\Gamma_n$ be a symmetric positive-semidefinite matrix depending only on $X_1^n$. Let $Z^X \mid X_1^n \sim N(0,\Gamma_n)$ and $\Delta_n = |\Gamma_n-\Gamma|_\infty$. Then there exists a constant $C(\ub) > 0$ such that for every sequence of real numbers $\bar{\Delta}_n > 0$, we have 
\begin{equation}
\label{eqn:gauss-comp-general}
\sup_{A \in {\cal A}^{re}} |\Prob(Y \in A) - \Prob(Z^X \in A \mid X_1^n) | \le C(\ub) (\log d)^{2/3} \bar\Delta_n^{1/3}
\end{equation}
on the event $\{\Delta_n \le \bar{\Delta}_n\}$.
\end{lem}

\begin{proof}[Proof of Lemma \ref{lem:gauss-comp-general}]
Let $F_\beta(\cdot)$ be defined in Proposition \ref{prop:general-GA-rate} with $\beta = \phi \log{d}$. By a variant of \cite[Theorem 1]{cck2014b}, we have
$$
\left| \E[u(F_\beta(Z^X)) \mid X_1^n] - \E[u(F_\beta(Y))] \right| \le (\|u''\|_\infty / 2 + \beta \|u'\|_\infty) \Delta_n,
$$
where $u$ and $F_\beta$ are the same as in the proof of Proposition \ref{prop:general-GA-rate}. By the smoothing steps in the proof of Proposition \ref{prop:general-GA-rate}, we obtain that 
\begin{eqnarray*}
\left| \Prob(Z^X \le y - \phi^{-1} \mid X_1^n) - \Prob(Y \le y - \phi^{-1}) \right| \le C(\ub) \{ \phi^{-1} (\log d)^{1/2} + (\phi^2 + \beta \phi) \Delta_n \}.\end{eqnarray*}
Since $y \in \mathbb{R}^d$ is arbitrary, we have
$$
\sup_{y \in \R^d} |\Prob(Y \le y) - \Prob(Z^X \le y \mid X_1^n) | \le C(\ub) \{ \phi^{-1} (\log d)^{1/2} + (\phi^2 + \beta \phi) \Delta_n \}.
$$
Using $\beta = \phi \log{d}$ and optimizing the last bound over $\phi$, we have
$$
\sup_{y \in \R^d} |\Prob(Y \le y) - \Prob(Z^X \le y \mid X_1^n) | \le C(\ub) (\log d)^{2/3} \bar\Delta_n^{1/3}
$$
on the event $\{\Delta_n \le \bar{\Delta}_n\}$. Then (\ref{eqn:gauss-comp-general}) follows from the argument in \cite[Corollary 5.1]{cck2015a}.
\end{proof}

\begin{proof}[Proof of Proposition \ref{prop:gaussian_comp_leading_term-EB}]
In this proof, we shall use $K_i > 0, i=1,2,\cdots,$ to denote absolute constants. The proof is to use Lemma \ref{lem:gauss-comp-general}. Specifically, we shall find a real sequence $\bar{\Delta}_n$ such that $\Prob(\hat\Delta_n \ge \bar{\Delta}_n) \le \gamma$ and then bound $(\bar\Delta_n \log^2{d})^{1/3}$. First, we may assume that 
$$
n^{-1} B_n^2 \log^5(nd) \log^2(1/\gamma) \le 1
$$
since otherwise the proof is trivial by setting the constants $C(\ub) \ge 1$ in (i) and $C(\ub,q) \ge 1$ in (ii).
Let $\hat\Gamma_{n,1} = n^{-3} \sum_{i,j,k=1}^n h(X_i,X_j) h(X_i,X_k)^\top$, $\hat\Gamma_{n,2} = V_n V_n^\top$, $\Gamma_1 = \E[h(X_1,X_2) h(X_1,X_3)^\top]$, and $\Gamma_2 = \theta \theta^\top$. Then, $\hat\Gamma_n = \hat\Gamma_{n,1}-\hat\Gamma_{n,2}$ and $\Gamma = \Gamma_1-\Gamma_2$. 

{\bf Case (i).} Assuming (E.1) and (E.1').  We shall first deal with $\hat\Gamma_{n,1}-\Gamma_1$, which can be decomposed as the summation of five U-statistics
\begin{equation}
\label{eqn:decomp_Delta_hat}
\sum_{i,j,k=1}^n = \underbrace{\sum_{1 \le i \neq j \neq k \le n} + \sum_{j=k \neq i}}_{\text{off diagonal of } h} + \underbrace{\left (\sum_{i=j \neq k} + \sum_{i=k \neq j} \right) + \sum_{i=j=k}}_{\text{on and off diagonal of } h}.
\end{equation}
The first two terms involve only off-diagonal entries $h(x_i, x_j)$ for $i \neq j$, while the last three terms involve both diagonal $h(x_i, x_i)$ and off-diagonal entries. We first tackle the case $1 \le i \neq j \neq k \le n$, which will be shown to be the leading term in $|\hat\Gamma_{n,1}-\Gamma_1|_\infty$ and the remaining cases can similarly be dealt with. 

{\it Term $\sum_{1 \le i \neq j \neq k \le n}$.} Let $H(x_1,x_2,x_3) = h(x_1,x_2) h(x_1,x_3)^\top$ for $x_i \in S, i=1,2,3,$ and
$$
\hat\Gamma_{n,1,1} = {(n-3)! \over n!} \sum_{1 \le i \neq j \neq k \le n} H(X_i, X_j, X_k).
$$
Then $\hat\Gamma_{n,1,1}$ is a U-statistics of order three and $\E [\hat\Gamma_{n,1,1}] = \Gamma_1$. Let $r = [n/3]$, $Z = r |\hat\Gamma_{n,1,1}-\Gamma_1|_\infty$, and
$$M = \max_{0 \le i \le r-1} \max_{1 \le m_1, m_2 \le d} |H_m(X_{3i+1}^{3i+3})|$$
for all $m = (m_1,m_2)$ where $m_1,m_2 = 1,\cdots,d$. Let $\tau = 8 \E[M]$. By Lemma \ref{lem:subexp-concentration-ineq} with $\alpha = 1/2$, $\eta=1$, and $\delta=1/2$, we have for all $t > 0$ that
\begin{equation}
\label{eqn:proof-eb-bootstrap-rawbound}
\Prob(Z \ge 2 \E Z_1 + t)  \le \exp\left(-{t^2 \over 3 \bar\zeta_n^2} \right) + 3 \exp\left[ -\left({t \over K_1 \left\| M\right\|_{\psi_{1/2}}} \right)^{1/2} \right],
\end{equation}
where
\begin{equation*}
\bar\zeta_n^2 = \ \max_m\sum_{i=0}^{r-1} \E H_m^2(X_{3i+1}^{3i+3}), \quad Z_1 = \max_m |\sum_{i=0}^{r-1} [\bar{H}_m(X_{3i+1}^{3i+3}) - \E \bar{H}_m]|,
\end{equation*}
and $\bar{H}_m(x_1^3) = H_m(x_1^3) \vone(\max_{1 \le m_1,m_2 \le d} |H_m(x_1^3)| \le \tau)$ for all $m=(m_1,m_2)$ and $m_1, m_2=1,\cdots,d$. By \cite[Lemma 8]{cck2014b} and Jensen's inequality,
\begin{eqnarray*}
\E Z_1 &\le& K_2 \Bigg\{ (\log d)^{1/2} \left[ \max_m \sum_{i=0}^{r-1} \E( \bar{H}_m(X_{3i+1}^{3i+3})-\E \bar{H}_m)^2 \right]^{1/2} \\
&& \qquad  \qquad + (\log d) [\E \max_{i,m} |\bar{H}_m(X_{3i+1}^{3i+3})-\E \bar{H}_m|^2 ]^{1/2} \Bigg\} \\
&\le& K_2 \left\{ (\log d)^{1/2} \bar\zeta_n + (\log d) \left\| M \right\|_{\psi_{1/2}} \right\}.
\end{eqnarray*}
By the Cauchy-Schwarz inequality and (M.2)
$$
\E H_m^2(X_{3i+1}^{3i+3}) \le [\E h_{m_1}^4(X_{3i+1},X_{3i+2})]^{1/2} [\E h_{m_2}^4(X_{3i+1},X_{3i+3})]^{1/2} \le B_n^2
$$
so that $\bar\zeta_n \le r^{1/2} B_n \le n^{1/2} B_n$. Note that for $\|X\|_{\psi_1} < \infty$, we have $\|X^2\|_{\psi_{1/2}} = \|X\|_{\psi_1}^2$. By Pisier's inequality and (E.1)
$$
\left\| M \right\|_{\psi_{1/2}} \le K_3 \log^2(r d) \max_{i,m} \left\| h_m(X_{3i+1}^{3i+2}) \right\|_{\psi_1}^2 \le K_3 B_n^2  \log^2(n d).
$$
Therefore, with the assumption that $B_n^2 \log^5(nd) \log^2(1/\gamma) \le n$, we have
$$
\E Z_1 \le K_4 \{ (n B_n^2 \log{d})^{1/2} + B_n^2 (\log{d})  \log^2(n d) \} \le 2 K_4 (n B_n^2 \log(nd))^{1/2}.
$$
Combining (\ref{eqn:proof-eb-bootstrap-rawbound}) with the last inequality, we have
\begin{eqnarray*}
&& \Prob(|\hat\Gamma_{n,1,1}-\Gamma_1|_\infty \ge K_5 (n^{-1} B_n^2 \log(nd))^{1/2} + t)   \\
&&  \qquad \le \exp\left(-{n t^2 \over K_6 B_n^2} \right) + 3 \exp\left(-{\sqrt{n t} \over K_7 B_n \log(nd) } \right).
\end{eqnarray*}
Choose 
\begin{equation}
\label{eqn:t_star}
t^* = K_8 \sqrt{B_n^2 \log(nd) \log^2(1/\gamma) \over n}
\end{equation}
for some large enough $K_8>0$. Then we have
\begin{eqnarray*}
\Prob(|\hat\Gamma_{n,1,1}-\Gamma_1|_\infty \ge 2 t^*)  &\le&   \exp\left(-(K_8^2 K_6^{-1}) \log(nd) \log^2(1/\gamma) \right) \\
&& + 3 \exp\left( - (K_8^{1/2} K_7^{-1}) n^{1/4} \log^{1/2}(1/\gamma) \log^{-3/4}(nd) B_n^{-1/2} \right).
\end{eqnarray*}
Recall that $d \ge 3$ and $\gamma \in (0, e^{-1})$. So $\log(dn) \ge \log(3) > 1$ and $\log(1/\gamma) > 1$.
Therefore, we obtain that
\begin{equation}
\label{eqn:gamma_n11}
\Prob( |\hat\Gamma_{n,1,1}-\Gamma_1|_\infty \ge 2 t^*) \le \gamma^{K_8^2 \over K_6} + 3 \gamma^{K_8^{1/2} \over K_7} \le \gamma / 17,
\end{equation}
from which it follows that
$$
\Prob( ( |\hat\Gamma_{n,1,1}-\Gamma_1|_\infty \log^2(dn) )^{1/3} \ge K_9 \varpi_{n,1}^B(\gamma) ) \le \gamma / 17.
$$

{\it Term $\sum_{j = k \neq i}$.} This term can be similarly handled as the pervious term and here we only sketch the argument. Let $H(x_1,x_2) = h(x_1,x_2) h(x_1,x_2)^\top$ and
$$
\hat\Gamma_{n,1,2} = {(n-2)! \over n!} \sum_{1 \le i \neq j \le n} H(X_i, X_j).
$$
Let $\Gamma_{1,2} = \E[h(X_1,X_2) h(X_1,X_2)^\top]$ and $r = [n/2]$. By the argument for bounding the term $\sum_{1 \le i \neq j \neq k \le n}$, we can show that
$$
\Prob(|\hat\Gamma_{n,1,2}-\Gamma_{1,2}|_\infty \ge t^*) \le \gamma/17,
$$
where $t^*$ is given by (\ref{eqn:t_star}). By the Cauchy-Schwarz inequality and (M.2),  $|\Gamma_{1,2}|_\infty \le \max_{1 \le m \le d} \E[h_m^2(X_1,X_2)] \le B_n^{2/3},$ from which we get $n^{-1} |\Gamma_{1,2}|_\infty \le t^*/2$. So it follows that
$$
\Prob(n^{-1} |\hat\Gamma_{n,1,2}|_\infty \ge 3 t^* / 2) \le \gamma/17.
$$

{\it Terms $(\sum_{i=j \neq k} + \sum_{i=k \neq j} ) + \sum_{i=j=k}$.} Those three terms can be similarly dealt with by invoking (M.2') and (E.1'). Details are omitted.

To deal with $\hat\Gamma_{n,2}-\Gamma_2$, we can split $\sum_{i,j,k,l=1}^n h(X_i,X_j) h(X_k,X_l)^\top$ into sums, where there are zero, two, three, and four equalities in the index set $(i,j,k,l)$, respectively. So, there are $2^4-4 = 12$ such groups, all of which obey the same bound as before. Therefore we can conclude that
$$
\Prob( ( |\hat\Gamma_n-\Gamma|_\infty \log^2(dn) )^{1/3} \ge K_{10} \varpi_{n,1} ) \le \gamma
$$
and (\ref{eqn:prop:gaussian_comp_leading_term-EB_subexp}) follows from Lemma \ref{lem:gauss-comp-general}. \\

{\bf Case (ii).} Assuming (E.2) and (E.2'). Again, we start from the decomposition (\ref{eqn:decomp_Delta_hat}) and use the same notations as in {\bf Case (i)}.

{\it Term $\sum_{1 \le i \neq j \neq k \le n}$.} Recall $H(x_1^3) = h(x_1,x_2) h(x_1,x_3)^\top$. Let $\tau = 4 \cdot 2^{2/q} \cdot \|M\|_{q/2}$. By Lemma \ref{lem:fuknagaev-concentration-ineq} with $\eta = 1$ and $\delta =1/2$, we have
$$
\Prob(Z \ge 2 \E Z_1 + t) \le \exp\left(-{t^2 \over 3 \bar\zeta_n^2 }\right) + C_1(q) {\E[M^{q/2}] \over t^{q/2}}.
$$
By (M.2), $\bar\zeta_n \le n^{1/2} B_n$ and by \cite[Lemma 8]{cck2014b}, Jensen's inequality, and (E.2) with $q \ge 4$, we have
\begin{eqnarray*}
\E Z_1 &\le& K_1 \{(\log d)^{1/2} \bar\zeta_n + (\log d) [\E \max_{i,m} H_m^2(X_{3i+1}^{3i+3})]^{1/2} \} \\
&\le& K_1 \{ (n B_n^2 \log d)^{1/2} + (\log d) n^{2/q} B_n^2 \}.
\end{eqnarray*}
Then, it follows that
\begin{eqnarray*}
&& \Prob \left( |\hat\Gamma_{n,1,1}-\Gamma_1|_\infty \ge K_2 \left\{ \left( {B_n^2 \log d \over n} \right)^{1/2} + \left( {B_n^2 \log d \over n^{1-2/q}} \right) \right\} + t \right) \\
&\le& \exp\left(-{ n t^2 \over K_3 B_n^2 }\right) + C_1(q) {B_n^q \over n^{q/2-1} t^{q/2}}.
\end{eqnarray*}
Choose
\begin{equation}
\label{eqn:t_star_poly}
t^* = K_4 \left[{B_n^2 \log(nd) \log^2(1/\gamma) \over n}\right]^{1/2} + C_2(q) {B_n^2 \log d \over n^{1-2/q} \gamma^{2/q}}
\end{equation}
for some $K_4, C_2(q) > 0$ large enough. Then, we get
\begin{eqnarray*}
\Prob ( |\hat\Gamma_{n,1,1}-\Gamma_1|_\infty \ge C_3(q) t^*) &\le& \exp(-(K_4^2 K_3^{-1}) \log(nd) \log^2(1/\gamma)) + {C_1(q) \over C_2(q)^{q/2}} \gamma \\
&\le& \gamma /17.
\end{eqnarray*}
So we obtain that
$$
\Prob( ( |\hat\Gamma_{n,1,1}-\Gamma_1|_\infty \log^2(dn) )^{1/3} \ge C_4(q) \{ \varpi_{n,1}^B(\gamma) + \varpi_{n,2}^B(\gamma) \} ) \le \gamma / 17.
$$
The rest terms can be similarly handled and and (\ref{eqn:prop:gaussian_comp_leading_term-EB_unifpoly}) follows from Lemma \ref{lem:gauss-comp-general}.
\end{proof}

Recall that for $q > 0$ 
$$
\hat{D}_{g,q} = \max_{1 \le j \le d} n^{-1} \sum_{k=1}^n W_{jk}^q \quad \text{and} \quad \hat{D}_q = \max_{1 \le j \le d} n^{-2} \sum_{k,l=1}^n |h_j(X_k,X_l)|^q,
$$
where $W_{jk} =  |n^{-1} \sum_{i=1}^n h_j(X_k,X_i) - V_{nj}|$ and $V_{nj} = n^{-2} \sum_{k,i=1}^n h_j(X_k,X_i)$ for $j=1,\cdots,d$ and $k=1,\cdots,n$.

\begin{lem}
\label{lem:bounds_on_terms_empirical-version-EB}
Let $\gamma \in (0,1)$ and $\beta > 0$. Suppose that $n^{-1} B_n^2 \log^7(nd) \le 1$. \\
(i) Assume that (M.1), (M.2), (M.2'), (E.1), (E.1') hold and in addition $\log(1/\gamma) \le K \log(dn)$. Then there exist constants $C_i := C_i(\beta,K) > 0, i=1,2,3$ such that 
\begin{eqnarray*}
\Prob(\hat{D}_{g,3} \ge C_1 B_n) &\le& \gamma/\beta, \\
\Prob(\hat{D}_2 \ge C_2 B_n^{2/3}) &\le& \gamma/\beta, \\
\Prob(\hat{D}_4 \ge C_3 B_n^2 \log(dn)) &\le& \gamma/\beta.
\end{eqnarray*}

(ii) Assume that (M.1), (M.2), (M.2'), (E.2), (E.2') hold with $q \ge 4$ and in addition $B_n^2 \log^3(dn) \gamma^{-2/q} n^{-1+2/q} \le 1$. Then there exist constants $C_i := C_i(\beta,q) > 0, i=1,2,3$ such that 
\begin{eqnarray*}
\Prob(\hat{D}_{g,3} \ge C_1 [B_n + n^{-1+3/q} B_n^3 \gamma^{-3/q} (\log d)] ) &\le& \gamma/\beta, \\
\Prob(\hat{D}_2 \ge C_2 [B_n^{2/3} + n^{-1+2/q} B_n^2 \gamma^{-2/q} (\log d)] ) &\le& \gamma/\beta, \\
\Prob(\hat{D}_4 \ge C_3 [B_n^2 + n^{-1+4/q} B_n^4 \gamma^{-4/q} (\log d)] ) &\le& \gamma/\beta.
\end{eqnarray*}
\end{lem}

\begin{proof}[Proof of Lemma \ref{lem:bounds_on_terms_empirical-version-EB}]
In this proof, we shall use $K_1,K_2,\cdots$ to denote absolute constants and $C_1,C_2>0$ to denote constants that only depends on $\beta,K$ in case (i) and $\beta,q$ in case (ii). 

{\bf Case (i).} Consider the term $\hat{D}_{g,3}$. By Jensen's inequality, $\hat{D}_{g,3} \le 8 n^{-1} (n-1) \hat{D}_{g,3,1} + 8 n^{-1} \hat{D}_{g,3,2}$, where
\begin{equation}
\label{eqn:hat_D_g31_g32}
\hat{D}_{g,3,1} = {1 \over n(n-1)} \max_{1\le j \le d} \sum_{1 \le k \neq i \le n} |h_j(X_k,X_i)|^3, \quad \hat{D}_{g,3,2} = {1 \over n} \max_{1\le j \le d} \sum_{i=1}^n |h_j(X_i,X_i)|^3.
\end{equation}
Let $H(x_1,x_2) = \{H_j(x_1,x_2)\}_{j=1}^d$ and $H_j(x_1,x_2) = |h_j(x_1,x_2)|^3$. Let $M = \max_{1 \le j \le d} \max_{0 \le i \le m-1} H_j(X_{2i+1}^{2i+2})$ and $Z' = \max_{1 \le j \le d} \sum_{i=0}^{m-1} H_j(X_{2i+1}^{2i+2})$, where $m = [n/2]$. Applying Lemma \ref{lem:subexp-concentration-ineq-nonnegative} to the non-negative kernel $H(x_1,x_2)$ with $\alpha=1/3$ and $\eta = 1$, we have that for all $t>0$
$$
\Prob(m \hat{D}_{g,3,1} \ge 2 \E Z' + t) \le 3 \exp \left\{ -\left( t \over K_1 \|M\|_{\psi_{1/3}} \right)^{1/3} \right\}.
$$
Since $\|H_j(X_{2i+1}^{2i+2})\|_{\psi_{1/3}} = \| h_j(X_{2i+1}^{2i+2})\|_{\psi_1}^3$, by (E.1) and Pisier's inequality, $\|M\|_{\psi_{1/3}} \le K_2 B_n^3 \log^3(dn)$. Then by \cite[Lemma 9]{cck2014b} and (M.2), we have
\begin{eqnarray*}
\E Z' &\le& K_3 \{ \max_{1 \le j \le d} \sum_{i=0}^{m-1} \E H_j(X_{2i+1}^{2i+2}) + (\log d) \E[M] \} \\
&\le& K_3 \{n B_n + B_n^3 \log^4(dn) \}.
\end{eqnarray*}
Therefore, it follows that
$$
\Prob(\hat{D}_{g,3,1} \ge K_4 \{B_n + n^{-1} B_n^3 \log^4(dn)\} + t) \le 3 \exp \left \{ - {(nt)^{1/3} \over K_5 B_n \log(dn)} \right \},
$$
which is equivalent to
$$
\Prob \left(\hat{D}_{g,3,1} \ge K_6 \left\{ B_n + n^{-1} B_n^3 \log^4(dn) + n^{-1} B_n^3 \log^3(dn) s^3 \right\} \right) \le 3 e^{-s}
$$
holds for all $s>0$. Now, choose $s=\log(6\beta/\gamma)$. Then $s \le C_1 \log(dn)$ for some constant $C_1 > 0$. Then, it is easy to check that
$$
B_n + n^{-1} B_n^3 \log^4(dn) + n^{-1} B_n^3 \log^3(dn) s^3 \le (2+C_1^3) B_n.
$$
Therefore, we have $\Prob(\hat{D}_{g,3,1}  \ge  C_2 B_n ) \le \gamma/(2\beta)$. Under (M.2') and (E.1'), by the same calculations, we can show that $\Prob(\hat{D}_{g,3,2} \ge C_3 B_n) \le \gamma/(2\beta)$. Thus, $\Prob(\hat{D}_{g,3}  \ge  C_4 B_n ) \le \gamma/\beta$. In addition, routine calculations show that $\hat{D}_2$ and $\hat{D}_4$ can be dealt in similar ways and we can show that $\Prob(\hat{D}_{n,2}  \ge C_5 B_n^{2/3}) \le \gamma/\beta$ and $\Prob(\hat{D}_{n,4} \ge C_6 B_n^2 \log(dn)) \le \gamma/\beta$.

{\bf Case (ii).} Applying Lemma \ref{lem:fuknagaev-concentration-ineq-nonnegative} to the non-negative kernel $H(x_1,x_2) = \{H_j(x_1,x_2)\}_{j=1}^d$ and $H_j(x_1,x_2) = |h_j(x_1,x_2)|^3$ with $\eta = 1$, we have that for all $t > 0$
$$
\Prob(m \hat{D}_{g,3,1} \ge 2 \E Z' + t) \le C_1(q) {\E[M^{q/3}] \over t^{q/3}}.
$$
By (M.2) and (E.2), we have $\E H_j(X_{2i+1}^{2i+2}) \le B_n$ and $E[M^{q/3}] \le n B_n^q$. So $\|M\|_1 \le \|M\|_{q/3} \le n^{3/q} B_n^3$ and
\begin{eqnarray*}
\E Z' &\le& K_1 \{ \max_{1 \le j \le d} \sum_{i=0}^{m-1} \E H_j(X_{2i+1}^{2i+2}) + (\log d) \E[M] \} \\
&\le& K_1 \{n B_n + n^{3/q} B_n^3 \log d \}.
\end{eqnarray*}
Then, we have for all $t > 0$
$$
\Prob(\hat{D}_{g,3,1,} \ge K_2 \{B_n + n^{-1+3/q} B_n^3 \log d\} + t) \le C_1(q) {B_n^q \over n^{q/3-1} t^{q/3}},
$$
from which it follows that $\Prob(\hat{D}_{g,3,1} \ge C_2 \bar{D}_{g,3}) \le \gamma/(2\beta)$, where $\bar{D}_{g,3} = B_n + n^{-1+3/q} B_n^3 \gamma^{-3/q} (\log d)$. By the same argument and using (M.2') and (E.2'), we can show that $\Prob(\hat{D}_{g,3,2} \ge C_3 \bar{D}_{g,3}) \le \gamma/(2\beta)$. Thus, $\Prob(\hat{D}_{g,3}  \ge  C_4 B_n ) \le \gamma/\beta$. Finally, $\hat{D}_2$ and $\hat{D}_4$ can be dealt in similar ways.
\end{proof}

\begin{proof}[Proof of Theorem \ref{thm:iid-weighted-bootstrap}]
Without loss of generality, we may assume (\ref{eqn:eb_reduction1}). Observe that $T_n^\diamond = L_n + R_n$, where
\begin{eqnarray}
\label{eqn:Ln_iid-weighted-bootstrap}
L_n &=& {1 \over \sqrt{n} (n-1)} \sum_{i=1}^n (\sum_{j \neq i} h(X_i,X_j)) (w_i-1), \\
\label{eqn:Rn_iid-weighted-bootstrap}
R_n &=& {1 \over 2 \sqrt{n} (n-1)} \sum_{1 \le i \neq j \le n} (w_i-1)(w_j-1) h(X_i,X_j).
\end{eqnarray}
Since $w_i$ are iid $N(1,1)$ and $L_n | X_1^n \sim N(0, \breve{\Gamma}_n)$, where 
$$\breve{\Gamma}_n = {1 \over n(n-1)^2} \sum_{i=1}^n \sum_{j \neq i} \sum_{k \neq i} h(X_i,X_j) h(X_i,X_k)^\top.$$
Let $Y \sim N(0,\Gamma)$. Since $\theta = 0$, by Theorem \ref{thm:CLT-hyperrec-momconds} and following the proof of Proposition \ref{prop:gaussian_comp_leading_term-EB}, we have: (i) if (E.1) holds, then with probability at least $1-2\gamma/6$ we have
$$
\rho^{re}(T_n, Y) + \sup_{A \in {\cal A}^{re}} |\Prob(Y \in A) - \Prob(L_n \in A \mid X_1^n)| \le C_1(\ub,\bar{b},K) \varpi_{1,n};
$$
(ii) if (E.2) holds, then with probability at least $1-2\gamma/6$ we have
$$
\rho^{re}(T_n, Y) + \sup_{A \in {\cal A}^{re}} |\Prob(Y \in A) - \Prob(L_n \in A \mid X_1^n)| \le C_1(\ub,\bar{b},q,K) \{ \varpi_{1,n} + \varpi^{B}_{2,n}(\gamma) \}.
$$
By (\ref{eqn:approx-diagram}) with $Z^X = L_n$ and $T_n^\natural = T_n^\diamond$, it remains to bound $\rho^{re}(L_n, T_n^\diamond \mid X_1^n)$. As in the proof of Proposition \ref{prop:gaussian_comp_leading_term-EB}, we have $\{\breve{\Gamma}_{n,jj} \ge \ub/2, \forall j=1,\cdots,d\}$ occurs with probability at least $1-\gamma/6$. Let $\E_w$ be the expectation taken w.r.t. $w_1,\cdots,w_n$. By the smoothing properties (\ref{eqn:smoothing_property_F_beta}) and (\ref{eqn:smoothing_property_u}), there exists an absolute constant $K_1 > 0$ and a constant $C_2 > 0$ depending only on $\ub$ such that
$$
\rho^{re}(L_n, T_n^\diamond \mid X_1^n) \le C_2(\ub) \phi^{-1} (\log d)^{1/2} + K_1 \phi \E_w |R_n|_\infty
$$
holds for all $\phi > 0$ on the event $\{\breve{\Gamma}_{n,jj} \ge \ub/2, \forall j=1,\cdots,d\}$. Minimizing over $\phi > 0$, we get $\rho^{re}(L_n, T_n^\diamond \mid X_1^n) \le C_3(\ub) (\log d)^{1/4} (\E_w|R_n|_\infty)^{1/2}$. Conditional on $X_1^n$, $R_n$ is a weighted and completely degenerate quadratic form. Therefore, by Lemma \ref{lem:maximal_ineq_weighted_quad_forms}, there exists an absolute constant $K_2 > 0$ such that
\begin{eqnarray*}
\E_w |R_n|_\infty &\le& K_2 {\log d \over n^{{3/2}}} \{ (\log d) [\E_w(M^2)]^{1/2} + A_1 \\
%\label{eqn:R_n_bound_iid_weights}
&& \qquad + (\log d)^{1/4} A_2 \|w_1\|_4 + (\log d)^{1/2} [\E_w(\tilde{M}^4)]^{1/4} \},
\end{eqnarray*}
where 
\begin{eqnarray*}
A_1 &=& \max_{1 \le m \le d} [\sum_{1 \le i \neq j \le n} h_m^2(X_i,X_j)]^{1/2}, \\
A_2 &=& \max_{1 \le m \le d} \{\sum_{i=1}^n [\sum_{j \neq i} h_m^2(X_i,X_j)]^2 \}^{1/4}, \\
M &=& \max_{1 \le m \le d} \max_{1 \le i \neq j \le n} |h_m(X_i,X_j) (w_i-1)(w_j-1)|, \\
\tilde{M} &=& \max_{1 \le m \le d} \max_{1 \le i \le n} [\sum_{j \neq i} h_m^2(X_i,X_j)]^{1/2} |w_i-1|.
\end{eqnarray*}
For $\tau > 0$, we let $G_\tau = \{|h_m(X_i,X_j)| \le \tau, \forall i,j=1,\cdots,n; i \neq j; m=1,\cdots,d\}$.
%$\chi_{\tau,ij} = \vone(\max_{1 \le m \le d} |h_m(X_i,X_j)| > \tau)$
Then, on the event $G_\tau$, we have $M \le \tau \max_{1 \le i \le n} |w_i-1|^2$ and $\tilde{M} \le n^{1/2} \tau \max_{1 \le i \le n} |w_i-1|$. By Pisier's inequality, we have $\E[\max_{1 \le i \le n} |w_i-1|^a] \le C(a) (\log n)^a$ for any $a \ge 1$. Therefore, we obtain that, on the event $\{ \min_{1 \le j \le d} \breve{\Gamma}_{n,jj} \ge \ub/2 \} \cap G_\tau$, there is a constant $C_3 > 0$ depending only on $\ub$ such that
\begin{eqnarray}
\nonumber
\rho^{re}(L_n, T_n^\diamond \mid X_1^n) &\le& C_3 \left({\log d \over n}\right)^{3/4} \{ (\log (nd))^{3/2} \tau^{1/2} + A_1^{1/2} \\
&& \qquad + (\log d)^{1/8} A_2^{1/2} + n^{1/4} (\log (nd))^{3/4} \tau^{1/2} \}.
\label{eqn:Ln_Tn_bound_iid_weights}
\end{eqnarray}

{\bf Case (i).} Choose $\tau = [n B_n / \log(nd)]^{1/3}$. Then, by (E.1) and the union bound
$$
\Prob(G_\tau^c) \le (2 d n^2) \exp\left(-{\tau \over B_n} \right) = (2 d n^2) \exp\left(-{\log(nd) \over (n^{-1} B_n^2 \log^4(nd))^{1/3}} \right).
$$
By the assumption (\ref{eqn:eb_reduction1}) and $\log(1/\gamma) \le K \log(dn)$, we get $\Prob(G_\tau^c) \le \gamma/6$. Since $A_1 = n \hat{D}_2^{1/2}$, where $\hat{D}_2$ is in Lemma \ref{lem:bounds_on_terms_empirical-version-EB}, we have $\Prob(A_1 \ge C_4 n B_n^{1/3}) \le \gamma/6$. By the Cauchy-Schwarz inequality, $A_2 \le n^{3/4} \hat{D}_4^{1/4}$, where $\hat{D}_4$ is in Lemma \ref{lem:bounds_on_terms_empirical-version-EB}. By a second application of Lemma \ref{lem:bounds_on_terms_empirical-version-EB}, we get $\Prob(A_2 \ge C_5 n^{3/4} B_n^{1/2} \log^{1/4}(dn)) \le \gamma/6$. Therefore, it follows from (\ref{eqn:Ln_Tn_bound_iid_weights}) that there exists a constant $C_6 > 0$ depending only on $\ub,\bar{b},K$ such that
$$
\rho^{re}(L_n, T_n^\diamond \mid X_1^n) \le C_6  \left\{ {B_n^{1/6} \log^{3/4} d \over n^{1/4}} + {B_n^{1/4} \log(dn) \over n^{3/8}} + {\log^{3/2} (dn) \over n^{1/2} } \tau^{1/2} \right\}
$$
holds with probability at least $1-4\gamma/6$. Elementary calculations show that $\rho^{re}(L_n, T_n^\diamond \mid X_1^n) \le C_6 \varpi_{1,n}$. Thus, we have $\rho^B(T_n, T_n^\diamond) \le C_7 \varpi_{1,n}$ with probability at least $1-\gamma$.

{\bf Case (ii).} Choose $\tau = [n B_n / \log(nd)]^{1/3} + (\gamma/6)^{-1/q} n^{2/q} B_n$. Then, by (E.2) and the union bound $\Prob(G_\tau^c) \le n^2 B_n^q \tau^{-q} \le \gamma/6$. By Lemma \ref{lem:bounds_on_terms_empirical-version-EB}, we have $\Prob(A_1 \ge C_4 [ n B_n^{1/3} + n^{1/2+1/q} B_n \gamma^{-1/q} \log^{1/2} d]) \le \gamma/6$ and $\Prob(A_2 \ge C_5 [n^{3/4} B_n^{1/2} + n^{1/2+1/q} B_n \gamma^{-1/q} \log^{1/4} d] ) \le \gamma/6$. Therefore, it follows from (\ref{eqn:Ln_Tn_bound_iid_weights}) that there exists a constant $C_6 > 0$ depending only on $\ub,\bar{b},q,K$ such that
\begin{eqnarray*}
\rho^{re}(L_n, T_n^\diamond \mid X_1^n) &\le& C_6  \Big\{ {B_n^{1/6} \log^{3/4} d \over n^{1/4}} + {B_n^{1/4} \log(dn) \over n^{3/8}} \\
&& \qquad + {B_n^{1/2} \log d \over \gamma^{1/(2q)} n^{1/2-1/(2q)}} + {\log^{3/2} d \over n^{1/2} } \tau^{1/2} \Big\}
\end{eqnarray*}
holds with probability at least $1-4\gamma/6$. Elementary calculations show that $\rho^{re}(L_n, T_n^\diamond \mid X_1^n) \le C_6 \{\varpi_{1,n} + \varpi_{2,n}^B(\gamma)\}$. Thus, we have $\rho^B(T_n, T_n^\diamond) \le C_7 \{\varpi_{1,n} + \varpi_{2,n}^B(\gamma)\}$ with probability at least $1-\gamma$.
\end{proof}

\begin{proof}[Proof of Theorem \ref{thm:iid-weighted-bootstrap-no-centering}]
Observe that we can write $T_n^\flat = L_n + R_n$, where 
$$
L_n = {1 \over \sqrt{n}} \sum_{i=1}^n [{1 \over n-1} \sum_{j \neq i} h(X_i,X_j) - U_n ] (w_i-1)
$$
and $R_n$ is defined in (\ref{eqn:Rn_iid-weighted-bootstrap}). Then, Theorem \ref{thm:iid-weighted-bootstrap-no-centering} follows from Theorem \ref{thm:CLT-hyperrec-momconds} and the argument in Proposition \ref{prop:gaussian_comp_leading_term-EB}.
\end{proof}

\begin{proof}[Proof of Theorem \ref{thm:mb-bootstrap-rate}]
Let $Y \sim N(0,\Gamma)$. By the triangle inequality, $\rho^B(T_n, T_n^\sharp) \le \rho^{re}(T_n,Y) + \sup_{A \in {\cal A}^{re}} |\Prob(Y \in A) - \Prob(T_n^\sharp \in A \mid X_1^n)|$. Then, Theorem \ref{thm:mb-bootstrap-rate} follows from Theorem \ref{thm:CLT-hyperrec-momconds} and the proof of Proposition \ref{prop:gaussian_comp_leading_term-EB} since $|\hat\Gamma_n-\Gamma|_\infty$ and $|\tilde\Gamma_n-\Gamma|_\infty$ obey the same large probability bound.
\end{proof}

Recall that $Y \sim N(0, \Gamma)$ and $\bar{Y} = \max_{1 \le k \le d} Y_k$. Without loss of generality, we assume that $Y$ is independent of $X_1^n$. Let $\rho_\ominus(\alpha) = \Prob(\{\bar{T}_n \le a_{\bar{T}_n^\sharp}(\alpha) \} \ominus \{ \bar{T}_n \le a_{\bar{Y}}(\alpha) \} )$, where $a_{\bar{Y}}(\alpha)$ is the $\alpha$-th quantile of $\bar{Y}$ and  and $A \ominus B = (A \setminus B) \cup (B \setminus A)$ is the symmetric difference of two subsets $A$ and $B$. 

\begin{lem}
\label{lem:bound-on-rho_ominus}
Assume (M.1). Then, there exists a constant $C > 0$ depending only on $\ub$ such that for every $\alpha \in (0,1)$ and $v > 0$, we have
\begin{equation*}
\rho_\ominus(\alpha) \le 2 \left[ \rho(\bar{T}_n, \bar{Y}) + C v^{1/3} (\log{d})^{2/3} + \Prob(\Delta_n > v) \right],
\end{equation*}
where $\Delta_n = |\tilde\Gamma_n - \Gamma|_\infty$ and $\tilde\Gamma_n$ is defined in (\ref{eqn:tilde_Gamma_n}). 
\end{lem}

\begin{proof}[Proof of Lemma \ref{lem:bound-on-rho_ominus}]
The proof is a modification (and an improvement) of \cite[Theorem 3.1]{cck2013}. For the sake of completeness, we give the details. Applying Lemma \ref{lem:gauss-comp-general} to the max-hyperrectangles, there exists a constant $C(\ub) > 0$ such that for every $v > 0$, we have 
\begin{equation*}
\sup_{t \in \R} | \Prob(\bar{T}_n^\sharp \le t | X_1^n) - \Prob(\bar{Y} \le t)| \le C(\ub) (\log d)^{2/3} v^{1/3} := \pi(v)
\end{equation*}
on the event $\{ \Delta_n \le v \}$. So on this event
$$
\Prob_e( \bar{T}_n^\sharp \le a_{\bar{Y}}(\alpha + \pi(v)) ) \ge \Prob( \bar{Y} \le a_{\bar{Y}}(\alpha + \pi(v)) ) - \pi(v) \ge  \alpha + \pi(v) - \pi(v) = \alpha,
$$
where $\Prob_e$ is the probability taken w.r.t. to the Gaussian multiplier random variables $e_1,\cdots,e_n$. This implies that 
$$
\Prob( a_{\bar{T}_n^\sharp}(\alpha) \le a_{\bar{Y}}(\alpha + \pi(v)) ) \ge \Prob( \Delta_n \le v ).
$$
Similarly, since $Y$ is independent of $X_1^n$, we also have $\Prob( a_{\bar{Y}}(\alpha) \le a_{\bar{T}_n^\sharp}(\alpha + \pi(v)) ) \ge \Prob( \Delta_n \le v )$. Then, the Lemma follows from 
\begin{eqnarray*}
\rho_\ominus(\alpha) &\le& \Prob(a_{\bar{Y}}(\alpha - \pi(v)) <  \bar{T}_n \le a_{\bar{Y}}(\alpha + \pi(v)) ) + 2 \Prob(\Delta_n > v) \\
&\le& \Prob(a_{\bar{Y}}(\alpha - \pi(v)) <  \bar{Y} \le a_{\bar{Y}}(\alpha + \pi(v)) ) + 2 \Prob(\Delta_n > v) + 2 \rho(\bar{T}_n, \bar{Y}) \\
&\le& 2 \pi(v) + 2 \Prob(\Delta_n > v) + 2 \rho(\bar{T}_n, \bar{Y}),
\end{eqnarray*}
where the last step is due to the fact that $\bar{Y}$ has no point mass.
\end{proof}

\begin{proof}[Proof of Theorem \ref{thm:comparison_with_naive_gaussian_wild_bootstrap}]
Let $\Delta_n = |\tilde\Gamma_n - \Gamma|_\infty$. By Lemma \ref{lem:bound-on-rho_ominus}, there exists a constant $C > 0$ depending only on $\ub$ such that for every $\alpha \in (0,1)$ and $v > 0$, we have
\begin{eqnarray*}
| \Prob( \bar{T}_n \le a_{\bar{T}_n^\sharp}(\alpha) ) - \alpha | &\le&  | \Prob( \bar{T}_n \le a_{\bar{T}_n^\sharp}(\alpha) ) - \Prob( \bar{T}_n \le a_{\bar{Y}}(\alpha) ) | + \rho(\bar{T}_n, \bar{Y}) \\
&\le& \rho_\ominus(\alpha) + \rho(\bar{T}_n, \bar{Y}) \\
&\le& C v^{1/3} (\log{d})^{2/3} + 2 \Prob(\Delta_n > v) + 3 \rho(\bar{T}_n, \bar{Y}).
\end{eqnarray*}
{\bf Case (i).} Since $|\tilde\Gamma_n-\Gamma|_\infty$ obeys the same probability bound as $|\hat\Gamma_n-\Gamma|_\infty$ in (\ref{eqn:tail_prob_bound_hatGamma_n}), there exists an absolute constant $K_1 > 0$ such that $\Prob( \Delta_n > v ) \le n^{-1}$, where $v = K_1 [n^{-1} B_n^2 \log^3(n d)]^{1/2}$. Then (\ref{eqn:comparison_with_naive_gaussian_wild_bootstrap}) follows from Theorem \ref{thm:CLT-hyperrec-momconds} and the assumption $B_n^2 \log^7(n d) \le n^{1-K}$.\\

\noindent {\bf Case (ii).} Again, $|\tilde\Gamma_n-\Gamma|_\infty$ obeys the same probability bound as $|\hat\Gamma_n-\Gamma|_\infty$, which is (\ref{eqn:t_star_poly}) under (E.2). Therefore, there exists a constant $C_1(q) > 0$ such that $\Prob( \Delta_n > v ) \le n^{-K/6}$, where $v = C_1(q) \{ [n^{-1} B_n^2 \log^3(n d)]^{1/2} + n^{-1+2(1+K/6)/q} B_n^2 \log(d) \}$. Then (\ref{eqn:comparison_with_naive_gaussian_wild_bootstrap}) follows from Theorem \ref{thm:CLT-hyperrec-momconds} and the assumptions that $B_n^2 \log^7(n d) \le n^{1-K}$ and $B_n^4 \log^6(d) \le n^{2-4(1+K/6)/q-K}$.
\end{proof}

\section{A maximal inequality for completely degenerate weighted quadratic forms}

\begin{lem}[A maximal inequality for completely degenerate weighted quadratic forms]
\label{lem:maximal_ineq_weighted_quad_forms}
Let $\va_{ij} \in \R^d$ be $d$-dimensional fixed vectors such that $\va_{ij} = \va_{ji}$ for all $i,j=1,\cdots,n$ and $i \neq j$. Let $w_1,\cdots,w_n$ be iid random variables such that $\E w_i = 0, \E w_i^2 = 1$, and $\E w_i^4 < \infty$. Then, there exists an absolute constant $K > 0$ such that
\begin{eqnarray}
\nonumber
\E |\sum_{1 \le i \neq j \le n} \va_{ij} w_i w_j |_\infty &\le& K (\log d) \big\{ (\log d) \|M\|_2 + A_1 \\ \label{eqn:maximal_ineq_weighted_quad_forms}
&& \qquad + (\log d)^{1/4} A_2 \|w_1\|_4 + (\log d)^{1/2} \|\tilde{M}\|_4 \big\},
\end{eqnarray} 
where
\begin{eqnarray*}
A_1 &=& \max_{1 \le m \le d} (\sum_{1 \le i \neq j \le n} a_{ij,m}^2)^{1/2}, \quad A_2 = \max_{1 \le m \le d} [\sum_{i=1}^n (\sum_{j \neq i} a_{ij,m}^2)^2 ]^{1/4}, \\ 
M &=& \max_{1 \le m \le d} \max_{1 \le i \neq j \le n} |a_{ij,m} w_i w_j|, \quad \tilde{M} = \max_{1 \le m \le d} \max_{1 \le i \le n} (\sum_{j\neq i} a_{ij,m}^2)^{1/2} |w_i|.
\end{eqnarray*} 
\end{lem}

\subsection{Proof of Lemma \ref{lem:maximal_ineq_weighted_quad_forms}}

Before proving Lemma \ref{lem:maximal_ineq_weighted_quad_forms}, we need some auxillary lemmas. Let $(X'_1,\cdots,X'_n)$ be an independent copy of $(X_1,\cdots,X_n)$.  Let $\varepsilon_1, \cdots, \varepsilon_n$ be a sequence of independent Rademacher random variables such that $\Prob(\varepsilon_i=\pm1)=1/2$. Let $(\varepsilon'_1,\cdots,\varepsilon'_n)$ be an independent copy of $(\varepsilon_1,\cdots,\varepsilon_n)$. Here, we assume that the underlying probability space is rich enough such that $X_1,\cdots,X_n,X'_1,\cdots,X'_n,\varepsilon_1, \cdots, \varepsilon_n,\varepsilon'_1, \cdots, \varepsilon'_n$ are all independent.

\begin{lem}[Symmetrization for functions depending on the index]
\label{lem:symmetrization}
Let $q \ge 1$ and $X_1,\cdots,X_n$ be independent random variables in $S$. Let $h_{mi}:S\to \R$ be $\cal S$-measurable functions such that $\E|h_{mi}(X_i)|^q < \infty$ and $\E[h_{mi}(X_i)] = 0$ for all $m =1,\cdots,d$ and $i=1,\cdots,n$. Then we have
\begin{eqnarray*}
2^{-q} \E[\max_{1 \le m \le d} |\sum_{i=1}^n \varepsilon_i h_{mi}(X_i)|^q] \le \E[\max_{1 \le m \le d} |\sum_{i=1}^n h_{mi}(X_i)|^q] \le 2^q \E[\max_{1 \le m \le d} |\sum_{i=1}^n \varepsilon_i h_{mi}(X_i)|^q].
\end{eqnarray*} 
\end{lem}

\begin{proof}[Proof of Lemma \ref{lem:symmetrization}]
Since the random vectors $\vh_i(X_i) - \vh_i(X'_i) = \{h_{1i}(X_i)-h_{1i}(X'_i),\cdots,h_{di}(X_i)-h_{di}(X'_i)\}$ for $i=1,\cdots,n$ are independent and symmetric, $\{\vh_i(X_i) - \vh_i(X'_i)\}_{i=1}^n$ has the same joint distribution as $\{\varepsilon_i [\vh_i(X_i) - \vh_i(X'_i)]\}_{i=1}^n$. Let $\E'$ be the expectation taken only w.r.t. the random variables $X'_1,\cdots,X'_n$. By the centering assumption and Jensen's inequality, we have
\begin{eqnarray*}
\E[\max_{1 \le m \le d} |\sum_{i=1}^n \varepsilon_i h_{mi}(X_i)|^q] &=& \E[\max_{1 \le m \le d} |\sum_{i=1}^n \varepsilon_i h_{mi}(X_i) - \E'[h_{mi}(X'_i)] |^q] \\
&\le& \E[ \max_{1 \le m \le d} |\sum_{i=1}^n \varepsilon_i [ h_{mi}(X_i) - h_{mi}(X'_i)] |^q ] \\
&=& \E[ \max_{1 \le m \le d} |\sum_{i=1}^n  [ h_{mi}(X_i) - h_{mi}(X'_i)] |^q ] \\
&\le& 2^q  \E[ \max_{1 \le m \le d} |\sum_{i=1}^n  h_{mi}(X_i) |^q ].
\end{eqnarray*}
The other half inequality follows similarly.
\end{proof}

\begin{lem}[Randomization for canonical kernels depending on the index]
\label{lem:randomization}
Let $q \ge 1$ and $X_1,\cdots,X_n$ be iid random variables in $S$. Let $h_{ij}:S \times S \to \R^d$ be symmetric ${\cal S} \otimes {\cal S}$-measurable functions such that for all $x_1, x_2 \in S$, $i,j=1,\cdots,n$ and $i \neq j$: (i) $h_{ij}(x_1, x_2) = h_{ji}(x_2, x_1)$ and $\E|h_{m,ij}(X_1,X_2)|^q < \infty$ for all $m =1,\cdots,d$; (ii) $\E[h_{ij}(X_1,x_2)] = \E[h_{ij}(x_1,X_2)] = 0$. Then we have
\begin{eqnarray*}
\E[\max_{1 \le m \le d} |\sum_{i \neq j} h_{m,ij}(X_i,X_j)|^q] &\cong_q& \E[\max_{1 \le m \le d} |\sum_{i \neq j} \varepsilon_i \varepsilon_j h_{m,ij}(X_i,X_j)|^q] \\
&\cong_q& \E[\max_{1 \le m \le d} |\sum_{i \neq j} \varepsilon_i \varepsilon'_j h_{m,ij}(X_i,X_j)|^q] \\
&\cong_q& \E[\max_{1 \le m \le d} |\sum_{i \neq j} \varepsilon_i \varepsilon'_j h_{m,ij}(X_i,X'_j)|^q].
\end{eqnarray*} 
where $a \cong_q b$ means that $c b \le a \le C b$ for some constants $c,C>0$ depending only on $q$.
\end{lem}

\begin{proof}[Proof of Lemma \ref{lem:randomization}]
The proof is a modification of \cite[Theorem 3.5.3]{delaPenaGine1999} for the special case $h_{ij} \equiv h$. For the sake of completeness, we write the proof. Let $E_1$ and $E_2$ be the expectations taken only w.r.t. $\varepsilon_1,\cdots,\varepsilon_n, X_1,\cdots,X_n$ and  $\varepsilon'_1,\cdots,\varepsilon'_n, X'_1,\cdots,X'_n$, respectively. Applying the Fubini theorem and Lemma \ref{lem:symmetrization} twice, we have
\begin{eqnarray*}
\E[\max_{1 \le m \le d} |\sum_{i \neq j} \varepsilon_i \varepsilon'_j h_{m,ij}(X_i,X'_j)|^q] &=& \E_1\{ \E_2 \max_{1 \le m \le d}  |\sum_{j=1}^n \varepsilon'_j (\sum_{i \neq j} \varepsilon_i  h_{m,ij}(X_i,X'_j)) |^q \} \\
&\cong_q& \E_1\{ \E_2 \max_{1 \le m \le d}  |\sum_{j=1}^n (\sum_{i \neq j} \varepsilon_i  h_{m,ij}(X_i,X'_j)) |^q \} \\
&=& \E_2\{ \E_1 \max_{1 \le m \le d}  |\sum_{i=1}^n \varepsilon_i  (\sum_{j \neq i} h_{m,ij}(X_i,X'_j)) |^q \} \\
&\cong_q& \E_2\{ \E_1 \max_{1 \le m \le d}  |\sum_{i=1}^n \sum_{j \neq i} h_{m,ij}(X_i,X'_j) |^q \} \\
&\cong_q& \E[ \max_{1 \le m \le d}  |\sum_{i \neq j} h_{m,ij}(X_i,X_j) |^q ],
\end{eqnarray*} 
where the last step follows from the decoupling inequality \cite[Theorem 3.1.1]{delaPenaGine1999} since $h_{m,ij}$ are symmtric kernels. So the first and last terms in Lemma \ref{lem:randomization} are on the same order. Then, the equivalence of the middle terms follows from several applications of the decoupling inequality.
\end{proof}

\begin{proof}[Proof of Lemma \ref{lem:maximal_ineq_weighted_quad_forms}]
By the symmetry of $\va_{ij}$ and the randomization inequality in Lemma \ref{lem:randomization}, we have
\begin{equation}
\label{eqn:maximal_ineq_quad_first_step_sym}
\E |\sum_{1 \le i \neq j \le n} \va_{ij} w_i w_j |_\infty  \le  K_1 \E |\sum_{1 \le i \neq j \le n} \va_{ij} w_i w_j \varepsilon_i \varepsilon_j |_\infty
\end{equation}
for some absolute constant $K_1 > 0$. Let $I = \E [\max_{1 \le m \le d} \sum_{1 \le i \neq j \le n} a_{ij,m}^2 w_i^2 w_j^2]$. By the argument in Theorem \ref{thm:expectation-bound}, we have
\begin{equation}
\label{eqn:maximal_ineq_quad_second_step}
\E |\sum_{1 \le i \neq j \le n} \va_{ij} w_i w_j \varepsilon_i \varepsilon_j |_\infty \le K_2 (\log d) \sqrt{I}.
\end{equation}
By the triangle inequality,
$$
I \le \max_m \sum_{i\neq j} a_{ij,m}^2 + 2 \E[\max_m |\sum_{i \neq j} a_{ij,m}^2 (w_i^2-1)| ] + \E [\max_m | \sum_{i \neq j} a_{ij,m}^2 (w_i^2-1) (w_j^2-1) |].
$$
Since $\{a_{ij,m}^2 (w_i^2-1) (w_j^2-1)\}_{m=1}^d$ is a completely degenerate kernel, by a second application of the randomization inequality in Lemma \ref{lem:randomization}, Jensen's inequality and the decoupling inequality \cite[Theorem 3.1.1]{delaPenaGine1999}, we deduce that
\begin{eqnarray*}
&& \E [\max_m | \sum_{i \neq j} a_{ij,m}^2 (w_i^2-1) (w_j^2-1) |] \\
&\le& K_3 \E [\max_m | \sum_{i \neq j} a_{ij,m}^2 (w_i^2-1) (w_j^2-1) \varepsilon_i \varepsilon_j |]  \\
&\le& K_3 \E [\max_m | \sum_{i \neq j} a_{ij,m}^2 w_i^2 w_j^2 \varepsilon_i \varepsilon_j |] .
\end{eqnarray*}
Let $\Lambda_m = \{a_{ij,m}^2 w_i^2 w_j^2\}_{1 \le i \neq j \le n}$ be an $n \times n$ matrix with zero diagonal entries. By the Hanson-Wright inequality \cite{rudelsonvershynin2013a}, we have for each $m=1,\cdots,d$ that
$$
\Prob(|\sum_{i \neq j} a_{ij,m}^2 w_i^2 w_j^2 \varepsilon_i \varepsilon_j| \ge t \mid w_1^n) \le 2 \exp\{-K_4 \min(t/T_1, t^2/T_2^2) \},
$$
where $T_1 = \|\Lambda_m\|_2$ is the spectral norm of $\Lambda_m$ and $T_2 = |\Lambda_m|_F$. Integrating this tail probability and by the union bound, we have
$$
\E [\max_m | \sum_{i \neq j} a_{ij,m}^2 w_i^2 w_j^2 \varepsilon_i \varepsilon_j |] \le K_5 (\log d) \E[ \max_m \sum_{i \neq j} a_{ij,m}^4 w_i^4 w_j^4 ]^{1/2}.
$$
Then, it follows from the Cauchy-Schwarz inequality that
$$
\E[ \max_m \sum_{i \neq j} a_{ij,m}^4 w_i^4 w_j^4 ]^{1/2} \le \|M\|_2 \sqrt{I}.
$$
By the Hoeffding inequality \cite[Lemma 2.2.7]{vandervaartwellner1996}, we have for each $m=1,\cdots,d$ that
\begin{eqnarray*}
\Prob( |\sum_{i=1}^n (\sum_{j \neq i} a_{ij,m}^2) (w_i^2-1) \varepsilon_i | \ge t \mid w_1^n) \le 2 \exp\left\{ - {t^2 \over 2 \sum_{i=1}^n (\sum_{j \neq i} a_{ij,m}^2)^2 (w_i^2-1)^2} \right\}
\end{eqnarray*}
holds for all $t > 0$. Then, by the symmetrization inequality in Lemma \ref{lem:symmetrization} and the union bound, we have
$$
\E[\max_m |\sum_{i \neq j} a_{ij,m}^2 (w_i^2-1)| ] \le K_6 (\log d)^{1/2}  \E[ \max_m \sum_{i=1}^n (\sum_{j \neq i} a_{ij,m}^2)^2 (w_i^2-1)^2 ]^{1/2}.
$$
By \cite[Lemma 9]{cck2014b}, we have
\begin{eqnarray*}
&& \E[ \max_m \sum_{i=1}^n (\sum_{j \neq i} a_{ij,m}^2)^2 (w_i^2-1)^2 ] \\
&\le& K_7 \{ \max_m \sum_{i=1}^n (\sum_{j \neq i} a_{ij,m}^2)^2 \|w_1\|_4^4 + (\log d) \E[\max_{m,i}  (\sum_{j \neq i} a_{ij,m}^2)^2 (w_i^2-1)^2 ] \} \\
&\le& K_7 \{A_2^4 \|w_1\|_4^4 + (\log d) \E(\tilde{M}^4) \}.
\end{eqnarray*}
Therefore, we obtain that
$$
I \le K_8 \{A_1^2  + (\log d)^{1/2} [A_2^2 \|w_1\|_4^2 + (\log d)^{1/2} \|\tilde{M}\|_4^2] + (\log d) \|M\|_2 \sqrt{I}  \}.
$$
Solving this quadratic equation for $I$, we get
$$
I \le K_9 \{ (\log d)^2 \|M\|_2^2 + A_1^2 + (\log d)^{1/2} A_2^2 \|w_1\|_4^2 + (\log d) \|\tilde{M}\|_4^2 \}.
$$
Then, (\ref{eqn:maximal_ineq_weighted_quad_forms}) follows from the last inequality combined with (\ref{eqn:maximal_ineq_quad_first_step_sym}) and (\ref{eqn:maximal_ineq_quad_second_step}).
\end{proof}

\section{Tail probability inequalities for the maxima of U-statistics}
\label{sec:conc-ineq-ustat}

In this section, we establish a collection of tail probability inequalities for the maxima of U-statistics of an arbitrary order $r=1,\cdots, n$ under various moment conditions. Let $X_1^n$ be iid random variables taking values in a measurable space $(S, {\cal S})$. Let $I_n^r = \{(i_1,i_2,\cdots,i_r) : 1 \le i_1 \neq i_2 \neq \cdots \neq i_r \le n\}$, $h: S^r \to \R^d$ be a measurable function and
$$
U_n = {(n-r)! \over n!} \sum_{(i_1,i_2,\cdots,i_r) \in I_n^r} h(X_{i_1},X_{i_2},\cdots,X_{i_r}).
$$
be the corresponding $U$-statistics of order $r$. In Appendix \ref{sec:conc-ineq-ustat}, we do not assume that $h$ is symmetric, centered and non-degenerate. Let $m = [n/r]$ be the largest integer no greater than $n/r$ and 
$$
V(X_1^n) = \sum_{i=0}^{m-1} [h(X_{ir+1}^{ir+r}) - \theta],
$$
where $\theta = \E[h(X_{ir+1}^{ir+r})]$ and $X_{ir+1}^{ir+r} = \{ X_{ir+1}, \cdots, X_{ir+r} \}$. Let $\pi_n$ be a permutation of $\{1,\cdots,n\}.$ As noted by \cite{hoeffding1963}, we can write the U-statistics as a dependent average of block sums of iid random variables over all possible permutations
\begin{equation}
\label{eqn:ustat-perm-rep}
m (U_n-\theta) = {1\over n!} \sum_{\all \pi_n} V(X_{\pi_n(1)}, \cdots, X_{\pi_n(n)}).
\end{equation}
The representation (\ref{eqn:ustat-perm-rep}) enables us to reduce the moment bounds on $m(U_n-\theta)$ to those of $V(X_1^n)$, whose size control has been well-understood in literature \cite{ledouxtalagrand1991,adamczak2008,kleinrio2005,massart2000,einmahlli2008}.

\subsection{Centered U-statistics}
\label{subsec:conc-ineq-ustat-centered}

We first deal with centered U-statistics $Z = m |U_n - \theta|_\infty$. Let $\tau >0$ and 
$$\bar{h}(x_1^r) = h(x_1^r) \vone(\max_{1\le j \le d}|h_j(x_1^r)| \le \tau)$$
be the truncated version of $h(x_1^r)$ at the level $\tau$. Define
\begin{eqnarray}
\label{eqn:Z_1_centered}
Z_1 &=& \max_{1 \le j \le d} \left| \sum_{i=0}^{m-1} [\bar{h}_j(X_{ir+1}^{ir+r}) - \E \bar{h}_j] \right|, \\
\label{eqn:M}
M &=&  \max_{1 \le j \le d} \max_{0 \le i \le m-1} |h_j(X_{ir+1}^{ir+r})|, \\
\label{eqn:zeta_bar_centered}
\bar\zeta_n^2 &=& \ \max_{1 \le j \le d} \sum_{i=0}^{m-1} \E h_j^2(X_{ir+1}^{ir+r}).
\end{eqnarray}

\begin{lem}[Sub-exponential inequality for the maxima of centered U-statistics]
\label{lem:subexp-concentration-ineq}
Let $X_1,\cdots,X_n$ be iid random variables taking values in $S$ and $\alpha \in (0,1]$. Suppose that $\|h_j(X_1^r)\|_{\psi_\alpha} < \infty$ for all $j=1,\cdots,d$. Let $\tau = 8 \E[M]$. Then, for any $0 < \eta \le 1$ and $\delta>0$, there exists a constant $C(\alpha,\eta,\delta)>0$ such that we have
\begin{equation}
\label{eqn:subexp-concentration-ineq}
\Prob(Z \ge (1+\eta) \E Z_1 + t)  \le \exp\left(-{t^2 \over 2 (1+\delta) \bar\zeta_n^2} \right) + 3 \exp\left[ -\left({t \over C(\alpha,\eta,\delta) \left\| M \right\|_{\psi_\alpha}} \right)^\alpha \right]
\end{equation}
holds for all $t > 0$.
\end{lem}

\begin{lem}[Fuk-Nagaev inequality for the maxima of centered U-statistics]
\label{lem:fuknagaev-concentration-ineq}
Let $X_1,\cdots,X_n$ be iid random variables taking values in $S$ and $q \ge 1$. Suppose that $\E[|h_j(X_1^r)|^q] < \infty$ for all $j=1,\cdots,d$. Choose $\tau = 4 \cdot 2^{1/q} \cdot \|M\|_q$. Then, for any $0 < \eta \le 1$ and $\delta>0$, there exists a constant $C(q,\eta,\delta)>0$ such that we have 
\begin{equation}
\label{eqn:fuknagaev-concentration-ineq}
\Prob(Z \ge (1+\eta) \E Z_1 + t)  \le \exp\left(-{t^2 \over 2 (1+\delta) \bar\zeta_n^2} \right) + C(q,\eta,\delta) {\E M^q \over t^q}
\end{equation}
holds for all $t > 0$.
\end{lem}

\subsection{Non-negative U-statistics}
\label{subsec:conc-ineq-ustat-nonnegative}

Next, we deal with non-negative U-statistics $Z = m \max_{1 \le j \le d} U_{nj}$, where the kernel $h \ge 0$. Let 
\begin{eqnarray}
\label{eqnZ_prime-nonnegative}
Z' = \max_{1 \le j \le d} \sum_{i=0}^{m-1} h_j(X_{ir+1}^{ir+r})
\end{eqnarray}
and $M$ be defined in (\ref{eqn:M}).

\begin{lem}[Sub-exponential inequality for the maxima of non-negative U-statistics]
\label{lem:subexp-concentration-ineq-nonnegative}
Let $X_1,\cdots,X_n$ be iid random variables taking values in $S$ and $\alpha \in (0,1]$. Suppose that $h \ge 0$ and $\|h_j(X_1^r)\|_{\psi_\alpha} < \infty$ for all $j=1,\cdots,d$. Then, for any $\eta>0$, there exists a constant $C(\alpha,\eta)>0$ such that we have
\begin{equation}
\label{eqn:subexp-concentration-ineq-nonnegative}
\Prob(Z \ge (1+\eta) \E Z' + t)  \le 3 \exp\left[ -\left({t \over C(\alpha,\eta) \left\| M \right\|_{\psi_\alpha}} \right)^\alpha \right]
\end{equation}
holds for all $t > 0$.
\end{lem}

\begin{lem}[Markov inequality for the maxima of non-negative U-statistics]
\label{lem:fuknagaev-concentration-ineq-nonnegative}
Let $X_1,\cdots,X_n$ be iid random variables taking values in $S$ and $q \ge 1$. Suppose that $h \ge 0$ and $E[h_j^q(X_1^r)] < \infty$ for all $j=1,\cdots,d$. Then, for any $\eta>0$, there exists a constant $C(q,\eta)>0$ such that we have
\begin{equation}
\label{eqn:fuknagaev-concentration-ineq-nonnegative}
\Prob(Z \ge (1+\eta) \E Z' + t)  \le C(q,\eta) {\E M^q \over t^q}
\end{equation}
holds for all $t > 0$.
\end{lem}

\subsection{Proof of results in Section \ref{sec:conc-ineq-ustat}}

\begin{proof}[Proof of Lemma \ref{lem:subexp-concentration-ineq}]
In this proof, the index $i$ implicitly runs from 0 to $m-1$ and the index $j$ runs from 1 to $d$. First, we may assume that 
$$
t \ge C_1(\alpha,\eta,\delta) \|M\|_{\psi_\alpha}
$$
because otherwise we can choose the constant $C(\alpha,\eta,\delta)$ in (\ref{eqn:subexp-concentration-ineq}) large enough, say $C(\alpha,\eta,\delta) \ge C_1(\alpha,\eta,\delta)$, to make (\ref{eqn:subexp-concentration-ineq}) trivially holds. Let $\underline{h}(x_1^r) = h(x_1^r) - \bar{h}(x_1^r)$, $V_1(X_1^n) = \sum_{i=0}^{m-1} [\bar{h}(X_{ir+1}^{ir+r}) - \E \bar{h}]$, and $V_2(X_1^n) = \sum_{i=0}^{m-1} [\underline{h}(X_{ir+1}^{ir+r}) - \E \underline{h}]$. Define 
\begin{equation*}
T_\ell = \left| {1\over n!} \sum_{\all \pi_n} V_\ell(X_{\pi_n(1)}, \cdots, X_{\pi_n(n)}) \right|_\infty, \qquad \ell = 1,2.
\end{equation*}
Denote $A_1 = \E Z_1$ and $A_2 = \E |\sum_{i=0}^{m-1} \underline{h}(X_{ir+1}^{ir+r})|_\infty$. Since $(n!)^{-1} \sum_{\all \pi_n}$ is a probability measure, by (\ref{eqn:ustat-perm-rep}), Jensen's inequality, and the iid assumption on $X_1,\cdots,X_n$, we have 
\begin{eqnarray*}
Z &\le& T_1+T_2, \\
\E Z &\ge& \E T_1 - 2 A_2, \\
\E T_1 &\le& A_1.
\end{eqnarray*}
Consider $1 \ge \eta > 0$ and $\varepsilon := \varepsilon(\delta) \in (0,1)$, whose value is to be determined. For all $t >0$, we have
\begin{eqnarray}
\label{eqn:ustat-cont-subexp-T1+T2-terms}
&& \Prob(Z \ge (1+\eta) A_1 + t) \le \Prob(T_1 \ge (1+\eta) A_1 + (1-\varepsilon) t) + \Prob(T_2 \ge \varepsilon t).
\end{eqnarray}
We first deal with $\Prob(T_2 \ge \varepsilon t)$. Since $\|h_j(X_1^r)\|_{\psi_\alpha} < \infty$, $\tau = 8 \E[M] \le C_2(\alpha) \left\| M \right\|_{\psi_\alpha}$.
By Markov's inequality, we have
$$
\Prob(\max_{0 \le k \le m-1} \max_{1 \le j \le d} |\sum_{i=0}^k \underline{h}_j(X_{ir+1}^{ir+r})| > 0) \le \Prob(M > \tau) \le 1/8.
$$
By the Hoffmann-J{\o}rgensen inequality \cite[Proposition 6.8, Equation (6.8)]{ledouxtalagrand1991}, we have $A_2 \le 8 \E[M]$, which in together with Jensen's inequality, we get
$$
\E|V_2(X_1^n)|_\infty = \E \left |\sum_{i=0}^{m-1} [\underline{h}(X_{ir+1}^{ir+r}) - \E \underline{h}] \right |_\infty \le 16 \E[M] \le 2 C_2(\alpha) \left\| M \right\|_{\psi_\alpha}.
$$
Since $\psi_\alpha$ is a quasi-norm for $0<\alpha\le1$, we have  by Jensen's inequality that 
$$
\left\| \max_{i,j} \left| [\underline{h}_j(X_{ir+1}^{ir+r}) -\E \underline{h}_j] \right| \right\|_{\psi_\alpha} \le C_3(\alpha) \left\| \max_{i,j} |\underline{h}_j(X_{ir+1}^{ir+r})| \right\|_{\psi_\alpha} \le C_3(\alpha) \left\| M  \right\|_{\psi_\alpha}.
$$
Now, by \cite[Theorem 6.21]{ledouxtalagrand1991} or \cite[Theorem 5]{adamczak2008}, we have
$$
\left\| \max_j \left|\sum_{i=0}^{m-1} [\underline{h}_j(X_{ir+1}^{ir+r}) - \E \underline{h}_j] \right| \right\|_{\psi_\alpha} \le C_4(\alpha) \left\| M \right\|_{\psi_\alpha}.
$$
By the Markov inequality,
$$
\Prob(T_2 \ge \varepsilon t) \le 2 \exp\left[ -\left( {\varepsilon t \over \|T_2\|_{\psi_\alpha}} \right)^\alpha \right],
$$
where
\begin{eqnarray*}
\|T_2\|_{\psi_\alpha} &=& \left\| \left| {1\over n!} \sum_{\all \pi_n} V_2(X_{\pi_n(1)}, \cdots, X_{\pi_n(n)}) \right|_\infty \right\|_{\psi_\alpha} \\
&\le& C_5(\alpha) \left\| \max_j \left|\sum_{i=0}^{m-1} [\underline{h}_j(X_{ir+1}^{ir+r}) - \E \underline{h}_j] \right| \right\|_{\psi_\alpha} \le C_6(\alpha) \left\| M \right\|_{\psi_\alpha},
\end{eqnarray*}
where the second last inequality follows from Jensen's inequality and the iid assumption, and $C_6(\alpha)=C_4(\alpha) C_5(\alpha)$. [In the second last inequality, we need to modify $\psi_\alpha$ to be linear near the origin $(0, (\alpha^{-1}-1)^{1/\alpha})$ so that the modified function $\psi_\alpha$ becomes nondecreasing and convex on $[0, \infty)$.] Therefore, we have
\begin{equation}
\label{eqn:ustat-cont-subexp-T2-term}
\Prob \left(T_2 \ge \varepsilon t \right) \le 2 \exp\left[ -\left( {\varepsilon t \over C_6(\alpha)  \| M \|_{\psi_\alpha}} \right)^\alpha \right].
\end{equation}

Next, we deal with the $T_1$ term in (\ref{eqn:ustat-cont-subexp-T1+T2-terms}). Let $L(\lambda) = \log \E[\exp(\lambda T_1/(2\tau))]$ be the logarithmic moment generating function of $T_1/(2\tau)$. By the convexity of the max norm and exponential function and using Jensen's inequality twice, we see that
\begin{eqnarray*}
L(\lambda) &\le& \log \E \left\{ {1\over n!} \sum_{\all \pi_n} \exp \left[ {\lambda \over 2\tau} \left| V_1(X_{\pi_n(1)}, \cdots, X_{\pi_n(n)}) \right|_\infty \right] \right\} \\
&=& \log \E \exp(\lambda Z_1/(2\tau)).
\end{eqnarray*}
Then it follows from \cite[Lemma 3.4]{kleinrio2005} that for $0 < \lambda < 2/3$, 
$$
L(\lambda) \le \lambda \E \left({Z_1\over 2\tau} \right) + \left[ 2 \E \left({Z_1\over 2\tau} \right) + \left({\zeta_n^2 \over 4 \tau^2} \right) \right] \cdot  {\lambda^2 \over 2 - 3 \lambda},
$$
where
$$
\zeta_n^2 = \max_{1 \le j \le d} \sum_{i=0}^{m-1} \E [\bar{h}_j(X_{ir+1}^{ir+r}) - \E \bar{h}_j]^2 \le \bar\zeta_n^2.
$$
By the Markov inequality and the Legendre transform of the function $\lambda \mapsto \lambda^2 / (2 - 3 \lambda)$ for $0 < \lambda < 2/3$, we get
$$
\Prob(T_1 \ge \E Z_1 + t) \le \exp \left[ -{t^2 \over 2 [{\bar\zeta}_n^2 + 2(2\tau)\E Z_1] + 3 (2\tau) t } \right].
$$
By \cite[Lemma 1]{adamczak2008}, we have for all $0 < \eta \le 1$ and $\delta>0$, there exists a constant $C_7(\eta,\delta)= 2 (1+\delta^{-1})(3+2\eta^{-1})$ such that for all $t>0$,
\begin{eqnarray*}
\Prob(T_1 \ge (1+\eta) \E Z_1+t) \le \exp\left[-{ t^2 \over 2(1+\delta) \bar\zeta_n^2} \right] + \exp\left[-{t \over C_7(\eta,\delta) \tau} \right].
\end{eqnarray*}
Then, we have for all $t>0$
\begin{equation}
\label{eqn:ustat-cont-subexp-T1-term}
\Prob(T_1 \ge (1+\eta) A_1 + (1- \varepsilon) t) \le \exp\left[-{ (1-\varepsilon)^2 t^2 \over 2(1+\delta/2) \bar\zeta_n^2} \right] + \exp\left[-{ (1-\varepsilon) t \over C_7(\eta,\delta/2) \tau} \right].
\end{equation}
Choose $\varepsilon := \varepsilon(\delta) < 1/2$ small enough such that $(1-\varepsilon)^{-2} (1+\delta/2) \le 1+ \delta$.  By (\ref{eqn:ustat-cont-subexp-T1+T2-terms}), (\ref{eqn:ustat-cont-subexp-T2-term}), and (\ref{eqn:ustat-cont-subexp-T1-term}), we obtain that
\begin{eqnarray*}
&& \Prob(Z \ge (1+\eta) A_1+ t)  \\
&\le& 2 \exp\left[ -\left( {\varepsilon t \over C_6(\alpha)  \left\|M\right\|_{\psi_\alpha}} \right)^\alpha \right] + \exp\left[-{ (1-\varepsilon)^2 t^2 \over 2(1+\delta/2) \bar\zeta_n^2} \right] + \exp\left[-{ (1-\varepsilon) t \over C_7(\eta,\delta/2) \tau} \right] \\
&\le& 3 \exp\left[ -\left( {\varepsilon t \over C_8(\alpha,\eta,\delta)  \left\|M\right\|_{\psi_\alpha}} \right)^\alpha \right] + \exp\left[-{ t^2 \over 2(1+\delta) \bar\zeta_n^2} \right],
\end{eqnarray*}
where the last step follows from $\tau \le C_2(\alpha) \|M\|_{\psi_\alpha}$ and the assumption that $t \ge C_1(\alpha,\eta,\delta) \|M\|_{\psi_\alpha}$.
\end{proof}

\begin{proof}[Proof of Lemma \ref{lem:fuknagaev-concentration-ineq}]
We shall use the same notations as in the proof of Lemma \ref{lem:subexp-concentration-ineq}. First, we may assume that 
$$
t \ge C_1(q,\eta,\delta) \|M\|_q
$$
because otherwise we can choose the constant $C(q,\eta,\delta)$ in (\ref{eqn:fuknagaev-concentration-ineq}) large enough to make this Lemma trivially holds. Arguing similarly as in the proof of Lemma \ref{lem:subexp-concentration-ineq}, we also have 
$$\Prob(Z \ge (1+\eta) A_1 + t) \le \Prob(T_1 \ge (1+\eta) A_1 + (1-\varepsilon) t) + \Prob(T_2 \ge \varepsilon t)$$
for some properly chosen $\varepsilon := \varepsilon(\delta) \in (0,1)$. By Markov's inequality,
$$
\Prob(\max_{0 \le k \le m-1} \max_{1 \le j \le d} |\sum_{i=0}^k \underline{h}_j(X_{ir+1}^{ir+r})| > 0) \le \Prob(M > \tau) \le {\E M^q \over \tau^q} = {1 \over 2 \cdot 4^q}.
$$
Then, it follows from the Hoffmann-J{\o}rgensen inequality \cite[Proposition 6.8, Equation (6.8)]{ledouxtalagrand1991} that 
$$
\E |\sum_{i=0}^{m-1} \underline{h}(X_{ir+1}^{ir+r})|_\infty^q \le 2 \cdot 4^q \E[M^q].
$$
By Markov's and Jensen's inequalities, we have
\begin{eqnarray*}
\Prob(T_2 \ge \varepsilon t) \le {\E T_2^q \over \varepsilon^q t^q} \le {\E |\sum_{i=0}^{m-1} \underline{h}(X_{ir+1}^{ir+r}) - \E \underline{h}|_\infty^q \over \varepsilon^q t^q} \le {2 \cdot 8^q \E[M^q] \over \varepsilon^q t^q}.
\end{eqnarray*}
By the argument in Lemma \ref{lem:subexp-concentration-ineq} leading to (\ref{eqn:ustat-cont-subexp-T1-term}), we get
$$
\Prob(T_1 \ge (1+\eta) A_1 + (1- \varepsilon) t) \le \exp\left[-{ (1-\varepsilon)^2 t^2 \over 2(1+\delta/2) \bar\zeta_n^2} \right] + \exp\left[-{ (1-\varepsilon) t \over C_2(q, \eta,\delta) \|M\|_q} \right].
$$
Choosing $\varepsilon := \varepsilon(\delta) < 1/2$ small enough such that $(1-\varepsilon)^{-2} (1+\delta/2) \le 1+ \delta$ and using $e^{-x} \le C_3(q) x^{-q}$, we conclude (\ref{eqn:fuknagaev-concentration-ineq}).
\end{proof}

\begin{proof}[Proof of Lemma \ref{lem:subexp-concentration-ineq-nonnegative}]
The proof is a modification of Lemma \ref{lem:subexp-concentration-ineq}. We may assume that $t/4 \ge (1+\eta) 8 \E[M]$ because otherwise since $8\E[M] \le C_1(\alpha)\|M\|_{\psi_\alpha}$, we can choose $C(\alpha,\eta)$ in (\ref{eqn:subexp-concentration-ineq-nonnegative}) large enough to make this Lemma trivially hold. Let $M$ be defined in (\ref{eqn:M}), $\tau = 8 \E[M]$, and 
$$\bar{h}(x_1^r) = h(x_1^r) \vone(\max_{1\le j \le d} h_j(x_1^r) \le \tau)$$
be the truncated version of the non-negative kernel $h(x_1^r)$ at the level $\tau$. Let $\underline{h}(x_1^r) = h(x_1^r) - \bar{h}(x_1^r) \ge 0$, $V_1(X_1^n) = \sum_{i=0}^{m-1} \bar{h}(X_{ir+1}^{ir+r})$, and $V_2(X_1^n) = \sum_{i=0}^{m-1} \underline{h}(X_{ir+1}^{ir+r})$. For $\ell = 1,2$, define  
\begin{eqnarray}
\label{eqnZ_1-nonnegative}
Z_\ell = \max_{1 \le j \le d} \sum_{i=0}^{m-1} V_{\ell j}(X_1^n)
\end{eqnarray}
and 
$$
T_\ell = \max_{1 \le j \le d} {1 \over n!} \sum_{\all \pi_n} V_{\ell j}(X_{\pi_n(1)}, \cdots, X_{\pi_n(n)}).
$$
So $Z \le T_1 + T_2$. Since by the triangle inequality $Z_1-Z_2 \le Z' \le Z_1+Z_2$, we have
\begin{eqnarray*}
\Prob(Z \ge (1+\eta) \E Z' + t) \le \Prob(T_1 \ge (1+\eta) \E Z_1 - (1+\eta) \E Z_2 + 3t/4) + \Prob(T_2 \ge t/4).
\end{eqnarray*}
Arguing as in the proof of Lemma \ref{lem:subexp-concentration-ineq}, we have
$$
\Prob(T_2 \ge t/4) \le 2 \exp\left[ -\left( { t \over 4 \|T_2\|_{\psi_\alpha}} \right)^\alpha \right]
$$
and
\begin{eqnarray*}
\|T_2\|_{\psi_\alpha} &\le& C_2(\alpha) \left\| Z_2 \right\|_{\psi_\alpha} \le C_3(\alpha) \{ \E Z_2 + \left\| M \right\|_{\psi_\alpha} \},
\end{eqnarray*}
where the first inequality follows from Jensen's inequality, iid assumption, and the fact that $\psi_\alpha$ is a quasi-norm, and the second inequality from the non-negative version of \cite[Theorem 6.21]{ledouxtalagrand1991}. Since
$$
\Prob(\max_{0 \le k \le m-1} \max_{1 \le j \le d} |\sum_{i=0}^k \underline{h}_j(X_{ir+1}^{ir+r})| > 0) \le \Prob(M > \tau) \le 1/8.
$$
By the Hoffmann-J{\o}rgensen inequality \cite[Proposition 6.8, Equation (6.8)]{ledouxtalagrand1991}, we have
$$
\E Z_2 \le 8 \E[M] = \tau \le \min \left\{ {t \over 4(1+\eta)}, \; C_1(\alpha) \|M\|_{\psi_\alpha} \right\}.
$$
So we obtain that
$$
\Prob(T_2 \ge t/4) \le 2 \exp\left[ -\left( {t \over C_4(\alpha) \|M\|_{\psi_\alpha}} \right)^\alpha \right]
$$
and
$$
\Prob(T_1 \ge (1+\eta) \E Z_1 - (1+\eta) \E Z_2 + 3t/4) \le \Prob(T_1 \ge (1+\eta) \E Z_1 + t/2). 
$$
Next, we bound $\Prob(T_1 \ge (1+\eta)\E Z_1 + t)$. Let $L(\lambda)=\log \E \exp[\lambda (T_1-\E Z_1)/\tau]$. By Jensen's inequality and \cite[Theorem 10]{massart2000}, we have
$$
L(\lambda) \le \varphi(\lambda) \E (Z_1 / \tau), \qquad \text{where } \varphi(\lambda) = e^\lambda-\lambda-1.
$$
Then, it follows from the argument in \cite[Lemma E.5]{cck2015a} that for all $t > 0$
\begin{equation}
\label{eqn:T_1-bound-nonnegative}
\Prob(T_1 \ge (1+\eta) \E Z_1 + t) \le e^{-C_5(\eta) t/\tau},
\end{equation}
where $C_5(\eta) = (\eta^{-1}+2/3)^{-1}$. Then (\ref{eqn:subexp-concentration-ineq-nonnegative}) follows from $\tau \le C_1(\alpha) \|M\|_{\psi_\alpha}$ and $t \ge 4(1+\eta) \tau$.
\end{proof}

\begin{proof}[Proof of Lemma \ref{lem:fuknagaev-concentration-ineq-nonnegative}]
The proof is a modification of Lemma \ref{lem:fuknagaev-concentration-ineq} and \ref{lem:subexp-concentration-ineq-nonnegative}. We shall use the same notations as in the proof of Lemma \ref{lem:subexp-concentration-ineq-nonnegative}. Let $\tau = 4 \cdot 2^{1/q} \cdot \|M\|_q$. First, we may assume that 
$$
t/4 \ge (1+\eta) 4 \cdot 2^{1/q} \cdot \|M\|_q = (1+\eta) \tau.
$$
By Jensen's and Markov's inequalities, we have $\E T_2^q \le \E Z_2^q$ and
$$
\Prob(\max_{0 \le k \le m-1} \max_{1 \le j \le d} |\sum_{i=0}^k \underline{h}_j(X_{ir+1}^{ir+r})| > 0) \le \Prob(M > \tau) \le {\E [M^q] \over \tau^q} = {1 \over 2 \cdot 4^q}.
$$
By the Hoffmann-J{\o}rgensen inequality \cite[Proposition 6.8, Equation (6.8)]{ledouxtalagrand1991}, we have
$$
\E Z_2^q \le 2 \cdot 4^q \E[M^q] \le \left[{t \over 4(1+\eta)}\right]^q
$$
so that $\E Z_2 \le t/[4(1+\eta)]$. Therefore, we obtain that
$$
\Prob(T_2 \ge t/4) \le {\E T_2^q \over (t/4)^q} \le 2 \cdot 16^q {\E[M^q] \over t^q} 
$$
and
$$
\Prob(T_1 \ge (1+\eta) \E Z_1 - (1+\eta) \E Z_2 + 3t/4) \le \Prob(T_1 \ge (1+\eta) \E Z_1 + t/2).
$$
Using (\ref{eqn:T_1-bound-nonnegative}), we conclude that
$$
\Prob(Z \ge (1+\eta) \E Z' + t) \le e^{-C_1(\eta) t/\tau} + C_2(q) {\|M\|_q^q \over t^q},
$$
from which (\ref{eqn:fuknagaev-concentration-ineq-nonnegative}) follows.
\end{proof}

\section{Additional proofs in Section 5.4}
\label{appendix:additional_proofs_sec5.4}

\begin{lem}
\label{lem:moment-bounds-gaussian-obs}
Suppose that $X$ and $X'$ are iid mean zero random vectors in $\R^p$ such that $X_m \sim \text{subgaussian}(\nu^2)$ for all $m=1,\cdots,p$. If $h$ is the covariance matrix kernel, then we have for all $m,k=1,\cdots,p,$
\begin{eqnarray*}
%\|M\|_{\psi_1} &\le& C \nu^2 \log(np), \\
%u(\gamma) &\le& C \nu^2 \log^2(np).
\E[\exp(|h_{mk}(X, X')| / \nu^2)] \le 2,
%\| h_{mk} \|_{\psi_1} \le C_0 \nu^2.
\end{eqnarray*}
i.e. $\|h_{mk}(X, X')\|_{\psi_1} \le \nu^2$.
\end{lem}

\begin{proof}[Proof of Lemma \ref{lem:moment-bounds-gaussian-obs}]
The lemma follows from direct calculations
\begin{eqnarray*}
\E \left[\exp \left({|h_{mk}(X,X')| \over \nu^2} \right) \right] &=& \E\left[\exp \left({1\over2} {|X_m-X'_m| \over  \nu} {|X_k-X'_k | \over \nu} \right) \right] \\
&\le& \E \left[ \exp \left(  {(X_m-X'_m)^2 \over 4 \nu^2 } +  {(X_k-X'_k)^2 \over 4 \nu^2 }  \right) \right] \\
&\le& \max_{1\le m \le p} \E \left[ \exp \left( {(X_m-X'_m)^2 \over 2 \nu^2 } \right)\right] \\
&\le& \max_{1\le m \le p} \E \left[ \exp \left( {X_m^2 + {X'_m}^2 \over \nu^2 } \right)\right] \\
&\le& \max_{1\le m \le p} \left\{ \E \left[ \exp \left( {X_m^2 \over \nu^2 } \right)\right] \right\}^2 \le 2,
\end{eqnarray*}
where we used the elementary inequality $|ab| \le (a^2+b^2)/2$ in the second step, the Cauchy-Schwarz inequality in the third step, $(a-b)^2 \le 2(a^2+b^2)$ in the fourth step, the iid assumption in the fifth step, and the assumption that $X_m \sim \text{subgaussian}(\nu^2)$ in the last step. 
\end{proof}

\begin{lem}
\label{lem:moment-bounds-unifpoly-obs}
Let $q\ge8$. Suppose that $X$ and $X'$ are iid mean zero random vectors such that $\|\max_{1 \le m \le p} |X_m| \|_q \le \nu$. If $h$ is the covariance matrix kernel, then we have
\begin{eqnarray*}
%\|M\|_{\psi_1} &\le& C \nu^2 \log(np), \\
%u(\gamma) &\le& C \nu^2 \log^2(np).
\E [\max_{1 \le m,k \le p} |h_{mk}(X,X')| / (2 \nu^2)]^4 \le 1.
%\| h_{mk} \|_{\psi_1} \le C_0 \nu^2.
\end{eqnarray*}
\end{lem}

\begin{proof}[Proof of Lemma \ref{lem:moment-bounds-unifpoly-obs}]
The lemma follows from direct calculations
\begin{eqnarray*}
&& \E ( \max_{1 \le m,k \le p} |h_{mk}(X,X')| / \nu^2 )^4 \\
&\le& \left[ \E(\max_{1 \le m,k \le p} |h_{mk}(X,X')| / \nu^2)^{q/2} \right]^{8/q} \\
&=& 2^{-4} \left[ \E \left(\max_{1 \le m \le p} |X_m-X'_m|^q / \nu^q \right) \right]^{8/q} \\
&\le& 2^{-4} \left[ 2^q \E (\max_{1 \le m \le p} |X_m|^q / \nu^q) \right]^{8/q} \le 2^4.
\end{eqnarray*}
\end{proof}

\begin{proof}[Proof of Theorem \ref{thm:thresholded_cov_mat_rate_adaptive_polymom}]
The proof is similar to that of Theorem \ref{thm:thresholded_cov_mat_rate_adaptive} and we only sketch the differences. By the assumptions and Lemma \ref{lem:moment-bounds-unifpoly-obs}, we have
$$
\max_{\ell=1,2} \E[|h_{mk}|^{2+\ell} / (C \nu_n^{2\ell})] \vee \E [\|h\| / (2\nu_n^2)]^4 \le 1.
$$
By Theorem \ref{thm:comparison_with_naive_gaussian_wild_bootstrap}, we have $ |\hat{S}_n-\Sigma|_\infty \le a_{\bar{T}_n^\sharp}(1-\alpha)$ with probability at least $1-\alpha-Cn^{-7K/6}$, where $C > 0$ is a constant depending only on $C_i, i=1,\cdots,4$. So (\ref{eqn:thresholded_cov_mat_rate_spectral_adaptive}) and (\ref{eqn:thresholded_cov_mat_rate_F_adaptive}) follow. Note that
$$
\|\max_{m,k}\max_{i \le \ell} |h_{mk}(X_i, X_{i+\ell})| \|_4 \le K_1 n^{1/4} \nu_n^2.
$$
Then, under the assumption that $\nu_n^8 \log^7(np) \le C_4 n^{1-7K/6}$, it follows that
$$
\E[\bar\psi^2] \le C \Big\{ \xi_4^4 +  \xi_8^4 \Big( {\log{p} \over n} \Big)^{1/2} + \nu_n^4 {\log(p) \over n^{3/4}} \Big\}.
$$
Therefore, we get $\E[\tau_*] \le C(\alpha, C_1,\cdots,C_5) \beta^{-1} \xi_4^2 (\log(p)/n)^{1/2}$.
\end{proof}

\section{Numerical comparisons}
\label{sec:numerical_comparison}

We present some numerical comparisons of the Gaussian approximation of the U-statistic with covariance matrix kernel. We consider two mean-zero data distributions from the elliptical family \cite{muirhead1982}:
\begin{enumerate}
\item[(M1)] (sub-exponential moment) The $\varepsilon$-contaminated $p$-variate elliptical normal distribution with density function
\begin{eqnarray}
\nonumber
f(x; \varepsilon, \nu, V) &=& {1-\varepsilon \over (2\pi)^{p/2} \det(V)^{1/2}} \exp\left(-{x^\top V^{-1} x \over2} \right) \\ \label{eqn:eps_contaminated_normal_distn}
&& \qquad + {\varepsilon \over (2 \pi \nu^2)^{p/2} \det(V)^{1/2}} \exp\left(-{x^\top V^{-1} x \over2 \nu^2} \right);
\end{eqnarray}
\item[(M2)] (polynomial moment) The $p$-variate elliptical $t$-distribution with degree of freedom $\nu$ and density function
\begin{equation}
\label{eqn:elliptic_t_distn}
f(x; \nu, V) = {\Gamma(\nu+p)/2 \over \Gamma(\nu/2) (\nu \pi)^{p/2} \det(V)^{1/2}} \left( 1 + {x^\top V^{-1} x \over \nu} \right)^{-(\nu+p)/2}.
\end{equation}
\end{enumerate}

For the $\varepsilon$-contaminated normal distribution, the kurtosis parameter $\kappa = [1+\varepsilon(\nu^4-1)] / [1+\varepsilon(\nu^2-1)]^2-1$; for the elliptic $t$-distribution, the kurtosis parameter $\kappa=2 / (\nu-4)$; c.f. \cite[Chapter 1]{muirhead1982}.

For the positive-definite matrix $V$, we consider three dependency models:
\begin{enumerate}
\item[(D1)] strong dependence model with $V = 0.9\times\vone_p\vone_p^\top + 0.1\times \Id_p$, where $\vone_p$ is the $p \times 1$ vector of all ones;
\item[(D2)] moderate dependence AR(1) model with $V = \{v_{mk}\}_{m,k=1}^p$ and $v_{mk} = 0.7^{|m-k|}$;
\item[(D3)] weak dependence AR(1) model with $V = \{v_{mk}\}_{m,k=1}^p$ and $v_{mk} = 0.3^{|m-k|}$.
\end{enumerate}
 
We use $\varepsilon=0.2$ and $\nu=1.5$ in (M1) and $\nu=10$ in (M2). For the chosen parameters, the two distributions have the same variance scaling for each $V$, while the kurtosis of the sub-exponential case is $0.16$ and the polynomial case is $1/3$. We compare the finite sample performance on $n=500$ and $p=40$ so that there are 820 covariance parameters. In each setup, we compare the approximation quality of $\bar{T}_n$ using $\bar{Y}$ (i.e. the max-hyperrectangles). All results are reported over 5000 simulation runs. First, the $\varepsilon$-contaminated normal distribution shows a better approximation than elliptical $t$-distribution. This is predicted by our theory in Section \ref{sec:gaussian-approx} and \ref{sec:bootstraps}. Second, the approximation becomes more accurate as the dependence gets stronger.

\begin{figure}[t!] %  figure placement: here, top, bottom, or page
   \centering
      \subfigure{\label{subfig:approx_wild_bootstrap_eps_contaminated_n=200_p=40} \includegraphics[scale=0.22]{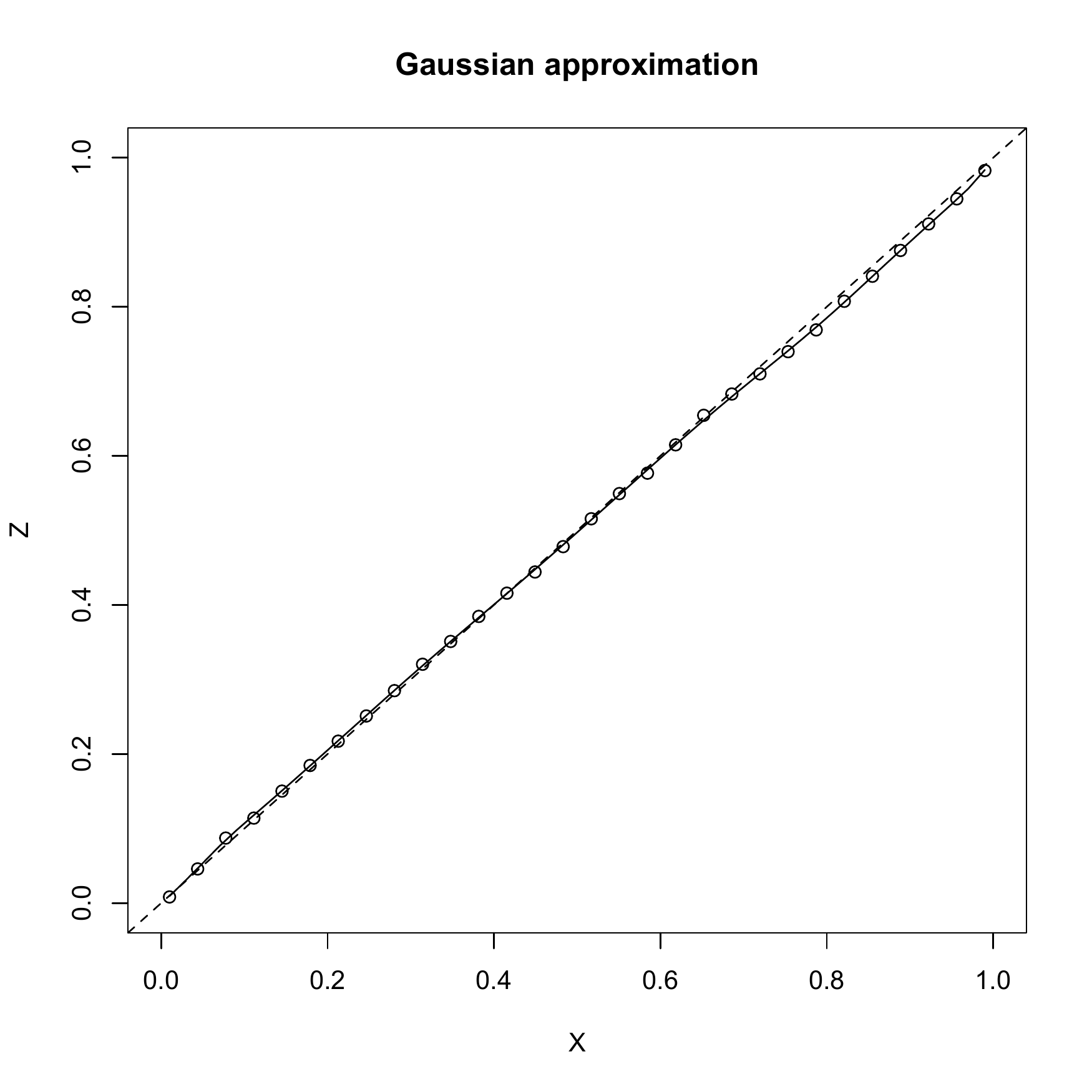}}
      \subfigure{\label{subfig:approx_wild_bootstrap_eps_contaminated_n=200_p=40} \includegraphics[scale=0.22]{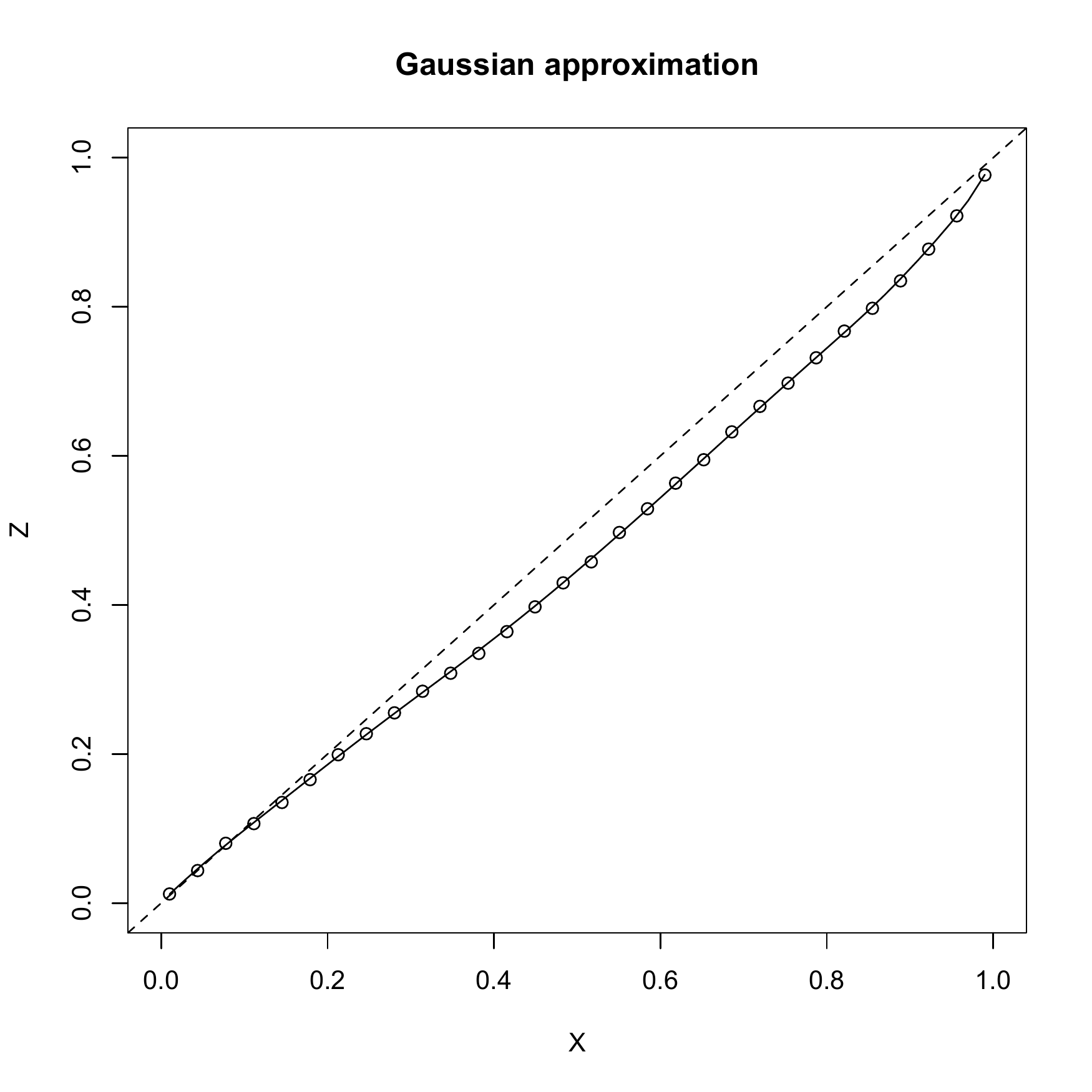}}
      \subfigure{\label{subfig:approx_wild_bootstrap_eps_contaminated_n=200_p=40} \includegraphics[scale=0.22]{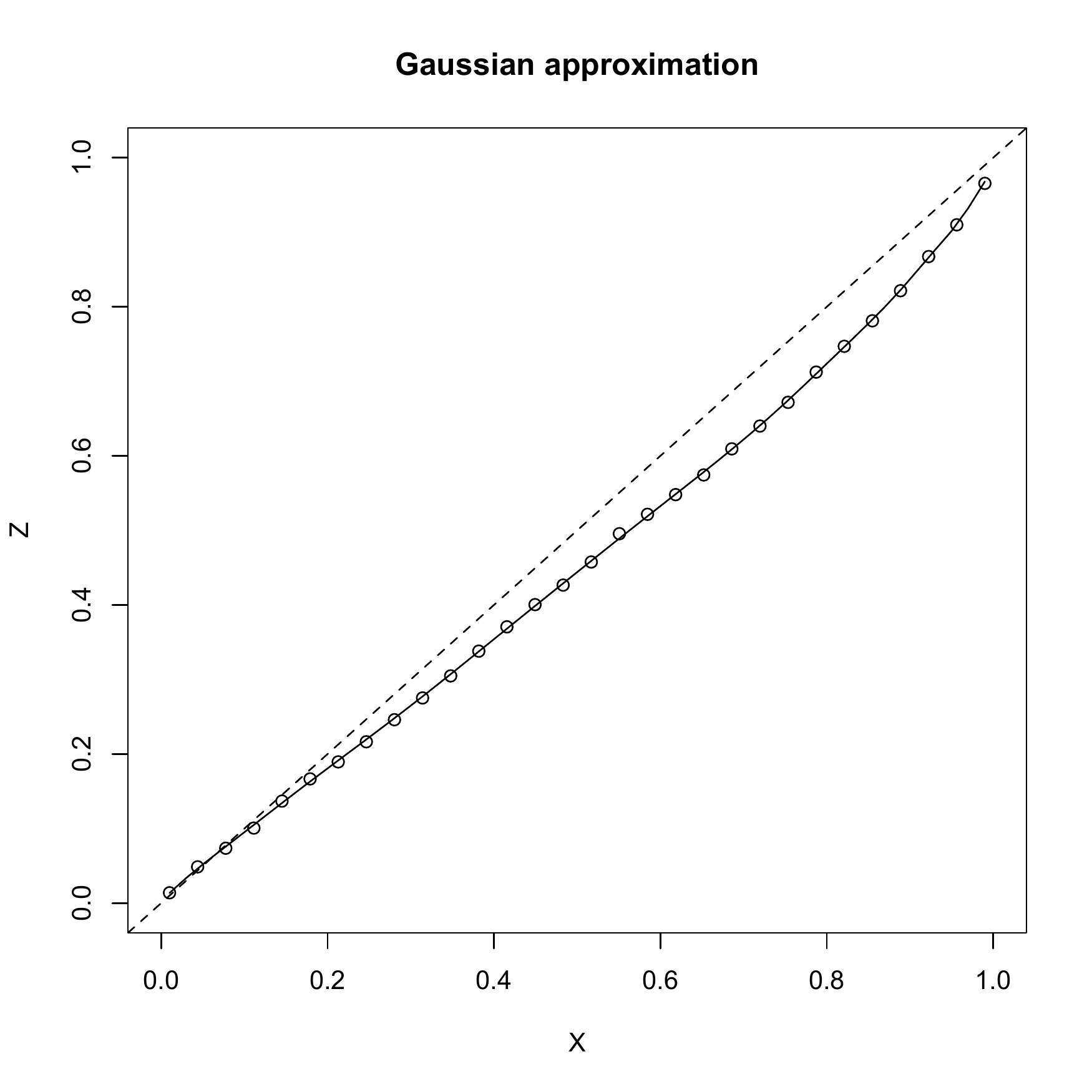}}  \\
      
	\subfigure{\label{subfig:approx_wild_bootstrap_eps_contaminated_n=200_p=40} \includegraphics[scale=0.22]{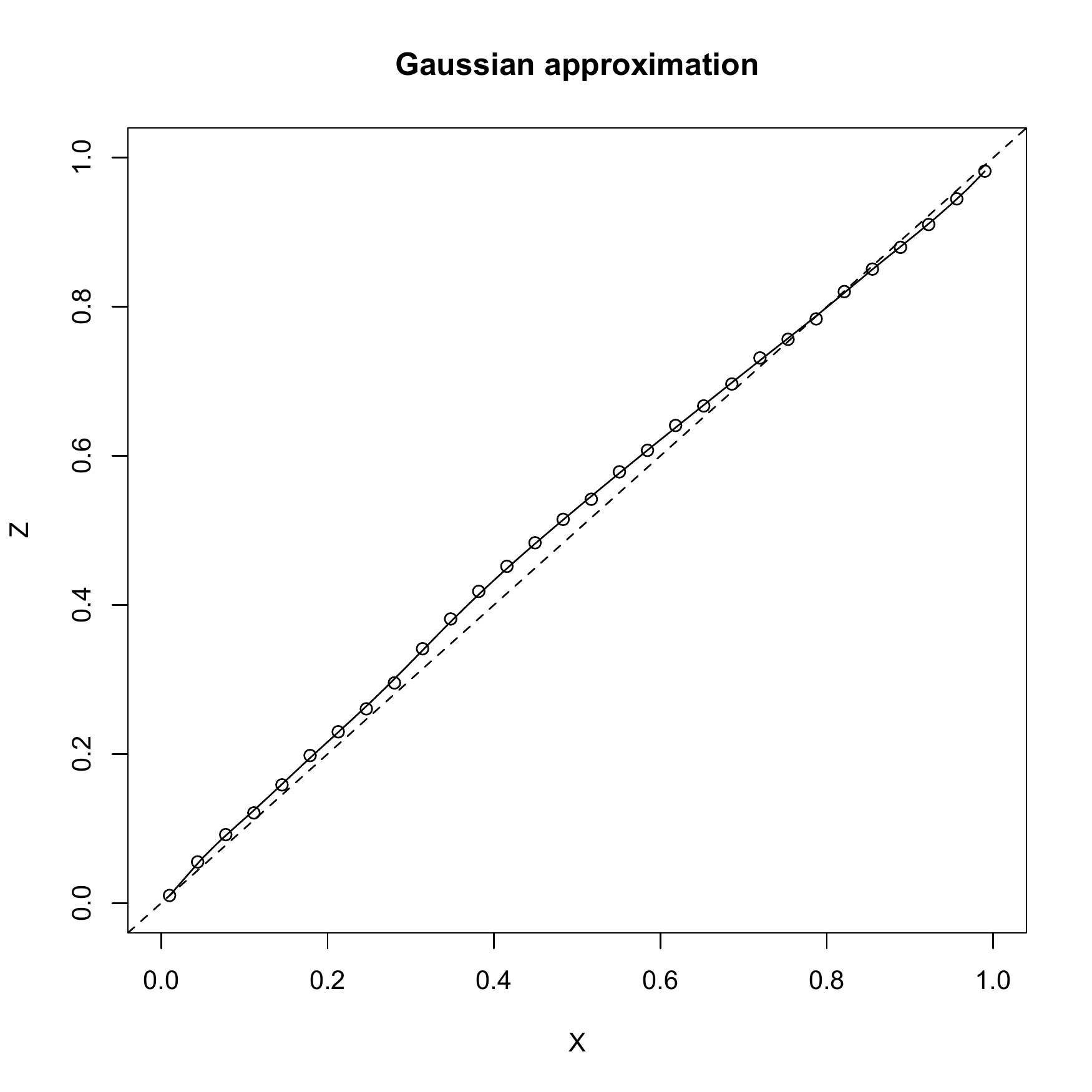}}
	\subfigure{\label{subfig:approx_wild_bootstrap_eps_contaminated_n=200_p=40} \includegraphics[scale=0.22]{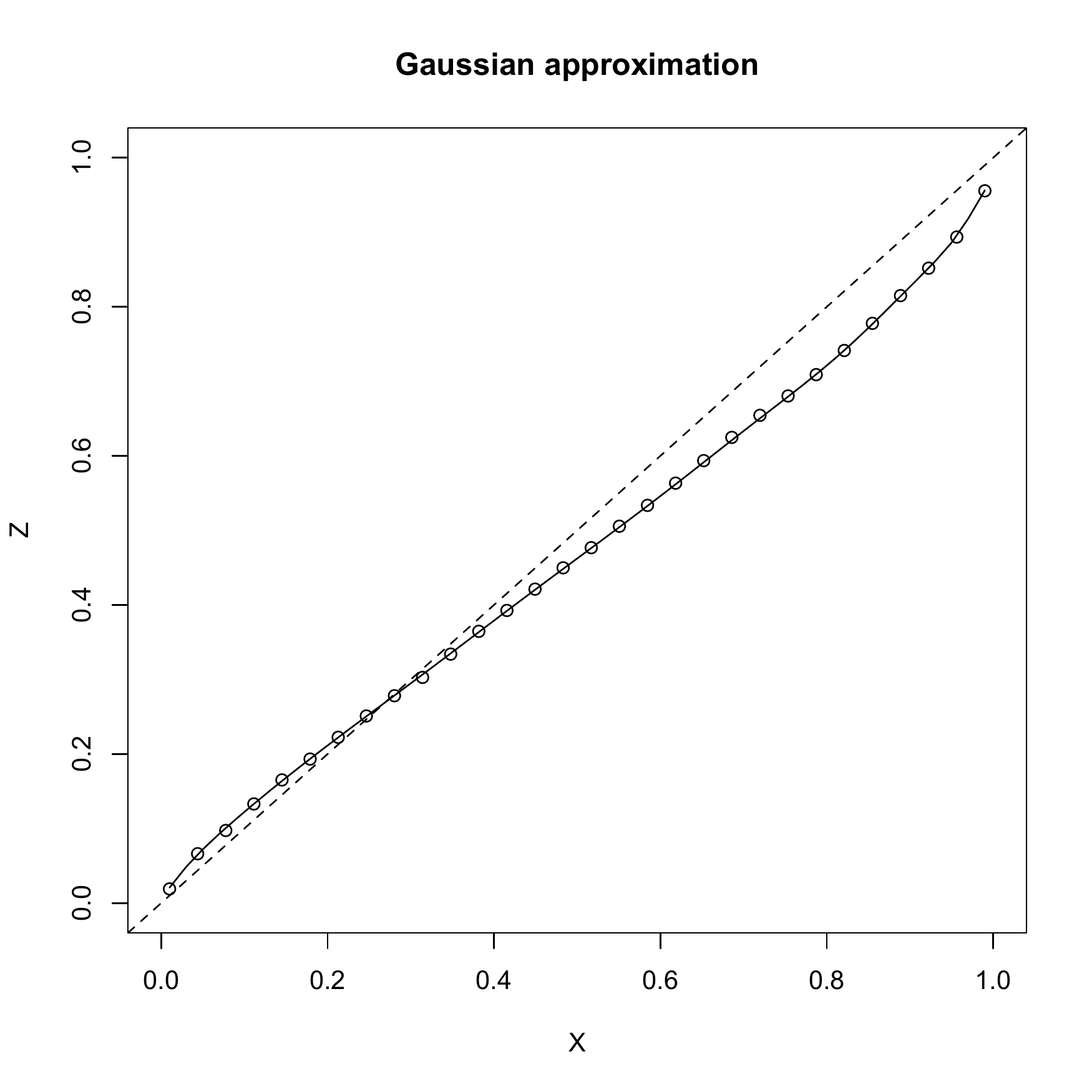}}
	\subfigure{\label{subfig:approx_wild_bootstrap_eps_contaminated_n=200_p=40} \includegraphics[scale=0.22]{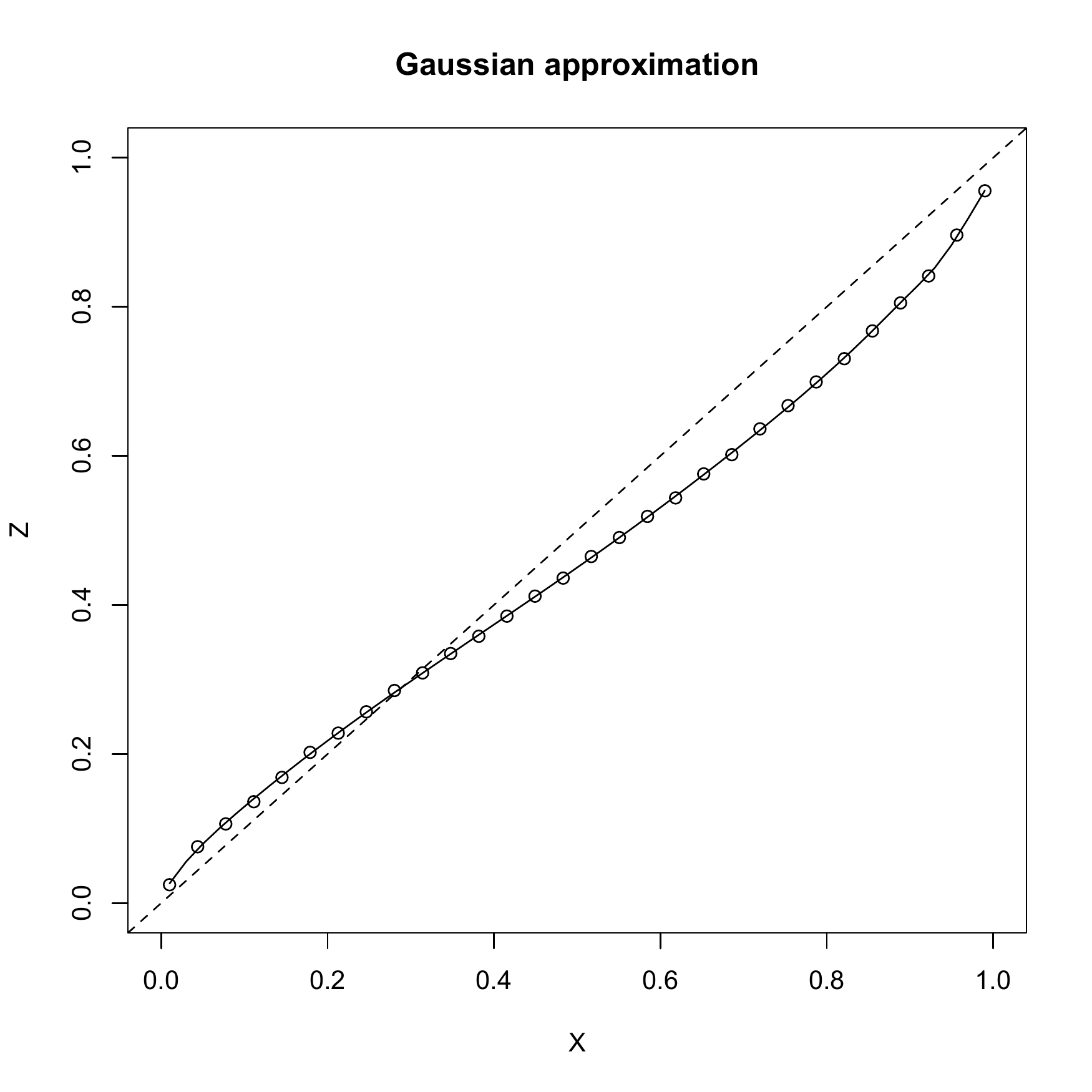}}  \\
   \caption{Gaussian approximations of $\bar{T}_n$ by $\bar{Y}$. Top row for the $\varepsilon$-contaminated normal distribution model: left (M1)+(D1), middle (M1)+(D2), and right (M1)+(D3). bottom row for the elliptic $t$-distribution model: left (M2)+(D1), middle (M2)+(D2), and right (M2)+(D3). Sample size $n=500$ and dimension $p=40$.}
   \label{fig:p=40_n=500}
\end{figure}

%Third, the wild bootstrap has high-quality approximation for the upper tail probabilities, which is particularly relevant for statistical applications; see Section \ref{sec:stat_apps} below. This occurs even when the Gaussian approximation has deteriorated performance such as in the weakly dependent AR(1) models (D2) and (D3). One possible explanation for this phenomenon can be the numeric instability for simulating the $p'\times1$ normal random vectors $Z_i$ in the approximation, where $p'=p(p+1)/2$. On the contrary, the wild bootstrap only requires the simulation of $n$ univariate normal random variables.

\section{Two additional application examples}
\label{sec:more_examples}

In this section, we provide two more examples for applying the bootstrap method. We only state results for subgaussian observations. For the uniform polynomial moment case, we can obtain similar results as in Section \ref{sec:stat_apps}. For a matrix $\Theta = \{\theta_{mk}\}_{m,k=1}^p$, we write $|\Theta|_{L^1} = \max_{1 \le k \le p}\sum_{m=1}^p |\theta_{mk}|$ is the matrix $L^1$-norm of $\Theta$. For a vector $\theta$, we write $|\theta|_w = (\sum_{j=1}^p |\theta_j|^w)^{1/w}, w\ge1$, is the $\ell^w$-norm of $\theta$, where $|\theta|_\infty = \max_{1 \le j \le p} |\theta_j| $ is the max-norm.

\subsection{Estimation of the sparse precision matrix}
\label{subsec:sparse_prec_mat}

Precision matrix, i.e. the inverse of the covariance matrix $\Omega = \Sigma^{-1}$, is an important object in high-dimensional statistics because it closely ties to the Gaussian graphical models and partial correlation graphs \cite{meinshausenbuhlmann2006,yuanlin2007,rothmanbickellevinazhu2008a,pengwangzhouzhu2009a,yuan2010a,cailiuluo2011a}. For multivariate Gaussian observations $X_i$, zero entries in the precision matrix correspond to missing edges in the graphical models; i.e. $\omega_{mk} = 0$ means that $X_m$ and $X_k$ are conditionally independent given the values of all other variables \cite{dempster1972,lauritzen1996a}. To avoid overfitting for graphical models with a large number of nodes, {\it sparsity} is a widely considered structural assumption. Here, we consider the estimation of $\Omega$ by using the CLIME method \cite{cailiuluo2011a}
\begin{equation}
\label{eqn:clime}
\hat\Omega(\lambda) = \text{argmin}_{\Theta \in \mathbb{R}^{p \times p}} |\Theta|_1 \quad \text{subject to} \quad |\hat{S}_n \Theta - \Id_{p \times p}|_\infty \le \lambda,
\end{equation}
where $\lambda \ge 0$ is a tuning parameter to control the sparsity in $\hat\Omega(\lambda)$ and $|\Theta|_1 = \sum_{m,k=1}^p |\theta_{mk}|$. As in the thresholded covariance matrix estimation case in Section \ref{subsec:tuning_selection_thresholded_cov_mat}, the performance of CLIME depends on the selection of tuning parameter $\lambda$. A popular approach is to use the cross-validation (CV), whose theoretical properties again are unclear in the high-dimensional setup. Here, we shall apply the bootstrap method to determine $\lambda$. Let $r \in [0,1)$ and 
$$
\tilde{\cal G}(r, M, \zeta_p) = \Big \{ \Theta \in \mathbb{S}_+^{p \times p} : |\Theta|_{L^1} \le M, \; \max_{m \le p} \sum_{k=1}^p |\theta_{mk}|^r \le \zeta_p \Big \},
$$
where $\mathbb{S}_+^{p \times p}$ is the collection of positive-definite $p \times p$ symmetric matrices. Let $\xi_q = \max_{1 \le k \le p} \|X_{1k}\|_q$ for $q > 0$.

\begin{thm}[Data-driven tuning parameter selection for CLIME: subgaussian observations]
\label{thm:clime_precision_mat_rate_adaptive}
Let $\nu_n \ge 1$ and $X_i$ be iid mean zero random vectors in $\R^p$ such that $X_{ik} \sim \text{subgaussian}(\nu_n^2)$ for all $k=1,\cdots,p$. Suppose that there exist constants $C_i > 0, i=1,\cdots,3,$ such that $\Gamma_{(j,k),(j,k)} \ge C_1$, $\xi_6 \le C_2 \nu_n^{1/3}$ and $\xi_8 \le C_3 \nu_n^{1/2}$ for all $j,k=1,\cdots,p$. Assume that $\Omega \in \tilde{\cal G}(r, M, \zeta_p)$ and and $\nu_n^4 \log^7(np) \le C_4 n^{1-K}$ for some $K \in (0,1)$. Let $\alpha \in (0,1)$ and choose $\lambda_* = M a_{\bar{T}_n^\sharp}(1-\alpha)$, where the bootstrap samples are generated with the covariance matrix kernel. Then, we have with probability at least $1-\alpha-C n^{-K/6}$ for some constant $C > 0$ depending only on $C_1,\cdots,C_4$ such that
\begin{eqnarray}
\label{eqn:spectral_rate_clime_precision_mat_rate_adaptive_wild_bootstrap}
\|\hat\Omega(\lambda_*) - \Omega\|_2 &\le& C_r \zeta_p M^{2-2r} a_{\bar{T}_n^\sharp}^{1-r}(1-\alpha), \\
\label{eqn:Frobenius_rate_clime_precision_mat_rate_adaptive_wild_bootstrap}
p^{-1} |\hat\Omega(\lambda_*) - \Omega|_F^2 &\le& 4 C_r \zeta_p M^{4-2r} a_{\bar{T}_n^\sharp}^{2-r}(1-\alpha),
\end{eqnarray}
where $C_r = 2^{3-2r} (1+2^{1-r}+3^{1-r})$. In addition, we have $\E[a_{\bar{T}_n^\sharp}(1-\alpha)] \le C' \xi_4^2 (\log(p)/n)^{1/2}$ for some constant $C' > 0$ depending only on $\alpha$ and $C_1,\cdots,C_4$. In particular, $\E[\lambda_*] \le C' M \xi_4^2 (\log(p)/n)^{1/2}$.
\end{thm}

Now, we compare Theorem \ref{thm:clime_precision_mat_rate_adaptive} with \cite[Theorem 1(a) and 4(a)]{cailiuluo2011a}. Let $\eta \in (0,1/4)$ and $K,\tau \in (0,\infty)$ be bounded constants. Assuming that $\log(p)/n \le \eta$, $\E[\exp(t X_{ij}^2)] \le K$ for all $|t| \le \eta$ and $i=1,\cdots,n; j=1,\cdots,p$, and $\Omega \in \tilde{\cal G}(r, M, \zeta_p)$, \cite{cailiuluo2011a} showed that with probability at least $1-4p^{-\tau}$
\begin{eqnarray*}
\|\hat\Omega(\lambda_\Delta) - \Omega\|_2 &\le& C'_\Delta \zeta_p M^{2-2r} (\log(p)/n)^{(1-r)/2}, \\
p^{-1} |\hat\Omega(\lambda_\Delta) - \Omega|_F^2 &\le& 4 C'_\Delta \zeta_p M^{4-2r} (\log(p)/n)^{1-r/2},
\end{eqnarray*}
where $\lambda_\Delta = C_\Delta M (\log(p)/n)^{1/2}$, $C_\Delta = 2 \eta^{-2} (2+\tau+\eta^{-1}e^2K^2)^2$, and $C'_\Delta = C_r C_\Delta^{1-r}$. If $X_{ij} \sim \text{subgaussian}(\nu_n^2)$, then $\eta \le \nu_n^{-2}$ for large enough $K$. Therefore, $C_\Delta \gtrsim \nu_n^8$ and $C'_\Delta \gtrsim \nu_n^{8(1-r)}$, both diverging to infinity as $\nu_n^2 \to \infty$ (i.e. $\eta \to 0$). Therefore, $\lambda_* = o_\Prob(\lambda_\Delta)$ and the convergence rates in (\ref{eqn:spectral_rate_clime_precision_mat_rate_adaptive_wild_bootstrap}) and (\ref{eqn:Frobenius_rate_clime_precision_mat_rate_adaptive_wild_bootstrap}) are much sharper than those obtained in \cite[Theorem 1(a) and 4(a)]{cailiuluo2011a}. In addition, for the block diagonal covariance matrix in Example \ref{exmp:thresholded_estimator_reduced_rank}, the bootstrap tuning parameter $\lambda_*$ can gain much tighter performance bounds than $\lambda_\Delta$.

On the other hand, the turning parameter $\lambda_*$ requires the knowledge of $M$ and thus the estimator $\hat\Omega(\lambda_*)$ is not fully data-dependent. In contrast with the thresholded covariance matrix estimation problem in Section \ref{subsec:tuning_selection_thresholded_cov_mat}, the fundamental difficulty here for estimating the precision matrix is that there is no sample analog of $\Omega$ when $p > n$ and $M$ can be viewed as a stability parameter in the sparse inversion of the matrix $\hat{S}_n$. That is, the larger $M$, the more difficult to estimate $\Omega = \Sigma^{-1}$; in particular for CLIME, the rates (\ref{eqn:spectral_rate_clime_precision_mat_rate_adaptive_wild_bootstrap}) and (\ref{eqn:Frobenius_rate_clime_precision_mat_rate_adaptive_wild_bootstrap}) become slower. In addition, $|\Omega|_{L^1}$ plays a similar role in the graphical Lasso model for estimating the sparse precision matrix \cite{ravikumarwainwrightraskuttiyu2008a}. The same comments apply to the problem of estimating the sparse linear functionals in Section \ref{subsec:sparse_linear_functionals}.

\subsection{Estimation of the sparse linear functionals of precision matrix}
\label{subsec:sparse_linear_functionals}

Consider estimation of the linear functional $\theta = \Sigma^{-1} b$, where $b$ is a fixed known $p \times 1$ vector and $\Sigma = \Var(X_i)$. Functionals of such form are related to the solution of the linear equality constrained quadratic program
\begin{equation}
\label{eqn:lcqp}
\text{minimize}_{w \in \mathbb{R}^{p\times p}} w^\top \Sigma w \quad \text{subject to} \quad b^\top w = 1,
\end{equation}
which arises naturally in Markowitz portfolio selection, linear discriminant analysis, array signal processing, best linear unbiased estimator (BLUE), and optimal linear prediction for univariate time series \cite{markowitz1952,maizouyuan2012a,guerci1999a,mcmurrypolitis2015}. For example, in the Markowitz portfolio selection, the portfolio risk $\Var(X_i^\top w)$ is minimized subject to the constraint that the expected mean return $\E(X_i^\top w)$ is fixed at certain level. The solution of (\ref{eqn:lcqp}) $w^*= (b^\top \Sigma^{-1} b)^{-1} \Sigma^{-1} b$ is proportional to $\theta$ and the optimal value of (\ref{eqn:lcqp}) is $(b^\top \Sigma^{-1} b)^{-1}$. A naive approach to estimate $\theta$ has two steps: first construct an invertible estimator $\hat\Sigma$ of $\Sigma$ and second estimate $\theta$ by $\hat\Sigma^{-1} b$. This two-step estimator may not be consistent for $\theta$ in high-dimensions even though $\hat\Sigma$ is a spectral norm consistent (and typically regularized) estimator of $\Sigma$ because in the worst case $|\hat{\theta}-\theta|_2 = |\hat\Sigma^{-1} b - \Sigma^{-1} b|_2 \le \|\hat\Sigma^{-1} - \Sigma^{-1}\|_2 \cdot |b|_2$ does not converge if $|b|_2 \to \infty$ at faster rate than $\|\hat\Sigma^{-1} - \Sigma^{-1}\|_2 \to 0$. However, if $\theta$ has some structural assumptions such as sparsity, then we can directly estimate $\theta$ without the intermediate step for estimating $\hat\Sigma$ or $\hat\Sigma^{-1}$. Sparsity in $\theta$ is often a plausible assumption in real applications. For instance, sparse portfolios have been considered in \cite{brodieetal2009} to obtain the stable portfolio optimization and to facilitate the transaction cost for a large number of assets. When $\theta$ is sparse, the following Dantzig-selector type problem has been proposed in \cite{chenxuwu2016+} to estimate $\theta$
\begin{equation}
\label{eqn:dantzig_linear_functional}
\hat{\theta}(\lambda) = \text{argmin}_{w \in \mathbb{R}^{p\times p}} |w|_1 \quad \text{subject to} \quad |\hat{S}_n w - b|_\infty \le \lambda,
\end{equation}
where $\lambda \ge 0$ is a tuning parameter to control the sparsity in $\hat{\theta}(\lambda)$. The optimization problem (\ref{eqn:dantzig_linear_functional}) can be solved by linear programming and thus there are computationally efficient algorithms for obtaining $\hat{\theta}(\lambda)$. The intuition of (\ref{eqn:dantzig_linear_functional}) is that since $\Sigma \theta = b$, we should expect that $\hat{S}_n \hat{\theta} \approx b$ for a reasonably good estimator $\hat{\theta}$. Under the sparsity assumption on $\theta$ and suitable moment conditions on $X_i$, \cite{chenxuwu2016+} obtained the rate of convergence for $\hat{\theta}(\lambda)$. However, a remaining issue for using (\ref{eqn:dantzig_linear_functional}) on real data is to properly select the tuning parameter $\lambda$. Different from the thresholded covariance matrix estimation where the sparsity is assumed in $\Sigma$, here we do not require this structure in the linear functional estimation. Instead, we impose the sparsity assumption directly on $\theta$. Let $r \in [0,1)$ and
$$
{\cal G}'(r, \zeta_p) = \Big \{w \in \mathbb{R}^p : \sum_{j=1}^p |w_j|^r \le \zeta_p \Big \}.
$$
Here, $\zeta_p$ controls the sparsity level of the elements in ${\cal G}'(r, \zeta_p)$. Without assuming any structure on $\Sigma$, we can allow stronger dependence in $\Sigma$ and therefore the bootstrap approximation may perform better in this case.

\begin{thm}[Data-driven tuning parameter selection in estimation of the sparse linear functional of precision matrix: subgaussian observations]
\label{thm:linear_functional_estimation_adaptive_wild_bootstrap}
Let $\nu_n \ge 1$ and $X_i$ be iid mean zero random vectors in $\R^p$ such that $X_{ik} \sim \text{subgaussian}(\nu_n^2)$ for all $k=1,\cdots,p$. Suppose that there exist constants $C_i > 0, i=1,\cdots,3,$ such that $\Gamma_{(j,k),(j,k)} \ge C_1$, $\xi_6 \le C_2 \nu_n^{1/3}$ and $\xi_8 \le C_3 \nu_n^{1/2}$ for all $j,k=1,\cdots,p$.  Assume that $\theta \in {\cal G}'(r, \zeta_p)$ and $\nu_n^4 \log^7(np) \le C_4 n^{1-K}$ for some $K \in (0,1)$. Let $|\theta|_1 \le M$, $\alpha \in (0,1)$, and choose $\lambda_* = M a_{\bar{T}_n^\sharp}(1-\alpha)$, where the bootstrap samples are generated with the covariance matrix kernel. Then we have for all $w \in [1,\infty]$
\begin{equation}
\label{eqn:rate_linear_functional_estimation_adaptive_wild_bootstrap}
|\hat{\theta}(\lambda_*) - \mbf\theta|_w \le (2 \cdot 6^{1 \over w} \cdot 5^{1-r\over w}) \zeta_p^{1\over w} (M |\Sigma^{-1}|_{L^1})^{1-{r \over w}} a_{\bar{T}_n^\sharp}^{1-{r \over w}}(1-\alpha)
\end{equation}
with probability at least $1-\alpha-C n^{-K/6}$ for some constant $C > 0$ depending only on $C_1,\cdots,C_4$. In addition, we have
\begin{equation}
\label{eqn:linfun_lambda_*-bound_subgaussian}
\E[\lambda_*] \le C' M \xi_4^4 (\log(p)/n)^{1/2},
\end{equation}
where $C' > 0$ is a constant depending only on $\alpha$ and $C_1,\cdots,C_4$.
\end{thm}

The tuning parameter $\lambda_\Delta = C_\Delta M \sqrt{\log(p)/n}$ is selected in \cite[Theorem II.1]{chenxuwu2016+}, which is non-adaptive and the constant $C_\Delta > 0$ depends on the underlying data distribution $F$ through $\nu_n^2$. In particular, $C_\Delta \to \infty$ as $\nu_n^2 \to \infty$. Theorem \ref{thm:linear_functional_estimation_adaptive_wild_bootstrap} shows that the bootstrap tuning parameter selection strategy is less conservative in view of (\ref{eqn:linfun_lambda_*-bound_subgaussian}) and the rate (\ref{eqn:rate_linear_functional_estimation_adaptive_wild_bootstrap}) can be much tighter than $\hat{\theta}(\lambda_\Delta)$ when $\nu_n^2 \to \infty$. Again, for the block diagonal covariance matrix in Example \ref{exmp:thresholded_estimator_reduced_rank}, the bootstrap tuning parameter $\lambda_*$ can gain much tighter performance bounds than $\lambda_\Delta$. However, as in Section \ref{subsec:sparse_prec_mat}, the turning parameter $\lambda_*$ here requires the knowledge of $M$ and thus the estimator $\hat{\theta}(\lambda_*)$ is not fully data-dependent. But this is due to the fundamental difficulty for the lack of the sample analog of $\theta$ in this problem.

\subsection{Proof of Theorem \ref{thm:clime_precision_mat_rate_adaptive} and \ref{thm:linear_functional_estimation_adaptive_wild_bootstrap}}

\begin{lem}
\label{lem:precmat_general_bound}
Suppose that $\Omega \in \tilde{\cal G}(r, M, \zeta_p)$ and let $\lambda \ge |\Omega|_{L^1} |\hat{S}_n-\Sigma|_\infty$. Then, we have
\begin{eqnarray*}
|\hat\Omega(\lambda) - \Omega|_\infty &\le& 4 |\Omega|_{L^1} \lambda, \\
\|\hat\Omega(\lambda) - \Omega\|_2 &\le& C_1 \zeta_p \lambda^{1-r}, \\
p^{-1} |\hat\Omega(\lambda) - \Omega|_F^2 &\le& C_2 \zeta_p \lambda^{2-r},
\end{eqnarray*}
where $C_1 \le 2 (1+2^{1-r}+3^{1-r}) (4 |\Omega|_{L^1})^{1-r}$ and $C_2 \le 4 |\Omega|_{L^1} C_1$.
\end{lem}

\begin{proof}[Proof of Lemma \ref{lem:precmat_general_bound}]
See \cite[Theorem 6]{cailiuluo2011a}.
\end{proof}

\begin{lem}
\label{lem:linfun_general_bound}
Let $\lambda \ge |\theta|_1 |\hat{S}_n - \Sigma|_\infty$. Then, $\theta$ satisfies $|\hat{S}_n \theta - b|_\infty \le \lambda$. For the Dantzig-selector estimator $\hat{\theta}(\lambda)$ in (\ref{eqn:dantzig_linear_functional}), we have
$$
|\hat{\theta}(\lambda) - \theta|_w \le [6 D(5\lambda|\Sigma^{-1}|_{L^1})]^{1 \over w} (2\lambda|\Sigma^{-1}|_{L^1})^{1-{1 \over w}},
$$
where $D(u) = \sum_{j=1}^p (|\theta_j| \wedge u), u \ge 0$, is the smallness measure of $\theta$.
\end{lem}

\begin{proof}[Proof of Lemma \ref{lem:linfun_general_bound}]
See \cite[Lemma V.6]{chenxuwu2016+}.
\end{proof}

\begin{proof}[Proof of Theorem \ref{thm:clime_precision_mat_rate_adaptive}]
Let $\lambda_\diamond = |\Omega|_{L^1} |\hat{S}_n-\Sigma|_\infty$. By the subgaussian assumption and Lemma \ref{lem:moment-bounds-gaussian-obs}, we have (\ref{eqn:check_GA1}) for some large enough constant $C > 0$ depending only on $C_2,C_3$ so that (\ref{eqn:check_GA1}) holds. Since $\Gamma_{(j,k),(j,k)} \ge C_1$ for all $j,k=1,\cdots,p$ and $\nu_n^4 \log^7(np) \le C_5 n^{1-K}$, by Theorem \ref{thm:comparison_with_naive_gaussian_wild_bootstrap}, we have $ |\hat{S}_n-\Sigma|_\infty \le a_{\bar{T}_n^\sharp}(1-\alpha)$ with probability at least $1-\alpha-Cn^{-K/6}$, where $C > 0$ is a constant depending only on $C_i,i=1,\cdots,4$. Since $|\Omega|_{L^1} \le M$ for $\Omega \in \tilde{\cal G}(r, M, \zeta_p)$, $\Prob(\lambda_\diamond \le \lambda_*) \ge 1- \alpha -Cn^{-K/6}$. Then, (\ref{eqn:spectral_rate_clime_precision_mat_rate_adaptive_wild_bootstrap}) and (\ref{eqn:Frobenius_rate_clime_precision_mat_rate_adaptive_wild_bootstrap}) follow from Lemma \ref{lem:precmat_general_bound} applied to the event $\{\lambda_\diamond \le \lambda_*\}$. The bounds for $\E[a_{\bar{T}_n^\sharp}(1-\alpha)]$ and $\E[\lambda_*]$ are the same as those in Theorem \ref{thm:thresholded_cov_mat_rate_adaptive}.
\end{proof}

\begin{proof}[Proof of Theorem \ref{thm:linear_functional_estimation_adaptive_wild_bootstrap}]
For $\theta \in {\cal G}'(r, \zeta_p)$, we have $D(u) \le 2 u^{1-r} \zeta_p$. Let $\lambda_\diamond = |\theta|_1 |\hat{S}_n-\Sigma|_\infty$. Following the proof of Theorem \ref{thm:clime_precision_mat_rate_adaptive}, we have $\Prob(\lambda_\diamond \le \lambda_*) \ge 1- \alpha -Cn^{-K/6}$.  By Lemma \ref{lem:linfun_general_bound}, we have with probability at least $1- \alpha -Cn^{-K/6}$
\begin{eqnarray*}
|\hat{\theta}(\lambda_*) - \theta|_w &\le& [6 D(5 \lambda_* |\Sigma^{-1}|_{L^1}) ]^{1 \over w} (2 \lambda_* |\Sigma^{-1}|_{L^1})^{1-{1 \over w}} \\
&\le& (2 \cdot 6^{1 \over w} \cdot 5^{1-r\over w}) \zeta_p^{1\over w} |\Sigma^{-1}|_{L^1}^{1-{r \over w}} \lambda_*^{1-{r \over w}},
\end{eqnarray*}
which is (\ref{eqn:rate_linear_functional_estimation_adaptive_wild_bootstrap}). The bound for $\E[\lambda_*]$ is immediate.
\end{proof}

\end{document}